\documentclass[12pt]{amsart}

\usepackage{amscd,latexsym,amsthm,amsfonts,amssymb,amsmath,amsxtra}
\usepackage[mathscr]{eucal}
\usepackage{pictexwd,dcpic}
\usepackage[normalem]{ulem}
\usepackage{hyperref}
\renewcommand\theequation{\thesection.\arabic{equation}}

\newcommand{\BA}{{\mathbb {A}}}

\newcommand{\BC}{{\mathbb {C}}}

\newcommand{\BF}{{\mathbb {F}}}

\newcommand{\BN}{{\mathbb {N}}}

\newcommand{\BR}{{\mathbb {R}}}

\newcommand{\BZ}{{\mathbb {Z}}}

\newcommand{\CF}{{\mathcal {F}}}

\newcommand{\CL}{{\mathcal {L}}}

\newcommand{\CO}{{\mathcal {O}}}
\newcommand{\CP}{{\mathcal {P}}}
\newcommand{\CQ}{{\mathcal {Q}}}

\newcommand{\CT}{{\mathcal {T}}}

\newcommand{\CX}{{\mathcal {X}}}
\newcommand{\CY}{{\mathcal {Y}}}
\newcommand{\CZ}{{\mathcal {Z}}}

\newcommand{\Fa}{{\mathfrak {a}}}

\newcommand{\Fg}{{\mathfrak {g}}}
\newcommand{\Fh}{{\mathfrak {h}}}

\newcommand{\Fl}{{\mathfrak {l}}}
\newcommand{\Fm}{{\mathfrak {m}}}

\newcommand{\Fo}{{\mathfrak {o}}}
\newcommand{\Fp}{{\mathfrak {p}}}

\newcommand{\Ft}{{\mathfrak {t}}}
\newcommand{\Fu}{{\mathfrak {u}}}

\newcommand{\Fz}{{\mathfrak {z}}}

\newcommand{\GL}{{\mathrm{GL}}}

\newcommand{\Hom}{{\mathrm{Hom}}}

\newcommand{\Ind}{{\mathrm{Ind}}}
\newcommand{\ind}{{\mathrm{ind}}}

\newcommand{\PGL}{{\mathrm{PGL}}}

\newcommand{\tr}{{\mathrm{tr}}}

\newcommand{\ol}{\overline}

\newcommand{\cg}{C_{c}^{\infty} (Z_G(F)\backslash G(F))}
\newcommand{\zg}{Z_G(F)\backslash G(F)}
\newcommand{\zh}{Z_H(F)\backslash H(F)}
\newcommand{\gf}{{}^g f^{\xi}}

\def\diag{{\rm diag}}

\newtheorem{thm}{Theorem}[section]
\newtheorem{cor}[thm]{Corollary}
\newtheorem{lem}[thm]{Lemma}
\newtheorem{prop}[thm]{Proposition}
\newtheorem {conj}[thm]{Conjecture}

\newtheorem {ques/conj}[thm]{Question/Conjecture}

\newtheorem{defn}[thm]{Definition}
\newtheorem{rmk}[thm]{Remark}

\makeatletter

\newcommand{\Rmnum}[1]{\expandafter\@slowromancap\romannumeral #1@}
\makeatother

\begin{document}
\renewcommand{\theequation}{\arabic{equation}}
\numberwithin{equation}{section}

\title[On the Ginzburg-Rallis models]{A local relative trace formula for the Ginzburg-Rallis model: the geometric side.}

\author{Chen Wan}
\address{School of Mathematics\\
University of Minnesota\\
Minneapolis, MN 55455, USA}
\email{wanxx123@umn.edu}

\date{\today}

\begin{abstract}
Following the method developed by Waldspurger and Beuzart-Plessis in their proofs of the local Gan-Gross-Prasad conjecture, we are able to prove the geometric side of a local relative trace formula for the Ginzburg-Rallis model. Then by applying such formula, we prove a multiplicity formula of the Ginzburg-Rallis model for the supercuspidal representations. Using that multiplicity formula, we prove the multiplicity one theorem for
the Ginzburg-Rallis model over Vogan packets in the supercuspidal case.
\end{abstract}

\maketitle

\tableofcontents

\section{Introduction and Main Result}

\subsection{The Ginzburg-Rallis model}
D. Ginzburg and S. Rallis found in their paper (\cite{GR00}) a global integral representation for the partial exterior cube L-function $L^S(s,\pi,\wedge^3)$ attached to any irreducible cuspidal automorphic representation $\pi$ of $\GL_6(\BA)$. By using the regularized Siegel-Weil formula of Kudla and Rallis(\cite{KR94}), they discovered that the nonvanishing of the central value of the partial exterior cube L-function $L^S(\frac{1}{2},\pi,\wedge^3)$ is closely related to the Ginzburg-Rallis period, which will be defined as follows.
The relation they discovered is similar to the global Gan-Gross-Prasad conjecture (\cite{GP92}, \cite{GP94}, \cite{GGP12}), but for a different setting.

Let $k$ be a number field, $\BA$ be the ring of adeles of $k$. Take $P=P_{2,2,2}=MU$ be the standard parabolic subgroup of $G=\GL_6$ whose Levi part $M$ is isomorphic to $\GL_2\times \GL_2\times \GL_2$, and whose unipotent radical $U$ consists of elements of the form
\begin{equation}\label{unipotent}
u=u(X,Y,Z):=\begin{pmatrix} I_2 & X & Z \\ 0 & I_2 & Y \\ 0 & 0 & I_2 \end{pmatrix}.
\end{equation}
We define a character $\xi$ on $U$ by
\begin{equation}\label{character}
\xi(u(X,Y,Z)):=\psi(a\tr(X)+b\tr(Y))
\end{equation}
where $\psi$ is a non-trivial additive character on $k\backslash\BA$, and $a,b\in \BA^{\times}$.

It's clear that the stabilizer of $\xi$ is the diagonal embedding of $\GL_2$ into $M$, which is denoted by $H$. For a given idele character $\chi$ of $\BA^{\times}/k^{\times}$, one induces a one dimensional representation $\sigma$ of $H(\BA)$ given by $\sigma(h):=\chi(\det(h))$, which is clearly trivial when restricted to $H(k)$. Now the character $\xi$ can be extended to the semi-direct product
\begin{equation}
R:=H\ltimes U
\end{equation}
by making it trivial on $H$. Similarly we can extend the character $\sigma$ to $R$ It follows that the one dimensional representation $\sigma\otimes \xi$ of $R(\BA)$ is well defined and it is trivial when restricted to the $k$-rational points $R(k)$. Then the Ginzburg-Rallis period for any cuspidal automorphic form $\phi$ on $\GL_6(\BA)$ with central character $\chi^2$ is defined to be
\begin{equation}\label{GR period 1}
\CP_{R,\sigma\otimes \xi}=\int_{H(k)Z_G(\BA)\backslash H(\BA)}\int_{U(k)\backslash U(\BA)}  \phi(hu)\xi^{-1}(u)\sigma^{-1}(h) du dh.
\end{equation}

As in the Jacquet conjecture for the trilinear period of $\GL_2$ (\cite{HK04}) and in the global Gan-Gross-Prasad conjecture
(\cite{GGP12}) more generally, Ginzburg and Rallis find that the central value of
the partial exterior cube L-function, $L^S(\frac{1}{2},\pi,\wedge^3)$ may also be related to the quaternion algebra version of the
Ginzburg-Rallis period $\CP_{R,\sigma\otimes \xi}$. More precisely, let $D$ be a quaternion algebra over $k$, and consider $G_D:=\GL_3(D)$,
a $k$-inner form of $\GL_6$. In the group $G_D$, they define
\begin{equation}
H_D=\{h_D=\begin{pmatrix} g & 0 & 0 \\ 0 & g & 0 \\ 0 & 0 & g \end{pmatrix}\mid g\in D^{\times}\}
\end{equation}
and
\begin{equation}
U_D=\{u_D(x,y,z)=\begin{pmatrix} 1 & x & z \\ 0 & 1 & y \\ 0 & 0 & 1 \end{pmatrix} \mid x,y,z \in D\}.
\end{equation}

In this case, the corresponding character $\xi_D$ of $U_D$ is defined in same way except that the trace in the definition of $\xi$ is replaced
by the reduced trace of the quaternion algebra $D$. Similarly, the character $\sigma_D$ on $H_D$ is defined by using the reduced norm of the quaternion algebra $D$. Now the subgroup $R_D$ is defined to be the semi-direct product $H_D\ltimes U_D$ and the corresponding one dimensional representation $\sigma_D\otimes \xi_D$ of $R_D(\BA)$ is well defined. The $D$-version of the Ginzburg-Rallis period for any cuspidal automorphic form
$\phi^D$ on $\GL_3(D)(\BA)$ with central character $\chi^2$ is defined to be
\begin{equation}\label{GR period 2}
\CP_{R_D,\sigma_D\otimes \xi_D}:=\int_{H_D(k)Z_{G_D}(\BA)\backslash H_D(\BA)}\int_{U_D(k)\backslash U_D(\BA)}
\phi^D(hu)\xi_{D}^{-1}(u)\sigma_{D}^{-1}(h) du dh.
\end{equation}

In \cite{GR00}, they form a conjecture on the relation between the periods above and the central value $L^S(\frac{1}{2},\pi,\wedge^3)$.

\begin{conj}[Ginzburg-Rallis, \cite{GR00}]\label{global}
Let $\pi$ be an irreducible cuspidal automorphic representation of $\GL_6(\BA)$ with central character $\omega_{\pi}$. Assume that there exists an idele character $\chi$ of $\BA^{\times}/k^{\times}$ such that $\omega_{\pi}= \chi^2$. Then the central value $L^S(\frac{1}{2},\pi,\Lambda^3)$ does not vanish if and only if there exists a unique quaternion algebra $D$ over $k$ and there exists the Jacquet-Langlands correspondence $\pi_D$ of $\pi$ from $\GL_6(\BA)$ to $\GL_3(D)(\BA)$, such that the period $\CP_{R_D,\sigma_D\otimes \xi_D}(\phi^D)$ does not vanish for some $\phi^D\in \pi_D$, and the period $\CP_{R_{D'},\sigma_{D'}\otimes \xi_{D'}}(\phi^{D'} )$ vanishes identically for all
quaternion algebra $D'$ which is not isomorphic to $D$ over k, and for all $\phi^{D'}\in \pi_{D'}$.
\end{conj}

It is clear that this conjecture is an analogy of the global Gan-Gross-Prasad conjecture for classical groups (\cite{GGP12}) and the Jacquet conjecture for the triple product L-functions for $GL_2$, which is proved by M. Harris and S. Kudla in \cite{HK04}. It is also clear that Conjecture \ref{global} is now
a special case of the general global conjecture of Y. Sakellaridis and A. Venkatesh for periods associated to general spherical varieties (\cite{SV}).

Similarly to the Gan-Gross-Prasad model, there is also a local conjecture for the Ginzburg-Rallis model, which is the main result of this paper. The conjecture at local places has been expected since the work of \cite{GR00}, and was first discussed in details by Dihua Jiang in his paper \cite{J08}. Now let $F$ be a local field (p-adic field or real field), $D$ be the unique quaternion algebra over $F$. Then we may also
define the groups $H, U, R, H_D, U_D$, and $R_D$ as above. The local conjecture can be stated as follows, using
the local Jacquet-Langlands correspondence established in \cite{DKV84}.

\begin{conj}[Jiang, \cite{J08}]\label{jiang}
For any irreducible admissible representation $\pi$ of $\GL_6(F)$, let $\pi_D$ be the local Jacquet-Langlands correspondence of $\pi$ to $\GL_3(D)$ if it exists, and zero otherwise. Assume that there exists a character $\chi$ of $F^{\times}$ such that $\omega_{\pi}=\chi^2$.
For a given non-trivial additive character $\psi$ of $F$, define the one dimensional representation $\sigma\otimes \xi$ of $R(F)$ and $\sigma_D\otimes \xi_D$ of $R_D(F)$, respectively. Then the following identity
\begin{equation}\label{equation 1}
\dim (\Hom_{R(F)} (\pi, \sigma\otimes \xi))+ \dim (\Hom_{R_D(F)} (\pi_D, \sigma_D \otimes \xi_D))=1
\end{equation}
holds for all irreducible generic representation $\pi$ of $\GL_6(F)$.
\end{conj}

As in the local Gan-Gross-Prasad conjecture (\cite{GGP12}), Conjecture \ref{jiang} can be reformulated in terms of local Vogan packets and
the assertion in the conjecture is expressed as the local multiplicity one over local Vogan packets. Here although $\GL_6(F)$ does not have non-trivial pure inner form, as we already make the central character assumption, we are actually working with $\PGL_6$ which have non-trivial pure inner form. For any quaternion algebra $D$ over $F$ which may be $F$-split, define
\begin{equation}
m(\pi_D):=m(\pi_D,\sigma_D\otimes\xi_D):=\dim (\Hom_{R_D(F)} (\pi_D, \sigma_D \otimes \xi_D)).
\end{equation}
The local multiplicity one theorem for each individual irreducible admissible representation $\pi_D$ of $\GL_3(D)$ asserts that
\begin{equation}
m(\pi_D)=m(\pi_D,\sigma_D\otimes\xi_D)\leq 1
\end{equation}
for any given $\sigma_D\otimes\xi_D$. This local multiplicity one theorem was proved in \cite{N06} over a $p$-adic local field and in \cite{JSZ11} over an archimedean local field.

\subsection{Main results}
There are two main results of this paper. One is the geometric side of a local relative trace formula for the Ginzburg-Rallis model. For simplicity, in the rest of the paper, when we say the trace formula, unless otherwise specified, we mean the local relative trace formula for the Ginzburg-Rallis model. To be specific, let $F$ be a p-adic field, $f\in C_{c}^{\infty}(\zg)$ be a strongly cuspidal function (see Section 4.1 for the definition of strongly cuspidal functions), define the function ${}^g f^{\xi}$ on $H(F)/Z_H(F)$ by
$$
{}^g f^{\xi}(x)=\int_{U(F)} f(g^{-1}xug)\xi(u) du.
$$
This is a function belonging to $C_{c}^{\infty}(Z_H(F)\backslash H(F))$. Define
\begin{equation}\label{7.1}
I(f,g)=\int_{Z_H(F)\backslash H(F)} {}^g f^{\xi}(x) dx,
\end{equation}
and for each $N\in \BN$, define
\begin{equation}\label{spectral 0}
I_N(f)=\int_{U(F)H(F)\backslash G(F)} I(f,g) \kappa_N(g) dg.
\end{equation}
Here $\kappa_N$ is a truncated function on $G(F)$, which is left $U(F)H(F)$-invariant, right $K$-invariant, and compactly supported modulo $U(F)H(F)$. For the complete definition of $\kappa_N$, see Section 5.2. The distribution in the trace formula is just $\lim_{N\rightarrow \infty} I_N(f)$.

Now we define the geometric side of the trace formula. For each strongly cuspidal function $f\in C_{c}^{\infty}(\zg)$, one can associate a distribution $\theta_f$ on $G(F)$ via the weighted orbital integral (see Section 4). Waldspurger has proved in his paper \cite{W10} that the distribution $\theta_f$ is a quasi-character in the sense that for every semisimple element $x\in G_{ss}(F)$, $\theta_f$ is a linear combination of the Fourier transforms of nilpotent orbital integrals of $\Fg_x$ near $x$. For each nilpotent orbit $\CO$ of $\Fg_x$, let $c_{\theta_f,\CO}(x)$ be the coefficient, it is called the germ of the distribution $\theta_f$. Let $\CT$ be a subset of subtorus of $H$ which will be defined in Section 5.1, for any $t\in T_{reg}(F)$ and $T\in \CT$, define $c_f(t)$ to be $c_{\theta_f,\CO_t}(t)$ where $\CO_t$ is the unique regular nilpotent orbit in $\Fg_t$. For detailed description of $\CO_t$, see Section 5.1. Then we define the geometric side of our trace formula to be
$$I(f)=\sum_{T\in \CT} | W(H,T)|^{-1} \nu(T) \int_{Z_G(F)\backslash T(F)} c_f(t) D^H(t) \Delta(t) dt$$
where $D^H(t)$ is the Weyl determinant and $\Delta(t)$ is some normalized function defined in Definition \ref{delta 1}. It is proved in Proposition \ref{integrable 1} that the integral defining $I(f)$ is absolutely convergent. Then the trace formula that we are going to prove in this paper is (see Theorem \ref{main 1})
\begin{equation}\label{1.1}
\lim_{N\rightarrow \infty} I_N(f)=I(f).
\end{equation}
Similar we can have the quaternion version of the trace formula. It is also possible to prove the trace formula for functions with nontrivial central character (see Theorem \ref{main 3}). We refer the readers to Section 5 for detailed discussion of the trace formula. The proof of the trace formula will be carried out from Section 7 to Section 10.

Another main result of this paper is to prove Conjecture \ref{jiang} for the case that $F$ is a $p$-adic local field and $\pi$ is an irreducible supercuspidal representation of $\GL_6(F)$. In this situation, $\pi_D$ always exists and it is also supercuspidal (Theorem B.2.b. of \cite{DKV84}). This result can be viewed as an application of the local relative trace formula \eqref{1.1}, it will be proved in Section 6.

\begin{thm}\label{main}
For any irreducible supercuspidal representation $\pi$ of $\GL_6(F)$ over a $p$-adic local field $F$ of characteristics zero, Conjecture \ref{jiang} holds.
\end{thm}

Both of our main results are a crucial step towards the proof of Conjecture \ref{jiang} in general case, which will be considered in the forthcoming paper \cite{Wan16}. To be more specific, in this paper we prove the geometric side of a local relative trace formula for the Ginzburg-Rallis model (Theorem \ref{main 3}), which implies Conjecture \ref{jiang} when the representation is supercuspidal. In the forthcoming paper, we prove the spectral side of the trace formula, then the general case of Conjecture \ref{jiang} will follow from the full trace formula. The forthcoming paper will also treat the archimedean case.

One remark is that it is easy to see that the multiplicity is independent of the choice of $a,b\in F^{\times}$ in the definition of $\xi$, but we will \textbf{NOT} fix $a,b$ at the beginning. The reason is that we actually need to change $a,b$ in our proof of the trace formula. For details, see Section 10.

\subsection{Organization of the paper and remarks on the proofs}
Our proof of Theorem \ref{main} uses Waldspurger's method in his proof of the local Gan-Gross-Prasad conjecture in \cite{W10}. In other words, we are going to prove a multiplicity formula:
\begin{equation}\label{equation 2}
m(\pi)=m_{geom}(\pi),\;  m(\pi_D)=m_{geom}(\pi_D).
\end{equation}
Here $m_{geom}(\pi)$ (resp. $m_{geom}(\pi_D)$) is defined in the same way as $I(f)$ except replacing the distribution $\theta_f$ by the distribution character $\theta_{\pi}$ (resp. $\theta_{\pi_D}$) associated to the representation $\pi$ (resp. $\pi_D$). For the complete definition of the multiplicity formula, see Section 6. Once this formula is proved, we can use the relation between the distribution characters $\theta_{\pi}$ and $\theta_{\pi_D}$ under the local Jacquet-Langlands correspondence to cancel out all terms in the expression of $m_{geom}(\pi)+m_{geom}(\pi_D)$ except the term $c_{\theta_{\pi},\CO_{reg}}$, which is the germ at the identity element. Then the work of Rodier (\cite{Rod81}) shows that $c_{\theta_{\pi},\CO_{reg}}=0$ if $\pi$ is non-generic, and $c_{\theta_{\pi},\CO_{reg}}=1$ if $\pi$ is generic. Because all supercuspidal representations of $GL_n(F)$ are generic,
we get the following identity
\begin{equation}\label{gm=1}
m_{geom}(\pi)+m_{geom}(\pi_D)=1.
\end{equation}
And this proves Theorem \ref{main}. It is worth to mention that this case is quiet different from the case of the local Gan-Gross-Prasad conjecture. Namely, in their case, the additive character is essentially attached to the simple roots, which is not the case in our situation. This difference leads to the technical complication on the proof of some unipotent invariance. This will be discussed in detail in Section 9. Another difference is that in this case we do need to worried about the center of the group, this will be discussed in Section 5.

In order to prove the multiplicity formula \eqref{equation 2}, we follow the arguments developed by Waldspurger in \cite{W10}. The key ingredient is to prove the trace formula (Theorem \ref{main 3}), which is carried out from Section 7 to Section 10. Once this is done, we can relate the multiplicity to a certain integral of the matrix coefficients of supercuspidal representations (which are strongly cuspidal functions). This integral happens to be the distribution $\lim_{N\rightarrow \infty}I_{N,\chi}(f)$ in the trace formula (i.e. the left hand side of \eqref{equation 3}) up to a constant. On the other hand, it is proved by Arthur in his paper \cite{Ar87} that the distribution character of the representation is the same constant times the distribution $\theta_f$ induced by its matrix coefficient $f$. It follows that the right hand side of \eqref{equation 2} is the same constant times the geometric side of the trace formula (i.e. the right hand side of \eqref{equation 3}). So after applying the trace formula as in Theorem \ref{main 3}, we prove the multiplicity formula \eqref{equation 2}, which implies Theorem \ref{main}. For details, see Section 6.

This paper is organized as follows.
In Section 2, we introduce basic notations and conventions of this paper, we will also recall the definition and some basic facts on weighted orbital integrals. In Section 3, we define quasi-characters and discuss the localization of the quasi-characters. In Section 4, we will talk about strongly cuspidal functions and the distribution associated to it. This is a key ingredient in the trace formula. For Section 3 and 4, we follow \cite{W10} closely and provide the details for the current case as needed.

In Section 5, we describe all the ingredients in the geometric side and state the trace formula in Theorem \ref{main 1} and Theorem \ref{main 3}. We also show that it is enough to prove it for functions with trivial central character (Proposition \ref{center issue}). Then in Section 6, we prove our main result Theorem \ref{main} by assuming the trace formula Theorem \ref{main 1} holds.

Starting from Section 7, we are going to prove the trace formula. In Section 7, we deal with the localization of the trace formula. The goal of this section is to reduce our problem to the Lie algebra level. In Section 8, we study the slice representation of the normal space. As a result, we transfer our integral to the form $\int_{A_T(F)\backslash G(F)}$ where $T$ is some maximal torus of $G$. The reason we do this is that we want to apply the local trace formula developed by Arthur in \cite{Ar91} as Waldspurger did in \cite{W10}. In Section 9, we prove that we are actually able to change our truncation function to the one given by Arthur in his local formula. After this is done, we can apply Arthur's local trace formula to calculate the distribution in our trace formula. More precisely, at beginning, the distribution is a limit of the truncated integral. After applying Arthur's local trace formula, we can calculate that limit explicitly. This is the most technical section of this paper, we will postpone the proof of two technical lemmas in this section to Appendix A. In Section 10, we finish the proof of the trace formula.

There are two appendices in this paper. In Appendix A, we prove two technical lemmas in Section 9 which is Lemma \ref{major 2} and Lemma \ref{split zero}. In Appendix B, we will state some similar results for some reduced models that occur naturally on the way of investigation of the Ginzburg-Rallis model. This will be needed for our study of Conjecture \ref{jiang} for general case (\cite{Wan16}). Since the proof are similar to the Ginzburg-Rallis model case we considered in this paper, we will skip it here and we refer the readers to my thesis \cite{Wan17} for details of the proof.

\subsection{Acknowledgement} I would like to thank my advisor Dihua Jiang for suggesting me thinking about this problem, providing practical and thought-provoking viewpoints that lead to solutions of the problem, and carefully reviewing the first draft of this paper. I would like to thank Professor Erik P. van den Ban and an anonymous referee for many helpful comments and corrections.

\section{Preliminarities}

\subsection{Notation and conventions}
Let $F$ be a p-adic filed, we fix the algebraic closure $\ol{F}$. Let val$_F$ and $\mid \cdot \mid_F$ be the valuation and absolute value on $F$, $\Fo_F$ be the ring of integers of $F$, and $\BF_q$ be the residue field. We fix an uniformizer $\varpi_F$.

For every connected reductive algebraic group $G$ defined over $F$, let $A_G$ be the maximal split central torus of $G$ and $Z_G$ be the center of $G$.
We denote by $X(G)$ the group of $F$-rational characters of $G$. Define $\Fa_G=$Hom$(X(G),\BR)$, and let $\Fa_{G}^{\ast}=X(G)\otimes_{\BZ} \BR$ be the dual of $\Fa_G$. We define a homomorphism $H_G:G(F)\rightarrow \Fa_G$ by $H_G(g)(\chi)=\log(|\chi(g)|_F)$ for every $g\in G(F)$ and $\chi\in X(G)$. Denote by $\Fg$ the Lie algebra of $G$. It is clear that $G$ acts on $\Fg$ by the adjoint action. Since the Ginzburg-Rallis model has non-trivial center, all our integration need to modulo the center. To simplify the notation, for any Lie algebra $\Fg$ contained in $\Fg\Fl_n$ (in our case it will always be contained in $\Fg\Fl_6(F)$ or $\Fg\Fl_3(D)$), denote by $\Fg_0$ the elements in $\Fg$ whose trace (as an element in $\Fg\Fl_n$) is zero.

For a Levi subgroup $M$ of $G$, let $\CP(M)$ be the set of parabolic subgroups of $G$ whose Levi part is $M$, $\CL(M)$ be the set of Levi subgroups of $G$ containing $M$, and $\CF(M)$ be the set of parabolic subgroups of $G$ containing $M$. We have a natural decomposition $\Fa_M=\Fa_{M}^{G}\oplus \Fa_G$, denote by $proj_{M}^{G}$ and $proj_G$ the projections of $\Fa_M$ to each factors. The subspace $\Fa_{M}^{G}$ has a set of coroots $\check{\Sigma}_M$, and for each $P\in \CP(M)$, we can associate a positive chamber $\Fa_{P}^{+}\subset \Fa_M$ and a subset of simple coroots $\check{\Delta}_P\subset \check{\Sigma}_M$. For each $P=MU$, we can also define a function $H_P:G(F)\rightarrow \Fa_M$ by $H_P(g)=H_M(m_g)$ where $g=m_g u_g k_g$ is the Iwasawa decomposition of $g$. According to Harish-Chandra, we can define the height function $\Vert \cdot\Vert$ on $G(F)$, take values in $\BR_{\geq 1}$, and a log-norm $\sigma$ on $G(F)$ by $\sigma(g)=\sup(1,\log(\Vert g\Vert))$. Similarly, we can define the log-norm function on $\Fg(F)$ as follows: fix a basis $\{X_i\}$ of $\Fg(F)$ over $F$, for $X\in \Fg(F)$, let $\sigma(X)=\sup(1,\sup\{ -val_F(a_i)\})$, where $a_i$ is the $X_i$-coordinate of $X$.

For $x\in G$ (resp. $X\in \Fg$), let $Z_G(x)$(resp. $Z_G(X)$) be the centralizer of $x$ (resp. $X$) in $G$, and let $G_x$(resp. $G_X$) be the neutral component of $Z_G(x)$ (resp. $Z_G(X)$). Accordingly, let $\Fg_x$ (resp. $\Fg_X$) be the Lie algebra of $G_x$ (resp. $G_X$). For a function $f$ on $G(F)$ (resp. $\Fg(F)$), and $g\in G(F)$, let ${}^g f$ be the $g$-conjugation of $f$, i.e. ${}^g f(x)=f(g^{-1}xg)$ for $x\in G(F)$ (resp. ${}^g f(X)=f(g^{-1}Xg)$ for $X\in \Fg(F)$).

Denote by $G_{ss}(F)$ the set of semisimple elements in $G(F)$, and by $G_{reg}(F)$ the set of regular elements in $G(F)$. The Lie algebra versions
are denoted by $\Fg_{ss}(F)$ and $\Fg_{reg}(F)$, respectively. Now for $X\in G_{ss}(F)$, the operator $ad(x)-1$ is defined and invertible on $\Fg(F)/\Fg_x(F)$. We define
$$
D^G(x)=\mid \det((ad(x)-1)_{\mid \Fg(F)/\Fg_x(F)})\mid_F.
$$
Similarly for $X\in \Fg_{ss}(F)$, define
$$
D^G(X)=\mid \det((ad(X))_{\mid \Fg(F)/\Fg_X(F)})\mid_F.
$$
For any subset $\Gamma\subset G(F)$, define $\Gamma^G:=\{g^{-1}\gamma g\mid g\in G(F),\gamma\in \Gamma\}$. We say an invariant subset $\Omega$ of $G(F)$ is compact modulo conjugation if there exist a compact subset $\Gamma$ such that $\Omega\subset \Gamma^G$. A $G$-domain on $G(F)$ (resp. $\Fg(F)$) is an open subset of $G(F)$ (resp. $\Fg(F)$) invariant under the $G(F)$-conjugation.

For two complex valued functions $f$ and $g$ on a set $X$ with $g$ taking values in the positive real numbers, we write that
$$
f(x)\ll g(x)
$$
and say that $f$ is essentially bounded by $g$, if there exists a constant $c>0$ such that for all $x\in X$, we have
$$
| f(x)| \leq cg(x).
$$
We say $f$ and $g$ are equivalent, which is denoted by
$$f(x)\sim g(x)$$
if $f$ is essentially bounded by $g$ and $g$ is essentially bounded by $f$.

\subsection{Measures}
Through this paper, we fix a non-trivial additive character $\psi: F\rightarrow \BC^{\times}$. If $G$ is a connected reductive group, we may fix a non-degenerate symmetric bilinear form $<\cdot,\cdot>$ on $\Fg(F)$ that is invariant under $G(F)$-conjugation. For any smooth compactly supported complex valued function $f\in C_{c}^{\infty}(\Fg(F))$, we can define its Fourier transform $f\rightarrow \hat{f}\in C_{c}^{\infty}(\Fg(F))$ to be
\begin{equation}\label{FT}
\hat{f}(X)=\int_{\Fg(F)} f(Y) \psi(<X,Y>) dY
\end{equation}
where $dY$ is the selfdual Haar measure on $\Fg(F)$ such that $\hat{\hat{f}}(X)=f(-X)$. Then we get a Haar measure on $G(F)$ such that the exponential map has Jacobian equals to 1. If $H$ is a subgroup of $G$ such that the restriction of the bilinear form to $\Fh(F)$ is also non-degenerate, then we can define the measures on $\Fh(F)$ and $H(F)$ by the same method.

Let $Nil(\Fg)$ be the set of nilpotent orbits of $\Fg$. For $\CO\in Nil(\Fg)$ and $X\in \CO$, the bilinear form $(Y,Z)\rightarrow <X,[Y,Z]>$ on $\Fg(F)$ can be descent to a symplectic form on $\Fg(F)/\Fg_X(F)$. The nilpotent $\CO$  has naturally a structure of $F$-analytic symplectic variety, which yields a selfdual measure on $\CO$. This measure is invariant under the $G(F)$-conjugation.

If $T$ is a subtorus of $G$ such that the bilinear form is non-degenerate on $\Ft(F)$, we can provide a measure on $T$ by the method above, denoted by $dt$. On the other hand, we can define another measure $d_c t$ on $T(F)$ as follows: If $T$ is split, we require the volume of the maximal compact subgroup of $T(F)$ is 1 under $d_c t$. In general, $d_c t$ is compatible with the measure $d_c t'$ defined on $A_T(F)$ and with the measure on $T(F)/A_T(F)$ of total volume 1. Then we have a constant number $\nu(T)$ such that $d_c t=\nu(T)dt$. In this paper, we will only use the measure $dt$, but in many cases we have to include the factor $\nu(T)$.
Finally, if $M$ is a Levi subgroup of $G$, we can define the Haar measure on $\Fa_{M}^{G}$ such that the quotient
$$\Fa_{M}^{G}/proj_{M}^{G} (H_M(A_M(F)))$$
is of volume 1.

\subsection{$(G,M)$-families}
From now on until Section 4, $G$ will be a connected reductive group, and $\Fg(F)$ be its Lie algebra, with a bilinear pairing invariant under conjugation. For a Levi subgroup $M$ of $G$, we recall the notion of $(G,M)$-family introduced by Arthur. A $(G,M)$-family is a family $(c_P)_{P\in \CP(M)}$ of smooth functions on $i\Fa_{M}^{\ast}$ taking values in a locally convex topological vector space $V$ such that for all adjacent parabolic subgroups $P,P'\in \CP(M)$, the functions $c_p$ and $c_{P'}$ coincide on the hyperplane supporting the wall that separates the positive chambers for $P$ and $P'$. For such a $(G,M)$-family, one can associate an element $c_M\in V$ (\cite[Page 37]{Ar81}). If $L\in \CL(M)$, for a given $(G,M)$-family, we can deduce a $(G,L)$-family. Denote by $c_L$ the element in $V$ associated to such $(G,L)$-family. If $Q=L_Q U_Q\in \CF(L)$, we can deduce a $(L_Q,L)$-family from the given $(G,M)$-family, the element in $V$ associated to which is denoted by $c_{L}^{Q}$.

If $(Y_P)_{P\in \CP(M)}$ is a family of elements in $\Fa_M$, we say it is a $(G,M)$-orthogonal set (resp. and positive) if the following condition holds: if $P,P'$ are two adjacent elements of $\CP(M)$, there exists a unique coroot $\check{\alpha}$ such that $\check{\alpha}\in \check{\Delta}_P$ and $-\check{\alpha}\in \check{\Delta}_{P'}$, we require that $Y_P-Y_{P'}\in \BR \check{\alpha}$ (resp. $Y_P-Y_{P'}\in \BR_{\geq 0} \check{\alpha}$). For $P\in \CP(M)$, define a function $c_P$ on $i\Fa_{M}^{\ast}$ by $c_P(\lambda)=e^{-\lambda(Y_P)}$. Suppose the family $(Y_P)_{P\in \CP(M)}$ is a $(G,M)$-orthogonal set. Then the family $(c_P)_{P\in \CP(M)}$ is a $(G,M)$-family. If the family $(Y_P)_{P\in \CP(M)}$ is positive, then the number $c_M$ associated to this $(G,M)$-family is just the volume of the convex hull in $\Fa_{M}^{G}$ generated by the set $\{Y_P\mid P\in \CP(M)\}$. If $L\in \CL(M)$, the $(G,L)$-family deduced from this $(G,M)$-family is the $(G,L)$-family associated to the $(G,L)$-orthogonal set $(Y_Q)_{Q\in \CP(L)}$ where $Y_Q=proj_L(Y_P)$ for some $P\in \CP(M)$ such that $P\subset Q$ (it is easy to see this is independent of the choice of $P$). Similarly, if $Q\in \CP(L)$, then the $(L,M)$-family deduced from this $(G,M)$-family is the $(L,M)$-family associated to the $(L,M)$-orthogonal set $(Y_{P'})_{P'\in \CP^L(M)}$ where $Y_{P'}=Y_P$ with $P$ being the unique element of $\CP(M)$ such that $P\subset Q$ and $P\cap L=P'$.

\subsection{Weighted orbital integrals}
If $M$ is a Levi subgroup of $G$ and $K$ is a maximal open compact subgroup in good position with respect to $M$. For $g\in G(F)$, the family $(H_P(g))_{P\in \CP(M)}$ is $(G,M)$-orthogonal and positive. Let $(v_P(g))_{P\in \CP(M)}$ be the $(G,M)$-family associated to it and $v_M(g)$ be the number associated to this $(G,M)$-family. Then $v_M(g)$ is just the volume of the convex hull in $\Fa_{M}^{G}$ generated by the set $\{H_P(g),\; P\in \CP(M)\}$. The function $g\rightarrow v_M(g)$ is obviously left $M(F)$-invariant and right $K$-invariant.

If $f\in C_{c}^{\infty}(G(F))$ and $x\in M(F)\cap G_{reg}(F)$, define the weighted orbital integral to be
\begin{equation}\label{WOI}
J_M(x,f)=D^G(x)^{1/2} \int_{G_x(F)\backslash G(F)} f(g^{-1}xg)v_M(g)dg.
\end{equation}
Note the definition does depend on the choice of the hyperspecial open compact subgroup $K$. But we will see in later that if $f$ is strongly cuspidal, then this definition is independent of the choice of $K$.

\begin{lem}
With the notation as above, the following holds.
\begin{enumerate}
\item If $f\in C_{c}^{\infty}(G(F))$, the function $x\rightarrow J_M(x,f)$ defined on $M(F)\cap G_{reg}(F)$ is locally constant,
invariant under $M(F)$-conjugation and has a compact support modulo conjugation.
\item There exists an integer $k\geq 0$, such that for every $f\in C_{c}^{\infty}(G(F))$, there exists $c>0$ such that
$$
| J_M(x,f)| \leq c(1+| \log D^G(x) |)^k
$$
for every $x\in M(F)\cap G_{reg}(F)$.
\end{enumerate}
\end{lem}

\begin{proof}
See Lemma 2.3 of \cite{W10}.
\end{proof}

The next result is due to Harish-Chandra (Lemma 4.2 of \cite{Ar91}), which will be heavily used in Section 8 and Section 9. See \cite[Section 1.2]{B15} for a more general argument.
\begin{prop}\label{h-c}
Let $T$ be a torus of $G(F)$, and $\Gamma\subset G(F)$, $\Omega\subset T(F)$ be compact subsets. Then there exists $c>0$ such that for every $x\in \Omega\cap G(F)_{reg}$ and $g\in G(F)$ with $g^{-1}x g\in \Gamma$, we have
\begin{equation}\label{major 1}
\sigma_T(g)\leq c(1+\mid \log(D^G(x)) \mid)
\end{equation}
where $\sigma_T(g)=\inf\{\sigma(tg)\mid t\in T(F)\}$.
\end{prop}

\subsection{Shalika Germs}
For every $\CO\in Nil(\Fg)$ and $f\in C_{c}^{\infty}(\Fg(F))$, define the nilpotent orbital integral by
$$
J_{\CO}(f)=\int_{\CO} f(X)dX.
$$
Its Fourier transform is defined to be
$$\hat{J}_{\CO}(f)=J_{\CO}(\hat{f}).$$

For $\lambda\in F^{\times}$, define $f^{\lambda}$ to be $f^{\lambda}(X)=f(\lambda X)$. Then it is easy to see that
for $\lambda\in (F^{\times})^2$, we have
\begin{equation}\label{scalar 1}
J_{\CO}(f^{\lambda})=\mid \lambda \mid^{-dim(\CO)/2} J_{\CO}(f).
\end{equation}
Define $\delta(G)=\dim(G)-\dim(T)$, where $T$ is any maximal torus of $G$ (i.e. $\delta(G)$ is twice of the dimension of maximal unipotent subgroup if $G$ split). There exists a unique function $\Gamma_{\CO}$ on $\Fg_{reg}(F)$, called the Shalika germ associated to $\CO$, satisfies the following conditions:
\begin{equation}\label{Shalika germ 1}
\Gamma_{\CO}(\lambda X)=\mid \lambda\mid_{F}^{(\delta(G)-\dim(\CO))/2} \Gamma_{\CO}(X)
\end{equation}
for all $X\in \Fg_{reg}(F), \lambda\in (F^{\times})^2$, and for every $f\in C_{c}^{\infty}(\Fg(F))$, there exists an neighborhood $\omega$ of $0$ in $\Fg(F)$ such that
\begin{equation}\label{Shalika germ 2}
J_G(X,f)=\Sigma_{\CO\in Nil(\Fg)} \Gamma_{\CO}(X)J_{\CO}(f)
\end{equation}
for every $X\in \omega\cap \Fg_{reg}(F)$, where $J_G(X,f)$ is the orbital integral.

Harish-Chandra proved that there exists a unique function $\hat{j}$ on $\Fg_{reg}(F)\times \Fg_{reg}(F)$, which is locally constant on $\Fg_{reg}(F)\times \Fg_{reg}(F)$, and locally integrable on $\Fg(F)\times \Fg(F)$, such that for every $f\in C_{c}^{\infty}(\Fg(F))$ and every $X\in \Fg_{reg}(F)$,
\begin{equation}\label{representative 1}
J_G(X,\hat{f})=\int_{\Fg(F)} f(Y)\hat{j}(X,Y) dY.
\end{equation}
Also, for all $\CO\in Nil(\Fg)$, there exists a unique function $Y\rightarrow \hat{j}(\CO,Y)$ on $\Fg_{reg}(F)$,  which is locally constant on $\Fg_{reg}(F)$, and locally integrable on $\Fg(F)$, such that for every $f\in C_{c}^{\infty}(\Fg(F))$,
\begin{equation}\label{representative 2}
\hat{J}_{\CO}(f)=\int_{\Fg(F)} f(Y) \hat{j}(\CO,Y) dY.
\end{equation}
It follows that
\begin{eqnarray}\label{scalar 2}
\hat{j}(\lambda X,Y)
&=&\mid \lambda \mid_{F}^{\delta(G)/2} \hat{j}(X, \lambda Y), \\
\hat{j}(\CO, \lambda Y)
&=&\mid \lambda \mid_{F}^{\dim(\CO)/2} \hat{j}(\CO,Y)\nonumber
\end{eqnarray}
for all $X,Y\in \Fg_{reg}(F),\CO\in Nil(\Fg)$ and $\lambda\in (F^{\times})^2$.
Moreover, by the above discussion, if $\omega$ is an $G$-domain of $\Fg(F)$ that is compact modulo conjugation and contains $0$, there exists an $G$-domain $\omega'$ of $\Fg(F)$ that is compact modulo conjugation and contains $0$ such that for every $X\in \omega' \cap \Fg_{reg}(F)$ and
$Y\in \omega\cap \Fg_{reg}(F)$,
\begin{equation}\label{germ 1}
\hat{j}(X,Y)=\Sigma_{\CO\in Nil(\Fg)} \Gamma_{\CO}(X) \hat{j}(\CO,Y).
\end{equation}

\section{Quasi-Characters}
In this section we recall the definition and basic properties of quasi-characters in the p-adic case. For details, see \cite[Section 4]{W10} and \cite[Section 4]{B15}.

\subsection{Neighborhoods of Semisimple Elements}
\begin{defn}\label{good nbd defn}
For every $x\in G_{ss}(F)$, we say a subset $\omega\subset \Fg_x(F)$ is a good neighborhood of $0$ if it satisfies the following seven conditions, together with condition $(7)_{\rho}$ for finitely many finite dimensional algebraic representations $(\rho,V)$ of $G$ which will be fixed in advance (\cite[Section 3.1]{W10}):
\begin{itemize}
\item[(1)] $\omega$ is an $G_x$-domain, compact modulo conjugation, invariant under $Z_G(x)(F)$ conjugation and contains $0$.
\item[(2)] The exponential map is defined on $\omega$, i.e. it is a homeomorphism between $\omega$ and $\exp(\omega)$, and is $G_x$-equivalent,
    where the action is just conjugation.
\item[(3)] For every $\lambda\in F^{\times}$ with $\mid \lambda\mid \leq 1$, we have $\lambda \omega\subset \omega$.
\item[(4)] We have
\begin{equation}
\{ g\in G(F)\mid g^{-1}x\exp(\omega)g\cap x\exp(\omega)\neq \emptyset\}=Z_G(x)(F).
\end{equation}
\item[(5)] For every compact subset $\Gamma\subset G(F)$, there exists a compact subset $\Gamma'\subset G(F)$ such that
$$\{g\in G(F)\mid g^{-1}x\exp(\omega)g\cap \Gamma=\emptyset\}\subset G_x(F)\Gamma'.$$
\item[(6)] Fix a real number $c_F>0$ such that $c_{F}^{k}<\mid (k+1)!\mid_F$ for every integer $k\geq 1$. Then for every maximal subtorus $T\subset G_x$, every algebraic character $\chi$ of $T$ and every element $X\in \Ft(F)\cap \omega$, we have $\mid \chi(X)\mid_F<c_F$.
\item[(7)] Consider an eigenspace $W\subset \Fg(F)$ for the operator $ad(x)$ and let $\lambda$ be the eigenvalue. If $X\in \omega$, then $ad(X)$ preserve $W$. Let $W_X$ be an eigenspace of it with eigenvalue $\mu$. Then it is easy to see $W_X$ is also an eigenspace for the operator
    $ad(x \exp(X))$, with eigenvalue $\lambda \exp(\mu)$. Now suppose $\lambda\neq 1$. Then
$$\mid \lambda \exp(\mu)-1\mid_F=\mid \lambda-1\mid_F.$$
\item[$(7)_{\rho}$] If we fix a finite dimensional algebraic representation $(\rho,V)$ of $G$, by replacing the adjoint representation by $(\rho,V)$ in (7), we can define condition $(7)_{\rho}$ in a similar way.
\end{itemize}
\end{defn}
The properties for good neighborhoods are summarized below, the details of which will be referred to \cite[Section 3]{W10}.
\begin{prop}\label{good nbd}
The following hold.
\begin{enumerate}
\item  If $\omega_0$ is a neighborhood of $0$ in $\Fg_x(F)$, there exists a good neighborhood $\omega$ of 0 such that $\omega\subset \omega_{0}^{G_x}$.
\item $\Omega=(x\exp(\omega))^G$ is an $G$-domain in $G(F)$, and has compactly support modulo conjugation.
\item For every $X\in \omega$, $Z_G(x\exp(X))(F)\subset Z_G(x)(F)$ and $G_{x\exp(X)}=(G_x)_X\subset G_x$.
\item The exponential map between $\omega$ and $\exp(\omega)$ preserve measures, i.e. the Jacobian of the map equals $1$.
\item For every $X\in \omega$, $D^G(x\exp(X))=D^G(x) D^{G_x}(X)$.
\end{enumerate}
\end{prop}

\begin{proof}
See Section 3.1 of \cite{W10}.
\end{proof}

\subsection{Quasi-characters of $G(F)$}
If $\theta$ is a smooth function defined on $G_{reg}(F)$, invariant under $G(F)-$conjugation. We say it is a quasi-character on $G(F)$ if and only if, for every $x\in G_{ss}(F)$, there is a good neighborhood $\omega_x$ of $0$ in $\Fg_x(F)$, and for every $\CO\in Nil(\Fg_x)$, there exists coefficient $c_{\theta,\CO}(x)\in \BC$ such that
\begin{equation}\label{germ 2}
\theta(x\exp(X))=\Sigma_{\CO\in Nil(\Fg_x)} c_{\theta,\CO}(x) \hat{j}(\CO,X)
\end{equation}
for every $X\in \omega_{x,reg}$. It is easy to see that $c_{\theta,\CO}(x)$ are uniquely determined by $\theta$.
If $\theta$ is a quasi-character on $G(F)$ and $\Omega\subset G(F)$ is an open $G$-domain, then $\theta 1_{\Omega}$ is still a quasi-character.

\subsection{Quasi-characters of $\Fg(F)$}
Let $\theta$ be a function on $\Fg_{reg}(F)$, invariant under $G(F)-$conjugation. We say it is a quasi-character on $\Fg(F)$ if and only if for every $X\in \Fg_{ss}(F)$, their exists an open $G_X$-domain $\omega_X$ in $\Fg_X(F)$, containing $0$, and for every $\CO \in Nil(\Fg_X)$, there exists $c_{\theta,\CO}(X)\in \BC$ such that
\begin{equation}\label{germ 3}
\theta(X+Y)=\Sigma_{\CO\in Nil(\Fg_X)} c_{\theta,\CO}(X) \hat{j}(\CO,Y)
\end{equation}
for every $Y\in \omega_{X,reg}$.
If $\theta$ is a quasi-character on $\Fg(F)$, define $c_{\theta,\CO}=c_{\theta,\CO}(0)$. If $\lambda \in F^{\times}$, then $\theta^{\lambda}(X)=\theta(\lambda X)$ is still a quasi-character on $\Fg(F)$. By Section 4.2 of \cite{W10}, for every $\CO\in Nil(\Fg_X)$, we have
\begin{equation}\label{scalar 3}
c_{\theta^{\lambda},\CO}(\lambda^{-1}X)=\mid \lambda\mid_{F}^{-dim(\CO)/2} c_{\theta,\CO}(X).
\end{equation}

\subsection{Localization}
We fix $x\in G_{ss}(F)$ and a good neighborhood $\omega$ of $0$ in $\Fg_x(F)$. If $\theta$ is a quasi-character of $G(F)$, we define a function $\theta_{x,\omega}$ on $\omega$ by
\begin{equation}\label{local 1}
\theta_{x,\omega}(X)
=
\begin{cases}
\theta(x\exp(X)), & \text{if}\ X \in \omega; \\
0, & \text{otherwise}.
\end{cases}
\end{equation}
Then $\theta_{x,\omega}$ is a quasi-character of $\Fg_x(F)$, and we have $c_{\theta,\CO}(x\exp(X))=c_{\theta_{x,\omega},\CO}(X)$ for every $X\in \omega\cap \Fg_{x,ss}(F)$ and $\CO\in Nil(\Fg_{x,X})$ (Note we have $G_{x\exp(X)}=(G_x)_X$ since $\omega$ is a good neighborhood). In particular, by taking $X=0$ we have $c_{\theta,\CO}(x)=c_{\theta_{x,\omega},\CO}$ for every $\CO\in Nil(\Fg_x)$.

Now if $\theta$ is a quasi-character of $G(F)$ that is $Z_G(F)$-invariant, then
$$c_{\theta,\CO}(zx)=c_{\theta,\CO}(x)$$
for all $z\in Z_G$. For $\omega$ as above, we can define a quasi-character on $\Fg_x(F)$ that is invariant by $\Fz_{\Fg}(F)$, which is
still denoted by $\theta_{x,\omega}$, to be
\begin{equation}\label{local center 1}
\theta_{x,\omega}(X)
=
\begin{cases}
\theta(x\exp(X')), & \text{if}\ X=X'+Z,X' \in \omega, Z\in \Fz_{\Fg}(F); \\
0, & \text{otherwise}.
\end{cases}
\end{equation}

\section{Strongly Cuspidal Functions}
In this section, we will recall the definition and some basic properties of the strongly cuspidal functions. For details, see Sections 5 and 6 of \cite{W10}, and Section 5 of \cite{B15}.

\subsection{Definition and basic properties}
If $f\in C_{c}^{\infty}(Z_G(F)\backslash G(F))$, we say $f$ is strongly cuspidal if and only if for every proper parabolic subgroup $P=MU$ of $G$, and for every $x\in M(F)$, we have
\begin{equation}\label{cuspidal 1}
\int_{U(F)} f(xu)du=0.
\end{equation}
The most basic example for strongly cuspidal functions is the matrix coefficient of a supercuspidal representation.

The following proposition is easy to prove, following mostly from the definition. See Section 5.1 of \cite{W10}.
\begin{prop}
The following hold.
\begin{enumerate}
\item $f$ is strongly cuspidal if and only if for every proper parabolic subgroup $P=MU$ of $G$, and for every $x\in M(F)$, we have
\begin{equation}\label{cuspidal 2}
\int_{U(F)} f(u^{-1}xu) du=0.
\end{equation}
\item If $\Omega$ is a $G$-domain in $G(F)$ and if $f$ is strongly cuspidal, then $f 1_{\Omega}$ is strongly cuspidal.
\item If f is strongly cuspidal, so is ${}^g f$ for every $g\in G(F)$.
\end{enumerate}
\end{prop}

Now we study the weighted orbital integral associated to strongly cuspidal functions. The following lemma is proved in Section 5.2 of \cite{W10}.
\begin{lem}
Let $M$ be a Levi subgroup of $G$ and $K$ be a hyperspecial open compact subgroup with respect to $M$.
If $f\in C_{c}^{\infty}(Z_G(F)\backslash G(F))$ is strongly cuspidal and $x\in M(F)\cap G_{reg}(F)$, then the following hold.
\begin{enumerate}
\item The weighted orbital integral $J_M(x,f)$ does not depend on the choice of $K$.
\item For every $y\in G(F)$, we have $J_M(x,{}^y f)=J_M(x,f)$.
\item If $A_{G_x}\neq A_M$, then $J_M(x,f)=0$.
\end{enumerate}
\end{lem}

For $x\in G_{reg}(F)$, let $M(x)$ be the centralizer of $A_{G_x}$ in $G$, which is clearly a Levi subgroup of $G$.
For any strongly cuspidal $f$ belonging to the space $C_{c}^{\infty}(Z_G(F)\backslash G(F))$,
define the function $\theta_f$ on $Z_G(F)\backslash G_{reg}(F)$ by
\begin{equation}\label{distribution 1}
\theta_f(x)=(-1)^{a_{M(x)}-a_G} \nu(G_x)^{-1} D^G(x)^{-1/2} J_{M(x)}(x,f).
\end{equation}
Here $a_G$ is the dimension of $A_G$, and the same for $a_{M(x)}$. By the lemma above, the weighted orbital integral is independent of
the choice of the hyperspecial open compact subgroup, and so is the function $\theta_f$.

\begin{prop}\label{distribution 2}
The following hold.
\begin{enumerate}
\item The function $\theta_f$ is invariant under $G(F)$-conjugation, has a compact support modulo conjugation and modulo the center, is locally integrable on $Z_G(F)\backslash G(F)$ and locally constant on $Z_G(F)\backslash G_{reg}(F)$.
\item $\theta_f$ is a quasi-character.
\end{enumerate}
\end{prop}

\begin{proof}
The first part is Lemma 5.3 of \cite{W10}, the second part is Corollary 5.9 of the loc. cit.
\end{proof}

Here we only write down the results for the trivial central character case, but the argument can be easily extended to the non-trivial central character case (i.e. $f\in C_{c}^{\infty}(Z_G(F)\backslash G(F),\chi)$), or the case without central character (i.e. $f\in C_{c}^{\infty}(G(F))$).

\subsection{The Lie algebra case}
\begin{defn}
We say a function $f\in C_{c}^{\infty}(\Fg_0(F))$ is strongly cuspidal if for every proper parabolic subgroup $P=MU$, and for every $X\in \Fm(F)$, we have
$$\int_{\Fu(F)} f(X+Y)dY=0.$$
This is equivalent to say that for every proper parabolic subgroup $P=MU$, and for every $X\in \Fm(F)$, we have
$$\int_{U(F)} f(u^{-1}Xu)du=0.$$
\end{defn}
If $f\in C_{c}^{\infty}(\Fg_0(F))$ is strongly cuspidal, we define a function $\theta_f$ on $\Fg_{0,reg}(F)$ by
\begin{equation}\label{distribution 3}
\theta_f(X)=(-1)^{a_{M(X)}-a_G} \nu(G_X)^{-1} D^G(X)^{-1/2} J_{M(X)}(X,f).
\end{equation}
Here $M(X)$ is the centralizer of $A_{G_X}$ in $G$, $a_G$ is the dimension of $A_G$, and the same for $a_{M(X)}$. We have a similar result as Proposition \ref{distribution 2}.

\begin{prop}\label{distribution 4}
If $f\in C_{c}^{\infty}(\Fg_0(F))$ is strongly cuspidal, $\theta_f$ is independent of the choice of $K$. (Recall we need to fix the open compact subgroup $K$ in the definition of orbital integral.) And in this case, $\theta_f$ is a quasi-character.
\end{prop}

\subsection{Localization}
In this section, we discuss the localization of the quasi-character $\theta_f$, which will be used in the localization of the trace formula in Section 7. Some results of this section will also be used in Section 9 when we trying to change the truncated function in the trace formula. For $x\in G_{ss}(F)$, recall that $\Fg_{x,0}$ is the subspace of elements in $\Fg_x$ whose trace is zero.
Suppose $\Fg_{x,0}=\Fg_{x}'\oplus \Fg''$ where $\Fg_{x}'$ and $\Fg''$ are the Lie algebras of some connected reductive groups (See Section 7.2).
For any element $X\in \Fg_{x,0}(F)$, it can be decomposed as $X=X'+X''$ for $X'\in \Fg_{x}'$ and $X''\in \Fg''$.
We denote by $f\rightarrow f^{\sharp}$ the partial Fourier transform for $f\in C_{x}^{\infty}(\Fg_{x,0}(F))$ with respect to $X''$. i.e.
\begin{equation}\label{FT 1}
f^{\sharp}(X)=\int_{\Fg''(F)} f(X'+Y'')\psi(<Y'',X''>) dY''.
\end{equation}
Let $\omega$ be a good neighborhood of $0$ in $\Fg_x$. We can also view $\omega$ as an neighborhood of $0$ in $\Fg_{x,0}$ by considering
its image in $\Fg_{x,0}$ under the projection $\Fg_x\rightarrow \Fg_{x,0}$. If $f\in C_{c}^{\infty}(Z_G(F)\backslash G(F))$, for $g\in G(F)$, define ${}^g f_{x,\omega} \in C_{c}^{\infty} (\Fg_{x,0}(F))$ by
\begin{equation}\label{local center 2}
{}^g f_{x,\omega}(X)
=\begin{cases}
f(g^{-1}x\exp(X)g), & \text{if}\  X \in \omega; \\
0, & \text{otherwise}.
\end{cases}
\end{equation}
Also define
\begin{equation}
{}^g f_{x,\omega}^{\sharp}=({}^g f_{x,\omega})^{\sharp}.
\end{equation}
Note that for $X\in \Fg_{x,0}(F)$, $X\in \omega$ means there exist $X'\in \omega$ and $Z\in \Fz_{\Fg}(F)$ such that $X=X'+Z$. It follows that
the value $f(g^{-1}x\exp(X)g)$ is just $f(g^{-1}x\exp(X')g)$, which is independent of the choice of $X'$ and $Z$.

If $M$ is a Levi subgroup of $G$ containing the given $x$, fix a hyperspecial open compact subgroup $K$ with respect to $M$.
If $P=MU\in \CP(M)$, for $f\in C_{c}^{\infty}(Z_G(F)\backslash G(F))$, define the functions $\varphi[P,f],\varphi^{\sharp}[P,f]$ and $J_{M,x,\omega}^{\sharp}(\cdot, f)$ on $\Fm_{x,0}(F)\cap \Fg_{x,reg}(F)$ by
\begin{equation}
\varphi[P,f](X)=D^{G_x}(X)^{1/2} D^{M_x}(X)^{-1/2} \int_{U(F)} {}^u f_{x,\omega}(X) du,
\end{equation}
\begin{equation}
\varphi^{\sharp}[P,f](X)=D^{G_x}(X)^{1/2} D^{M_x}(X)^{-1/2} \int_{U(F)} {}^u f_{x,\omega}^{\sharp}(X) du,
\end{equation}
and
\begin{equation}\label{6.1}
J_{M,x,\omega}^{\sharp}(X, f)=D^{G_x}(X)^{1/2}\int_{G_{x,X}(F)\backslash G(F)} {}^g f_{x,\omega}^{\sharp}(X) v_M(g)dg.
\end{equation}
The following two lemmas are proved in Sections 5.4 and 5.5 of \cite{W10}, which will be used in the localization of the trace formula. The second lemma will also be used in Section 9 when we trying to change the truncated function in the trace formula.
\begin{lem}
The following hold.
\begin{enumerate}
\item The three integrals above are absolutely convergent.
\item The function $\varphi[P,f]$ and $\varphi^{\sharp}[P,f]$ can be extended to elements in $C_{c}^{\infty}(\Fm_{x,0}(F))$ and we have $(\varphi[P,f])^{\sharp}=\varphi^{\sharp}[P,f]$.
\item The function $X\rightarrow J_{M,x,\omega}^{\sharp}(X, f)$ is invariant under $M_x(F)$-conjugation, and has a compactly support
modulo conjugation. Further, it is locally constant on $\Fm_{x,0}(F)\cap \Fg_{x,reg}(F)$, with the property that
there exist $c>0$ and an integer $k\geq 0$ such that
$$\mid J_{M,x,\omega}^{\sharp}(X, f)\mid \leq c(1+\mid \log(D^{G_x}(X))\mid)^k$$
for every $X\in \Fm_{x,0}\cap \Fg_{x,reg}(F)$.
\end{enumerate}
\end{lem}

\begin{lem}\label{strongly cuspidal}
Suppose $f$ is strongly cuspidal.
\begin{enumerate}
\item If $P\neq G$, the function $\varphi[P,f]$ and $\varphi^{\sharp}[P,f]$ are zero.
\item The function $J_{M,x,\omega}^{\sharp}(\cdot, f)$ does not depend on the choice of $K$. It is zero if $A_{M_x}\neq A_M$.
For every $y\in G(F)$ and $X\in \Fm_{x,0}\cap \Fg_{x,reg}(F)$, we have
$$J_{M,x,\omega}^{\sharp}(X, f)=J_{M,x,\omega}^{\sharp}(X, {}^y f).$$
\end{enumerate}
\end{lem}

For $f\in C_{c}^{\infty}(Z_G(F)\backslash G(F))$ strongly cuspidal, we define a function $\theta_{f,x,\omega}$ on $(\Fg_{x,0})_{reg}$ by
\begin{equation}\label{6.3}
\theta_{f,x,\omega}(X)
=\begin{cases}
\theta_f(x\exp(X)), & \text{if}\  X \in \omega; \\
0, & \text{otherwise}.
\end{cases}
\end{equation}
If $X\in (\Fg_{x,0})_{reg}$, let $M(X)$ be the centralizer of $A_{G_{x,X}}$ in $G$. We define
\begin{equation}\label{6.2}
\theta_{f,x,\omega}^{\sharp}(X)=(-1)^{a_{M(X)}-a_G} \nu(G_{x,X})^{-1} D^{G_x}(X)^{-1/2} J_{M(X),x,\omega}^{\sharp}(X, f)
\end{equation}
By the lemma above this is independent of the choice of $K$. From the discussion of $\theta_f$, we have a similar lemma:

\begin{lem}
The functions $\theta_{f,x,\omega}$ and $\theta_{f,x,\omega}^{\sharp}$ are invariant under $G_x(F)$-conjugation, compactly supported modulo conjugation, locally integrable on $\Fg_{x,0}(F)$, and locally constant on $\Fg_{x,0,reg}(F)$.
\end{lem}

The next result about $\theta_{f,x,\omega}$ and $\theta_{f,x,\omega}^{\sharp}$ is proved in Section 5.8 of \cite{W10}. It tells us $\theta_{f,x,\omega}^{\sharp}$ is the partial Fourier transform of $\theta_{f,x,\omega}$ with respect to $X''$.

\begin{prop}\label{partial FT}
If $f\in C_{c}^{\infty}(Z_G(F)\backslash G(F))$ is strongly cuspidal, then $\theta_{f,x,\omega}^{\sharp}$ is the partial Fourier transform of $\theta_{f,x,\omega}$ in the sense that, for every $\varphi\in C_{c}^{\infty}(\Fg_{x,0}(F))$, we have
\begin{equation}
\int_{\Fg_{x,0}(F)} \theta_{f,x,\omega}^{\sharp}(X)\varphi(X) dX=\int_{\Fg_{x,0}(F)} \theta_{f,x,\omega}(X)\varphi^{\sharp}(X) dX.
\end{equation}
\end{prop}

\section{Statement of the Trace Formula}
In this section, we will write down our trace formula for both the group case and the Lie algebra case. In Section 5.1, we define all the ingredients in the geometric expansion. In 5.2, we define our truncated function $\kappa_N$ and state the trace formula in Theorem \ref{main 1} and Theorem \ref{main 3}. We also show that it is enough to prove the trace formula for the functions with trivial central character (see Proposition \ref{center issue}). In 5.3, we state the Lie algebra version of the trace formula.
\subsection{The ingredients of the integral formula}
From this section and on, unless otherwise specified, we consider the Ginzburg-Rallis model. This is to consider a pair $(G,H)$, which is either $(GL_6(F),GL_2(F))$ or $(GL_3(D),GL_1(D))$. Let $P=MU$ be the parabolic subgroup of the form $\begin{pmatrix} A & X & Z \\ 0 & B & Y \\ 0 & 0 & C \end{pmatrix}$ where $A,B,C$ belong to $GL_2(F)$ (the split case) or $GL_1(D)$ (the non-split case), and $X,Y,Z$ belong to $M_2(F)$ (the split case) or $D$
(the non-split case). We can diagonally embed $H$ into $M$, and define the character $\xi$ on $U(F)$ by
\begin{equation}\label{xi}
\xi(\begin{pmatrix} 1 & X & Z \\ 0 & 1 & Y \\ 0 & 0 & 1 \end{pmatrix})=\psi(a\tr(X)+b\tr(Y))
\end{equation}
for some $a,b\in F^{\times}$.

\begin{defn}\label{delta 1}
We define a function $\Delta$ on $H_{ss}(F)$ by
$$
\Delta(x)=\mid \det((1-ad(x)^{-1})_{\mid U(F)/U_x(F)}) \mid_F.$$
Similarly, we can define $\Delta$ on $\Fh_{ss}(F)$ by
$$\Delta(X)=\mid \det((ad(X))_{\mid \Fu(F)/\Fu_x(F)}) \mid_F.$$
\end{defn}

Let $\CT$ be a subset of subtori of $H$ defined as follows:
\begin{itemize}
\item If $H=\GL_2(F)$, then $\CT$ contain the trivial torus $\{1\}$ and the non-split torus $T_v$ for $v\in F^{\times}/(F^{\times})^2, v\neq 1$ where $T_v=\{\begin{pmatrix} a & bv \\ b & a \end{pmatrix} \in H(F) \mid a,b\in F,\;(a,b)\neq (0,0)\}$.
\item If $H=\GL_1(D)$, then $\CT$ contain the subtorus $T_v$ for $v\in F^{\times}/(F^{\times})^2$ with $v\neq 1$, where $T_v\subset D$ is isomorphic to the quadratic extension $F(\sqrt{v})$ of $F$.
\end{itemize}
Let $\theta$ be a quasi-character on $Z_G(F)\backslash G(F)$, and $T\in \CT$. If $T=\{1\}$, we are in the split case. In this case, we have a unique regular nilpotent orbit $\CO_{reg}$ in $\Fg(F)$ and take $c_\theta(1)=c_{\theta,\CO_{reg}}(1)$. If $T=T_v$ for some $v\in F^{\times}/(F^{\times})^2$
with $v\neq 1$, we take $t\in T_v$ to be a regular element (in $H(F)$). It is easy to see in both cases that $G_t(F)$ is $F$-isomorphic to $\GL_3(F_v)$ where $F_v=F(\sqrt{v})$ is the quadratic extension of $F$. Let $\CO_v$ be the unique regular nilpotent orbit in $\Fg \Fl_3(F_v)$, and take $c_\theta(t)=c_{\theta,\CO_v}(t)$.

\begin{prop}\label{integrable 1}
The function $c_{\theta}$ is locally constant on $T_{reg}(F)$ (here regular means as an element in $H(F)$). And the function $t\rightarrow c_{\theta}(t) D^H(t) \Delta(t)$ is locally integrable on $T(F)$.
\end{prop}

The first part follows from the definition. The rest of this section is to prove the second part, the idea comes from \cite{W10}. If $T=\{1\}$, there is nothing to prove since the integral is just evaluation. If $T=T_v$ for some $v\in F^{\times}/(F^{\times})^2$ with $v\neq 1$, since $c_{\theta}(t) D^H(t) \Delta(t)$ is locally constant on $T_{reg}(F)$, and
is invariant under $Z_H(F)$, we only need to show that the function is locally integrable around $t=1$.

We need some preparations. For a finite dimensional vector space $V$ over $F$, and any integer $i\in \BZ$, let $C_i(V)$ be the space of functions $\varphi: V\rightarrow \BC$ such that
$$\varphi(\lambda v)=|\lambda|^i \varphi(v)$$
for every $v\in V$ and $\lambda\in (F^{\times})^2$. Then we let $C_{\geq i}(V)$ be the space of functions which are linear combinations of functions in $C_j(V)$ for $j\geq i$. For $T=T_v$ and $i\in \BZ$, define the space $C_{\geq i}(T)$ to be the functions $f$ on $T_{reg}(F)$ such that there is a neighborhood $\omega$ of $0$ in $\Ft(F)$ and a function $\varphi\in C_{\geq i}(\Ft_0(F))$ such that
$$f(\exp(X))=\varphi(\bar{X})$$
for all $0\neq X\in \omega$, here $\bar{X}$ is the projection of $X$ in $\Ft_0(F)$. Then by \cite[Lemma 7.4]{W10}, if $f\in C_{\geq 0}(T)$, $f$ is locally integrable around $t=1$. Hence we only need to show that the function $t\rightarrow c_{\theta}(t) D^H(t) \Delta(t)$ lies inside the space $C_{\geq 0}(T)$.

Since if $\omega$ is small enough, we have
$D^H(\exp(X))=D^H(X),\; \Delta(\exp(X))=\Delta(X)$
for all $0\neq X\in \omega$. Hence the function $t\rightarrow D^H(t) \Delta(t)$ lies inside the space $C_{\geq 8}(T)$ where $8=\delta(H)+dim(U_X)$. Therefore in order to prove Proposition \ref{integrable 1}, it is enough to prove the following lemma.

\begin{lem}
With the notations above, the function $t\rightarrow c_{\theta}(t)$ belongs to the space $C_{\geq -8}(T)$.
\end{lem}

\begin{proof}
By Section 3.4, if we choose $\omega$ small enough, we have
$$c_{\theta}(\exp(X))=c_{\theta_{1,\omega},\CO_X}(X)$$
for all $0\neq X\in \omega$. Here $\theta_{1,\omega}$ is the localization of $\theta$ at $1$ defined in Section 3.4, and $\CO_X$ is the unique regular nilpotent orbit in $\Fg_X$. Since in a small neighborhood of $0\in \Fg_0(F)$, $\theta_{1,\omega}$ is a linear combination of $\hat{j}(\CO,\cdot)$ where $\CO$ runs over the nilpotent orbit in $\Fg_0$. Hence we may assume that $\theta_{1,\omega}=\hat{j}(\CO,\cdot)$ for some $\CO$.

If $\CO$ is regular, then we are in the split case (i.e. $G=\GL_6(F)$) and $\CO$ is the unique regular nilpotent orbit in $\Fg_0$. As a result, the distribution $\hat{j}(\CO,\cdot)$ is induced from the Borel subgroup and hence only supported in the Borel subalgebra. But by our construction of $T=T_V$, for any $t\in T_{reg}(F)$, we can always find a small neighborhood of $t$ in $G$ such that any element in such neighborhood does not belongs to the Borel subalgebra. Therefore the function $c_{\theta}(t)$ is identically zero, hence the function $t\rightarrow c_{\theta}(t) D^H(t) \Delta(t)$ is obviously locally integrable.

If $\CO$ is not regular, by \eqref{scalar 1} and \eqref{scalar 3}, the function $c_{\theta_{1,\omega},\CO_X}(X)$ belongs to the space $C_{\frac{dim(\CO_X)-dim(\CO)}{2}}(\Ft_0)$. The dimension of $\CO_X$ is equal to $\delta(G_X)=12$. On the other hand, since $\CO$ is not regular, $dim(\CO)\leq \delta(G)-2=28$. Hence the function $c_{\theta_{1,\omega},\CO_X}(X)$ belongs to the space $C_{\geq -8}(\Ft_0)$. This finishes the proof of the lemma, and hence the proof of Proposition \ref{integrable 1}.
\end{proof}

\subsection{The Main Theorem}
Let $f\in C_{c}^{\infty} (Z_G(F)\backslash G(F))$ be a strongly cuspidal function. For each $T\in \CT$, let $c_f$ be the function $c_{\theta_f}$ defined in the last section. Define
\begin{equation}\label{geometry 1}
I(f)=\sum_{T\in \CT} | W(H,T)|^{-1} \nu(T) \int_{Z_G(F)\backslash T(F)} c_f(t) D^H(t) \Delta(t) dt.
\end{equation}
Since for any $T\in \CT$, $Z_G(F)\backslash T$ is compact, the absolute convergence of the integral above follows from Proposition \ref{integrable 1}.

Now for $g\in G(F)$, we define the function ${}^g f^{\xi}$ on $H(F)/Z_H(F)$ by
$$
{}^g f^{\xi}(x)=\int_{U(F)} f(g^{-1}xug)\xi(u) du.
$$
This is a function belonging to $C_{c}^{\infty}(Z_H(F)\backslash H(F))$. Define
\begin{equation}\label{7.1}
I(f,g)=\int_{Z_H(F)\backslash H(F)} {}^g f^{\xi}(x) dx,
\end{equation}
and for each $N\in \BN$, define
\begin{equation}\label{spectral 0}
I_N(f)=\int_{U(F)H(F)\backslash G(F)} I(f,g) \kappa_N(g) dg.
\end{equation}
Here $\kappa_N$ is a characteristic function on $G(F)$ defined below, which is left $U(F)H(F)$-invariant, right $K$-invariant,
and compactly supported modulo $U(F)H(F)$:
If $G$ is split (i.e. $G=GL_6(F)$), for $g\in G(F)$, let $g=umk$ be its Iwasawa-decomposition with $u\in U(F)$, $m\in M(F)$ and $k\in K$. Then $m$ is of the form $\diag(m_1,m_2,m_3)$ with $m_i\in GL_2(F)$. For any $1\leq i,j\leq 3$,
let $m_{i}^{-1}m_j=\begin{pmatrix} a_{ij} & c_{ij} \\ 0 & b_{ij} \end{pmatrix} k_{ij}$ be its Iwasawa decomposition. We define $\kappa_N$ to be
\begin{equation}\label{kappa split}
\kappa_N(g)
=\begin{cases}
1, & \text{if}\  \sigma(a_{ij}),\sigma(b_{ij})\leq N, \sigma(c_{ij})\leq (1+\epsilon)N; \\
0, & \text{otherwise}.
\end{cases}
\end{equation}
Here $\epsilon>0$ is a fixed positive real number. Note that we do allow some more freedom on the unipotent part, which will
be used when we are trying to change our truncated function to the one given by Arthur in his local trace formula. For details, see Section 9.
If $G$ is not split (i.e. $G=GL_3(D)$), we still have the Iwasawa decomposition $g=umk$ with $m=\diag(m_1,m_2,m_3)$, and $m_i\in GL_1(D)$. We define $\kappa_N$ to be
\begin{equation}\label{kappa nonsplit}
\kappa_N(g)
=\begin{cases}
1, & \text{if}\  \sigma(m_{i}^{-1}m_j)\leq N; \\
0, & \text{otherwise}.
\end{cases}
\end{equation}
It follows that the integral in \eqref{spectral 0} is absolutely convergent because the integrand is compactly supported.

\begin{thm}\label{main 1}
For every function $f\in C_{c}^{\infty} (Z_G(F)\backslash G(F))$ that is strongly cuspidal, the following holds:
\begin{equation}\label{equation 3}
\lim_{N\rightarrow \infty} I_N(f)=I(f).
\end{equation}
\end{thm}

Theorem \ref{main 1} is the geometric side of the local relative trace formula for the Ginzburg-Rallis model, and will be proved from Section 7 to 10. Here we assume that $f$ has trivial central character. However, it is also possible to consider the case that $f$ has a non-trivial central character.
If $f$ has a central character $\chi$ with $\chi=\omega^2$ for some character $\omega$ (i.e. $f\in C_{c}^{\infty} (Z_G(F)\backslash G(F),\chi)$),
define
\begin{equation}
I_{\chi}(f)=\sum_{T\in \CT}|W(H,T)|^{-1} \nu(T) \int_{Z_G(F)\backslash T(F)} c_f(t) D^H(t) \Delta(t) \omega(\det(t))^{-1} dt,
\end{equation}
\begin{equation}
I_{\chi}(f,g)=\int_{Z_H(F)\backslash H(F)} {}^g f^{\xi}(x) \omega(\det(x))^{-1} dx,
\end{equation}
and for each $N\in \BN$, define
\begin{equation}\label{spectral 1}
I_{N,\chi}(f)=\int_{U(F)H(F)\backslash G(F)} I(f,g) \kappa_N(g) dg.
\end{equation}
Then Theorem \ref{main 1} takes form:

\begin{thm}\label{main 3}
For every function $f\in C_{c}^{\infty} (Z_G(F)\backslash G(F),\chi)$ that is strongly cuspidal, the following holds:
\begin{equation}\label{equation 3}
\lim_{N\rightarrow \infty} I_{N,\chi}(f)=I_{\chi}(f).
\end{equation}
\end{thm}

\begin{prop}\label{center issue}
Theorem \ref{main 1} implies Theorem \ref{main 3}.
\end{prop}

\begin{proof}
Note that both side of \eqref{equation 3} are linear on $f$. Since
$$Z_G(F)\backslash G(F)/\{g\in G(F)\mid \det(g)=1\}$$
is finite, we can localize $f$ such that $f$ is supported on
$$Z_G(F)g_0 \{g\in G(F)\mid \det(g)=1\}$$
for some $g_0\in G(F)$. Let $G_0(F)=\{g\in G(F)\mid \det(g)=1\}$, which is $SL_6(F)$ or $SL_3(D)$. Fix a fundamental domain $X\subset G_0(F)$ of $G_0(F)/(Z_G(F)\cap G_0(F))=G_0(F)/Z_{G_0}(F)$. We may choose $X$ so that it is open in $G_0(F)$. It is easy to see that $Z_{G_0}(F)$ is finite.
By further localizing $f$ we may assume that $f$ is supported on $Z_G(F)g_0 X$. Define a function $f' \in C_{c}^{\infty} (Z_G(F)\backslash G(F))$ to be
\begin{equation}
f'(g)
=\begin{cases}
f(g'), & \text{if}\  g=g'z,\; g'\in g_0X,\; z\in Z_G(F); \\
0, & \text{otherwise}.
\end{cases}
\end{equation}
It is easy to see that $f'$ is well defined and is strongly cuspidal, and can be viewed as the extension by trivial central character of the function $f\mid_{g_0X}$. Now we have
\begin{eqnarray*}
&&\int_{Z_G(F)\backslash T(F)} c_f(t) D^H(t) \Delta(t) \omega(\det(t))^{-1} dt\\
&=&\int_{T(F)\cap (g_0 X)} c_f(t) D^H(t) \Delta(t) \omega(\det(t))^{-1} dt\\
&=&\int_{T(F)\cap (g_0 X)} c_{f'}(t) D^H(t) \Delta(t) \omega(\det(g_0))^{-1} dt\\
&=&\omega^{-1}(\det(g_0))\int_{T(F)\cap (g_0 X)} c_{f'}(t) D^H(t) \Delta(t) dt\\
&=&\omega^{-1}(\det(g_0))\int_{Z_G(F)\backslash T(F)} c_{f'}(t) D^H(t) \Delta(t) dt
\end{eqnarray*}
and
\begin{eqnarray*}
I_{\chi}(f,g)&=&\int_{Z_H(F)\backslash H(F)} {}^g f^{\xi}(x) \omega(\det(x))^{-1} dx\\
&=&\int_{H(F)\cap (g_0 X)} {}^g f^{\xi}(x) \omega(\det(x))^{-1} dx\\
&=&\int_{H(F)\cap (g_0 X)} {}^g (f')^{\xi}(x) \omega(\det(g_0))^{-1} dx\\
&=&\omega^{-1}(\det(g_0))\int_{H(F)\cap (g_0 X)} {}^g (f')^{\xi}(x) dx\\
&=&\omega^{-1}(\det(g_0))\int_{\zh} {}^g (f')^{\xi}(x) dx\\
&=&\omega^{-1}(\det(g_0)) I(f',g).
\end{eqnarray*}
This implies
\begin{equation}\label{center 2}
I_{\chi}(f)=\omega^{-1}(\det(g_0))I(f'),\; I_{N,\chi}(f)=\omega^{-1}(\det(g_0)) I_N(f).
\end{equation}
Applying Theorem \ref{main 1} to the function $f'$, we know
$$\lim_{N\rightarrow \infty} I_{N}(f')=I(f').$$
Combining it with \eqref{center 2}, we prove \eqref{equation 3}, and this finishes the proof of Theorem \ref{main 3}.

\end{proof}

\subsection{The Lie Algebra Case}
Let $f\in C_{c}^{\infty} (\Fg_0(F))$ be a strongly cuspidal function. Define the function $f^{\xi}$ on $\Fh_0(F)$ by
$$f^{\xi}(Y)=\int_{\Fu(F)} f(Y+N)\xi(N) dN.$$
For $g\in G(F)$, define
$$I(f,g)=\int_{\Fh_0(F)} {}^g f^{\xi}(Y) dY,$$
and for each $N\in \BN$, define
\begin{equation}\label{spectral 2}
I_N(f)=\int_{U(F)H(F)\backslash G(F)} I(f,g) \kappa_N(g) dg.
\end{equation}
As in Section 5.1, for each $T\in \CT$, we can define the function $c_f=c_{\theta_f}$ on $\Ft_{0,reg}(F)$, and define
\begin{equation}\label{geometry 2}
I(f)=\sum_{T\in \CT} \mid W(H,T)\mid^{-1} \nu(T) \int_{t_0(F)} c_f(Y) D^H(Y) \Delta(Y) dY.
\end{equation}

By a similar argument as Proposition \ref{integrable 1}, we know that the integral in \eqref{geometry 2} is absolutely convergent. The following theorem can be viewed as the Lie algebra version of Theorem \ref{main 1}.

\begin{thm}\label{main 2}
For every strongly cuspidal function $f\in C_{c}^{\infty} (\Fg_0(F))$, we have
\begin{equation}\label{equation 4}
\lim_{N\rightarrow \infty} I_N(f)=I(f).
\end{equation}
\end{thm}

This is the Lie algebra version of the local relative trace formula for the Ginzburg-Rallis model, which will be proved in Section \ref{prove of theorem}.

\section{Proof of Theorem \ref{main}}
In this section, we prove Theorem \ref{main} by assuming that Theorem \ref{main 1} holds. In particular, by Proposition \ref{center issue}, we know Theorem \ref{main 3} holds.
\subsection{Definition of multiplicity}
Let $(G,H)=(GL_6(F),GL_2(F))$ or $(GL_3(D),GL_1(D))$. Given $(\pi, E_{\pi})\in Irr(G)$, and a character $\chi:F^{\times}\rightarrow \BC^{\times}$ such that $\omega_{\pi}=\chi^2$. In Section 1, we have defined the representation $\sigma\otimes \xi$ of $R(F)$, and we have defined the Hom space
$$\Hom_{R(F)} (\pi, \sigma\otimes \xi).$$
We use $m(\pi)$ to denote the dimension of such a space.

For $T\in \CT$ as in Section 5.1, let $\theta=\theta_{\pi}$ and $c_{\pi}(t)=c_{\theta_{\pi}}(t)$. We define the geometric multiplicity to be
\begin{equation}\label{geom mult}
m_{geom}(\pi)=\Sigma_{T\in \CT} \mid W(H,T)\mid^{-1} \nu(T)\int_{Z_H(F)\backslash T(F)} \sigma^{-1}(h) c_{\pi}(t) D^H(t) \Delta(t)dt.
\end{equation}
By Proposition \ref{integrable 1}, the integral above is absolutely convergent.

\begin{prop}\label{prop 12}
Suppose that $\pi$ is supercuspidal. Then
$$m(\pi)=m_{geom}(\pi).$$
\end{prop}

\begin{proof}
Consider the induced representation $\rho'=\Ind_{H(F)U(F)}^{G(F)}(\sigma\otimes \bar{\xi})$ and
the induced representation $\rho=\ind_{H(F)U(F)}^{G(F)}(\check{\sigma}\otimes \xi)$. The first one is the smooth induction with $\bar{\xi}$ be the complex conjugation of $\xi$, and the second one is the compact induction with $\check{\sigma}$ be the contragredient of $\sigma$. Then by the Frobenius reciprocity, we have
\begin{equation}
\Hom_{H,\bar{\xi}}(\pi,\sigma)=\Hom_{G(F)}(\pi,\rho')=\Hom_{G(F)}(\rho,\check{\pi}).
\end{equation}
By assumption, $\pi$ is supercuspidal, and so is $\check{\pi}$. Since $\omega_{\pi}=\chi^2$, $\pi$ and $\sigma$ have the same central character, and hence $\rho$ and $\check{\pi}$ also have the same central character. By the theory of the Bernstein components, $\rho$ can be decomposed as a direct sum of a representation $\tau$ whose subquotient is not isomorphic to $\check{\pi}$ and a certain number of factors which are isomorphic to $\check{\pi}$. The number of those factors equals $m(\pi)$.

Let $f$ be a matrix coefficient of $\pi$, we know
\begin{equation}\label{12.1}
\tr(\check{\pi}(f)\mid_{E_{\check{\pi}}})=f(1) d(\pi)^{-1}
\end{equation}
where $d(\pi)$ is the formal degree of $\pi$.
On the other hand, the operator $\tau(f)$ is zero because any subquotient of $\tau$ is not isomorphic to $\check{\pi}$. Since $m(\pi)\leq 1$, $\rho(f)$ is of finite rank. By \eqref{12.1}, we have
\begin{equation}\label{12.3}
\tr(\rho(f))=m(\pi)f(1)d(\pi)^{-1}.
\end{equation}

Now we are going to prove that for $N$ large, we have
\begin{equation}\label{12.2}
\tr(\rho(f))=I_{N,\omega_{\pi}}(f).
\end{equation}
Fix an open subgroup $K'\subset K$ such that $f$ is bi-$K'$-invariant and $\det(K')\subset \ker(\chi)$. Let $\Omega_N$ be the support of $\kappa_N$, $E_{\rho,N}$ be the subspace of $E_{\rho}$ consisting of functions whose support lies in $\Omega_N$, and $E_{\rho,N}^{K'}$ be the subspace consisting  of $K'$-invariant elements. Since the image of $\rho(f)$ is of finite dimension, it will be contained in $E_{\rho,N}^{K'}$ for $N$ large. Hence $\tr(\rho(f))$ is the trace of $\rho(f)\mid_{E_{\rho,N}^{K'}}$.

Fix a set $\Gamma_N$ of representatives of the double coset $H(F)U(F)\backslash \Omega_N/K'$, which is a finite set. Let $\Gamma_N '$ be the subset consisting of $\gamma\in \Gamma_N$ such that $\xi$ is trivial on $\gamma K'\gamma^{-1}\cap U(F)$. Now for $\gamma\in \Gamma_N '$, there exists a unique element $\varphi[\gamma]\in E_{\rho}$, supported on $H(F)U(F)\gamma K'$, right $K'$-invariant, such that $\varphi[\gamma](\gamma)=1$. (Note that since $\sigma$ is of one dimension, we may just let $E_{\sigma}=\BC$.) Here we use the assumption $\det(K')\subset \ker(\chi)$, which implies that $\sigma$ is always trivial on $\gamma K'\gamma^{-1}\cap H(F)$. The set $\{\varphi[\gamma]\mid \gamma\in \Gamma_N '\}$ is a basis of $E_{\rho,N}^{K'}$, and  the trace of $\rho(f)$ on $E_{\rho,N}^{K'}$ is
\begin{equation}
\Sigma_{\gamma\in \Gamma_N '} (\rho(f)\varphi[\gamma])(\gamma).
\end{equation}
Again here we use the fact that $\sigma$ is of one dimension. We have
\begin{eqnarray*}
(\rho(f)\varphi[\gamma])(\gamma)&=&\int_{Z_G(F)\backslash G(F)} \varphi[\gamma](\gamma g)f(g) dg\\
&=&\int_{Z_G(F)\backslash G(F)} \varphi[\gamma](g)f(\gamma^{-1}g) dg\\
&=&m_\gamma\int_{Z_H(F)\backslash H(F)U(F)} \sigma^{-1}(h) \xi(u) f(\gamma^{-1}hu\gamma)du dh\\
&=&m_\gamma\int_{Z_H(F)\backslash H(F)} \sigma^{-1}(h){}^{\gamma} f^{\xi}(h)dh\\
&=&m_\gamma I_{\omega_{\pi}}(f,\gamma).
\end{eqnarray*}
where $m_\gamma=meas(H(F)U(F)\backslash H(F)U(F)\gamma K')$.
For $\gamma\in \Gamma_N$ and $\gamma\notin \Gamma_N '$, since $\xi$ is non trivial on $\gamma K'\gamma^{-1}\cap U(F)$, we have $I_{\omega_{\pi}}(f,\gamma)=0$. Therefore we have
\begin{eqnarray*}
\tr(\rho(f))&=&\Sigma_{\gamma\in \Gamma_N} meas(H(F)U(F)\backslash H(F)U(F) \gamma K') I_{\omega_{\pi}}(f,\gamma)\\
&=&\int_{H(F)U(F)\backslash G(F)} I_{\omega_{\pi}}(f,g) \kappa_N(g)dg =I_{N,\omega_{\pi}}(f).
\end{eqnarray*}
This proves \eqref{12.2}.

Since $f$ is strongly cuspidal, by \cite{Ar87}, we have the following relation for quasi-characters:
\begin{equation}
\theta_f=f(1)d(\pi)^{-1}\theta_{\pi}.
\end{equation}
Together with \eqref{12.2} and Theorem \ref{main 3}, we have
$$\tr(\rho(f))=I_{N,\omega_{\pi}}(f)=I_{\omega_{\pi}}(f)=f(1)d(\pi)^{-1} m_{geom}(\pi).$$
Then by $\eqref{12.3}$, we know $m(\pi)=m_{geom}(\pi)$. This finishes the proof of the Proposition.
\end{proof}

\subsection{Proof of Theorem \ref{main}}
We are going to finish the proof of Theorem \ref{main}. We use the same notation as in Section 1. Let $G=GL_6(F)$ and $G_D=GL_3(D)$. Similarly we have $H, H_D, U$ and $U_D$. Let $\pi,\pi_D,\chi,\sigma,\sigma_D, \xi$ and $\xi_D$ be the same as Conjecture \ref{jiang}. We assume that $\pi$ is supercuspidal and we are going to prove
\begin{equation}
m(\pi)+m(\pi_D)=1.
\end{equation}

\begin{proof}
By Proposition \ref{prop 12}, we know
\begin{eqnarray*}
m(\pi)&=&c_{\theta_{\pi},\CO_{reg}}(1)+\Sigma_{v\in F^{\times}/(F^{\times})^2, v\neq 1} \mid W(H,T_v)\mid^{-1} \nu(T_v)\\
&&\times \int_{Z_H\backslash T_v(F)} \sigma^{-1}(t) c_{\pi}(t) D^H(t) \Delta(t)dt
\end{eqnarray*}
and
\begin{eqnarray*}
m(\pi_D)&=&\Sigma_{v\in F^{\times}/(F^{\times})^2, v\neq 1} \mid W(H_D,T_v)\mid^{-1} \nu(T_v)\\
&&\times \int_{Z_{H_D}\backslash T_v(F)} \sigma_{D}^{-1}(t') c_{\pi_D}(t') D^{H_D}(t') \Delta_D(t')dt.
\end{eqnarray*}
Here we use $t$ to denote elements in $\GL_6(F)$ and $t'$ to denote elements in $\GL_3(D)$. We can match $t$ and $t'$ via the characteristic polynomial: we write $t \leftrightarrow t'$ if they have the same characteristic polynomial. Since $\pi$ is supercuspidal, it is generic. So by \cite{Rod81}, we know $c_{\theta_{\pi},\CO_{reg}}(1)=1$. Also for $v\in F^{\times}/(F^{\times})^2, v\neq 1$, we have
$$\mid W(H_D,T_v)\mid=\mid W(H,T_v)\mid, Z_H=Z_{H_D}.$$
So in order to prove Theorem \ref{main}, we only need to show that for any $v\in F^{\times}/(F^{\times})^2, v\neq 1$, the sum of
$$\int_{Z_H(F)\backslash T_v(F)} \sigma^{-1}(t) c_{\pi}(t) D^H(t) \Delta(t) d_c t$$
and
$$\int_{Z_H(F)\backslash T_v(F)} \sigma_{D}^{-1}(t')c_{\pi_D}(t') D^{H_D}(t') \Delta_D(t') d_c t=0$$
is zero. Because for $t,t'\in T_v(F)$ regular with $t\leftrightarrow t'$, we have
$$D^H(t)=D^{H_D}(t), \Delta(t)=\Delta_D(t'), \sigma(t)=\sigma_D(t').$$
It is enough to show that for any $v\in F^{\times}/(F^{\times})^2, v\neq 1$, and for any $t,t'\in T_v(F)$ regular with $t\leftrightarrow t'$, we have
\begin{equation}\label{6.4}
c_{\pi}(t)+c_{\pi_D}(t')=0.
\end{equation}

In fact, by Section 13.6 of \cite{W10} or Proposition 4.5.1 of \cite{B15}, we know
$$c_{\pi}(t)=D^G(t)^{-1/2} \mid W(G_t, T_{qs,t} \mid^{-1} \lim_{x\in T_{qs,t}(F)\rightarrow t} D^G(x)^{1/2} \theta_{\pi}(x)$$
and
$$c_{\pi_D}(t')=D^{G_D}(t)^{-1/2} \mid W((G_D)_{t'}, T_{qs,t'} \mid^{-1} \lim_{x'\in T_{qs,t'}(F)\rightarrow t} D^{G_D}(x')^{1/2} \theta_{\pi_D}(x')$$
where $T_{qs,t}$ (resp. $T_{qs,t'}$) is a maximal torus contained in the Borel subgroup $B_t$ (resp. $B_{t'}$) of $G_t$ (resp. $(G_D)_{t'}$). Note that if $t,t'\in T_v$ is regular, both $G_t$ and $(G_D)_{t'}$ are isomorphic to $GL_3(F_v)$ which is quasi-split over $F$. We are able to choose the Borel subgroup $B_t$ (resp. $B_{t'}$). In particular,
$$\mid W(G_t, T_{qs,t}) \mid^{-1}=\mid W((G_D)_t, T_{qs,t}) \mid^{-1}.$$
Also for those $t\leftrightarrow t'$, we have
$$D^G(t)=D^{G_D}(t).$$
And for $x\in T_{qs,t}(F)$ (resp. $x'\in T_{qs,t'}(F)$) sufficiently close to $t$ (resp. $t'$) with $x\leftrightarrow x'$, they are also regular and we have
$$D^G(x)=D^{G_D}(x').$$
Therefore in order to prove \eqref{6.4}, it is enough to show that for any regular $x\in G(F)$ and $x'\in G_D(F)$ with $x\leftrightarrow x'$, we have
\begin{equation}
\theta_{\pi}(x)+\theta_{\pi_D}(x')=0.
\end{equation}
This just follows from the relations of the distribution characters under the local Jacquet-Langlands correspondence as given in \cite{DKV84}. This proves Theorem \ref{main}
\end{proof}

\textbf{In summary, we have proved Theorem \ref{main} based on Theorem \ref{main 1}. In the following few sections, we are going to prove Theorem \ref{main 1}.}

\section{Localization}
In this section, we fix a strongly cuspidal function $f\in \cg$. Our goal is to localize both sides of the trace formula in \eqref{equation 3} (i.e $I_N(f)$ and $I(f)$), which enables us to reduce the proof of the trace formula to the Lie algebra level.
\subsection{A Trivial Case}
If $x\in G_{ss}(F)$ that is not conjugate to an element in $H(F)$, then we can easily find a good neighborhood $\omega$ of $0$ in $\Fg_x(F)$ small enough such that $x \exp(X)$ is not conjugate to an element in $H(F)$ for any $X\in \omega$. Let $\Omega=Z_G(F) \cdot (x\exp(\omega))^G$. It follows that $\Omega\cap H(F)=\emptyset$. Suppose that $f$ is supported on $\Omega$. For every $t\in H_{ss}(F)$, the complement of $\Omega$ in $G(F)$ is an open neighborhood of $t$ invariant under conjugation, and
is away from the support of $f$. It follows that $\theta_f$ also vanishes on an open neighborhood of $t$, and hence that $I(f)=0$.
On the other hand, the semisimple part of elements in $U(F)H(F)$ belongs to $H(F)$. Thus $\gf=0$ for every $g\in G(F)$, and so $I_N(f)=0$. Therefore Theorem \ref{main 1} holds for $f$.

\subsection{Localization of $I_N(f)$}
For $x\in H_{ss}(F)$, let $U_x=U\cap G_x$, fix a good neighborhood $\omega$ of 0 in $\Fg_{x}(F)$, and let $\Omega=(x\exp(\omega))^G \cdot Z_G(F)$. We can decompose $\Fg_{x,0}$ and $\Fh_{x,0}$ into $\Fg_{x,0}=\Fg_{x}'\oplus \Fg''$ and $\Fh_{x,0}=\Fh_{x}'\oplus \Fh''$, where $\Fg_{x}'=\Fh_{x}'$ is the common center of $\Fg_{x,0}$ and $\Fh_{x,0}$, $\Fg''$ and $\Fh''$ are the semisimple parts. To be specific, the decomposition is given as follows: (Recall for any Lie algebra $\Fp$, we define $\Fp_0$ to be the subalgebra consists of elements in $\Fp$ with zero trace.)
\begin{itemize}
\item If $x$ is contained in the center, then $G_x=G,H_x=H$. Define
$$\Fg_x '=\Fh_x '=0, \Fg''=\Fg_{x,0}, \Fh''=\Fh_{x,0},$$
\item If $x$ is split but not contained in the center, then $G_x=\GL_3(F)\times \GL_3(F), H_x=\GL_1(F)\times \GL_1(F)$. Define
\begin{eqnarray*}
\Fg_x'&=&\Fh_x '=\{\diag(a,-a,a,-a,a,-a)\mid a\in F\}, \\
\Fg''&=&\mathfrak{sl}_3(F)\oplus \mathfrak{sl}_3(F), \\
\Fh''&=&0.
\end{eqnarray*}
\item If $x$ is not split, then it is conjugate to a regular element in the torus $T_v$ for some $v\in F^{\times}/(F^{\times})^2, \; v\neq 1$. Recall $T_v$ is the non-split torus of $H(F)$ that is F-isomorphic to $F_v=F(\sqrt{v})$. In this case, $G_x=\GL_3(F_v),H_x=GL_1(F_v)$.
    Define
\begin{eqnarray*}
\Fg_x'&=&\Fh_x '=\{\diag(a,a,a)\mid a\in F_v, \tr(a)=0\}, \\
\Fg''&=&\mathfrak{sl}_3(F_v), \\
\Fh''&=&0.
\end{eqnarray*}
\end{itemize}
Then for every torus $T\in T(G_x)$ (here $T(G_x)$ stands for the set of maximal tori in $G_x$), we can write $\Ft_0=\Ft'\oplus \Ft''$ with $\Ft'=\Fg_x '=\Fh_x '$. The idea of the decomposition above is that $\Fg_x'=\Fh_x '$ is the extra center in $\Fg_x$, and $(\Fg'',\Fh''\oplus \Fu_x)$ stands for the reduced model after localization. In fact, if $x$ is in the center, it is just the Ginzburg-Rallis model; when $x$ is not in the center, it is the Whittaker model.
\\
\\
\textbf{From now on, we choose the function $f$ such that $Supp(f)\subset \Omega$.}
\begin{defn}
Define a function ${}^g f_{x,\omega}$ on $\Fg_{x,0}(F)$ by
\begin{equation}\label{8.13}
{}^g f_{x,\omega}(X)
=\begin{cases}
f(g^{-1}x\exp(X)g), & \mathrm{if}\  X \in \omega; \\
0, & \mathrm{otherwise}.
\end{cases}
\end{equation}
Here we still view $\omega$ as a subset of $\Fg_{x,0}$ via the projection $\Fg_x\rightarrow \Fg_{x,0}$. We define
\begin{equation}
\gf_{x,\omega}(X) =\int_{\Fu_x(F)} {}^g f_{x,\omega}(X+N) \xi(N) dN,
\end{equation}
\begin{equation}
I_{x,\omega}(f,g)=\int_{\Fh_{x,0}(F)} \gf_{x,\omega}(X) dX,
\end{equation}
\begin{equation}\label{8.1}
I_{x,\omega,N}(f)=\int_{U_x(F)H_x(F)\backslash G(F)} I_{x,\omega}(f,g) \kappa_N(g) dg.
\end{equation}
\end{defn}

\begin{rmk}
The function $g\rightarrow I_{x,\omega}(f,g)$ is left $U_x(F)H_x(F)$-invariant. By Condition $(5)$ of good neighborhood (as in Definition \ref{good nbd defn}), there exists a subset $\Gamma\subset G(F)$, compact modulo center, such that ${}^g f_{x,\omega}(X)\equiv 0$ for $g\notin G_x(F)\Gamma$. Together with the fact that the function $g\rightarrow \kappa_N(g\gamma)$ on $G_x(F)$ has compact support modulo $U_x(F)H_x(F)$ for all $\gamma\in G(F)$, we know that the integrand in \eqref{8.1} is compactly supported, and therefore, the integral is absolutely convergent.
\end{rmk}

\begin{prop}\label{localization}
$I_N(f)=C(x)I_{x,\omega,N}(f)$ where $C(x)=D^H(x)\Delta(x)$.
\end{prop}

\begin{proof}
By the Weyl Integration Formula, we have
\begin{equation}\label{8.11}
I(f,g)=\Sigma_{T\in T(H)} \mid W(H,T)\mid^{-1} \int_{Z_H(F)\backslash T(F)} J_H(t,\gf) D^H(t)^{1/2} dt
\end{equation}
where
$$J_H(t,F)=D^H(t)^{1/2} \int_{H_t(F)\backslash H(F)} F(g^{-1}tg) dg$$
is the orbital integral.
For given $T\in T(H)$ and $t\in T(F)\cap H_{reg}(F)$, we need the following lemma, the proof the lemma will be given after the proof this  proposition.
\begin{lem}\label{8}
\begin{enumerate}
For $t\in T(F)$, the following hold.
\item If $t$ does not belong to the following set
 $$
 \cup_{T_1\in T(H_x)} \cup_{w\in W(T_1,T)} w(x\exp(\Ft_1(F)\cap \omega))w^{-1}\cdot Z_G(F),
 $$
 then $J_H(t,\gf)=0$. Here $W(T_1,T)$ is the set of isomorphisms between $T$ and $T_1$ induced by conjugation by elements in $H(F)$, i.e. $W(T_1,T)=T\backslash \{h\in H(F)| hT_1h^{-1}=T\}/T_1$.
\item If $x$ is not contained in the center, each components in $(1)$ are disjoint. If $x$ is contained in the center, two components in $(1)$ either are disjoint or coincide. They coincide if and only if $T=T_1$ in $T(H)$. Therefore, for each component $(T_1,w)$, the number of components which coincide with it (include itself) is equal to $W(H_x,T_1)$.
\end{enumerate}
\end{lem}

By the lemma above, we can rewrite the expression \eqref{8.11} of $I(f,g)$ as
$$I(f,g)=\Sigma_{T_1\in T(H_x)} \Sigma_{T\in T(H)} \Sigma_{w_1\in W(T_1,T)} | W(H,T)|^{-1} | W(H_x,T_1)|^{-1}$$
$$\times \int_{\Ft_{1,0}(F)\cap \omega} J_H(w_1(w\exp(X))w_{1}^{-1}, \gf) D^H(w_1(w\exp(X))w_{1}^{-1})^{1/2} dX.$$
Note that both integrands above are invariant under $H(F)$-conjugate, $W(T_1,T)\neq \emptyset$ if and only if $T=T_1$ in $T(H)$, and in that case $W(T,T_1)=W(H,T)$. We have
\begin{equation}\label{8.2}
I(f,g)=\Sigma_{T_1\in T(H_x)} \mid W(H_x,T_1)\mid^{-1}  \int_{\Ft_{1,0}(F)\cap \omega} J_H(x\exp(X), \gf) D^H(x\exp(X))^{1/2} dX.
\end{equation}
On the other hand, by Parts (3) and (5) of Proposition \ref{good nbd},
for all $T_1\in T(H_x)$ and for all $X\in \omega \cap \Ft_{1,0,reg}(F)$, we have
\begin{eqnarray}\label{8.3}
J_H(x\exp(X),\gf)
&=&D^H(x\exp(X))^{1/2} \nonumber\\
&&\times\int_{H_x(F)\backslash H(F)}
\int_{T_1(F)\backslash H_x(F)} {}^{yg}f^{\xi}(x\exp(h^{-1}Xh)) dh dy,
\end{eqnarray}
and
\begin{equation}\label{8.4}
D^H(x\exp(X))=D^H(x)\cdot D^{H_x}(X).
\end{equation}
So combining \eqref{8.2}, \eqref{8.3}, \eqref{8.4}, together with the definition of $I_{N}(f)$ (as in \eqref{spectral 0}), we have
\begin{eqnarray}\label{8.9}
I_N(f)&=&\int_{U(F)H(F)\backslash G(F)} \Sigma_{T_1\in T(H_x)} \mid W(H_x,T_1)\mid^{-1} \nonumber\\
&&\times\int_{t_{1,0}(F)\cap \omega}
J_H(x\exp(X),\gf)D^H(x\exp(X))^{1/2} dX \kappa_N(g) dg\\
&=&D^H(x)\int_{U(F)H_x(F)\backslash G(F)} \Phi(g) \kappa_N(g) dg\nonumber
\end{eqnarray}
where
\begin{eqnarray*}
\Phi(g)&=&\Sigma_{T_1\in T(H_x)} \mid W(H_x,T_1)\mid^{-1}\\
&&\times\int_{\Ft_{1,0}(F)\cap \omega} \int_{T_1(F)\backslash H_x(F)} \gf(x\exp(h^{-1}Xh)) dh D^{H_x}(X) dX.
\end{eqnarray*}
Applying the Weyl Integration Formula to $\Phi(g)$, we have
\begin{equation}\label{8.8}
\Phi(g)=\int_{\Fh_{x,0}(F)} \varphi_g(X)dX
\end{equation}
where
\begin{equation}\label{8.7}
\varphi_g(X)
=\begin{cases}
\gf(x\exp(X')), & \text{if}\  X=X'+Z, X'\in \omega, Z\in \Fz_{\Fh}(F); \\
0, & \text{otherwise}.
\end{cases}
\end{equation}
On the other hand, for $X\in \omega\cap \Fh_{x,reg}(F), g\in G(F)$,
\begin{eqnarray}\label{8.12}
\gf(x\exp(X))&=&\int_{U(F)} {}^g f(x\exp(X)u)\xi(u)du\nonumber\\
&=& \int_{U_x(F)\backslash U(F)}\int_{U_x(F)} {}^g f(x\exp(X)uv)\xi(uv)du dv.
\end{eqnarray}

For $u\in U_x(F)$, the map $v\rightarrow (x\exp(X)u)^{-1} v^{-1} (x\exp(X)u) v$ is a bijection of $U_x(F)\backslash U(F)$. By the Condition $(7)_{\rho}$ of good neighborhood (as in Definition \ref{good nbd defn}), the Jacobian of this map is
$$\mid \det((1-ad(x)^{-1})\mid_{U(F)/U_x(F)}) \mid_F=\Delta(x).$$
Also it is easy to see that
$$\xi((x\exp(X)u)^{-1} v^{-1} (x\exp(X)u) v)=1.$$
By making the transform $v\rightarrow (x\exp(X)u)^{-1} v^{-1} (x\exp(X)u) v$ in \eqref{8.12}, we have
\begin{eqnarray}\label{8.5}
\gf(x\exp(X))&=&\Delta(x)\int_{U_x(F)\backslash U(F)} \int_{U_x(F)} {}^g f(v^{-1}x \exp(X) uv)\xi(u) du dv\nonumber\\
&=&\Delta(x)\int_{U_x(F)\backslash U(F)} \int_{U_x(F)} {}^{vg}f(x\exp(X)u) \xi(u) du dv.
\end{eqnarray}
By Condition (6) of good neighborhood (as in Definition \ref{good nbd defn}), for all $X\in \omega$, the map
$\Fu_x(F)\rightarrow U_x(F)$ given by
$$N\mapsto \exp(-X) \exp(X+N)$$
is a bijection and preserves the measure. Also we have
$$\xi(\exp(-X) \exp(X+N))=\xi(N).$$
So we can rewrite \eqref{8.5} as
$$\gf(x\exp(X))=\Delta(x) \int_{U_x(F)\backslash U(F)} \int_{\Fu_x(F)} {}^{vg} f(x\exp(X+N)) \xi(N) dN dv.$$
For $X\in \omega_{reg}$, $X+N$ can be conjugated to $X$ by an element in $G_x(F)$, so $X+N\in \omega$, and ${}^{vg} f(x\exp(X+N))={}^{vg}f_{x,\omega}(X+N)$ by the definition of ${}^g f_{x,\omega}$ (as in \eqref{8.13}). This implies that
\begin{equation}\label{8.6}
\gf(x\exp(X))=\Delta(x)\int_{U_x(F)\backslash U(F)} {}^{v}\gf_{x,\omega}(X) dv.
\end{equation}

Now, combining \eqref{8.6} and \eqref{8.7}, we have
$$\varphi_g(X)=\Delta(x)\int_{U_x(F)\backslash U(F)} {}^{v}\gf_{x,\omega}(X') dv.$$
Then combining the above equation with \eqref{8.8} and changing the order of integration, we have
\begin{equation}
\Phi(g)=\Delta(x)\int_{U_x(F)\backslash U(F)} I_{x,\omega}(f,vg)dv.
\end{equation}
Finally combining the above equation with \eqref{8.9} and using the fact that $C(x)=\Delta(x)D^{H}(x)$, we have
$$I_N(f)=C(x)\int_{U_x(F)H_x(F)\backslash G(F)} I_{x,\omega}(f,g)\kappa_N(g) dg=C(x)I_{x,\omega,N}(f).$$
This finishes the proof of the Proposition.
\end{proof}

Now we prove Lemma \ref{8}.
\begin{proof}
If $J_H(t,\gf)\neq 0$, there exists $u\in U(F)$ such that $tu$ is conjugate to an element in $Supp(f)$. If we only consider the semisimple part, since we assume that $Supp(f)\subset \Omega=Z_G(F)\cdot (x\exp(\omega))^G$, there exist $y\in G(F), X\in \omega$ and $z\in Z_G(F)$, such that $yty^{-1}=x\exp(X)z$. By changing $t$ to $tz$, we may assume that $z=1$. Then by conjugating $X$ by an element $y'\in G_x(F)$ and changing $y$ to $y'y$, we may assume that $X\in \Ft_1(F)$ for some $T_1\in T(G_x)$.

If $x$ is in the center, we have that $G_x=G$. Since $t\in H$, by changing $y$ we may assume that $X\in \Fh \cap \Fg_x=\Fh_x$. By further conjugating by an
element in $H_x(F)$, we can just assume that $X\in \Ft_1(F)$ for some $T_1\in T(H_x)$. If x is split but not contained in the center, then $G_x=\GL_3(F)\times \GL_3(F)$. Assume that the eigenvalues of $x$ are $\lambda,\lambda,\lambda,\mu,\mu,\mu$ for some $\lambda,\mu\in F,\lambda\neq \mu$. Note that for $t\in H$, its eigenvalues are of the same form, but may lie in a quadratic extension of $F$. Now if $\omega$ is small enough with respect to $\mu-\lambda$, the eigenvalues of the given $X\in \omega$ must have the same form. It follows that $X\in \Fh(F)$, and $X\in \Fh(F)\cap \Fg_x(F)=\Fh_x(F)$. After a further conjugation by an element in $H_x(F)$, we can still assume that $X\in \Ft_1(F)$ for some $T_1\in T(H_x)$. By applying the same argument, when $x$ is non-split, we can still assume that $X\in \Ft_1(F)$ for some $T_1\in T(H_x)$.

By the above discussion, we can always assume that $X\in \Ft_1(F)$ for some $T_1\in T(H_x)$. Since the Weyl group of $G$ with respect to $T$ equals the Weyl group of $H$ with respect to $T$, any $G(F)-$conjugation between $T$ and $T_1$ can be realized by an element in $H(F)$. Here we define the Weyl group of $T$ in $G$ to be the quotient of the normalizer of $T$ in $G$ with the centralizer of $T$ in $G$. Moreover, if such a conjugation exists, $T=T_1$ in $T(H)$ and the conjugation is given by the Weyl element $w\in W(T,T_1)$. This finishes the proof of Part (1).

Part (2) is very easy to verify. If $x$ is not in the center, let $\lambda$ and $\mu$ be the eigenvalues of $x$. Then $\lambda\neq \mu$, where $\lambda$ and $\mu$ may lie in a quadratic extension of $F$. Once we choose $\omega$ small enough with respect to $\lambda-\mu$, it is easy to see that each components in (1) are disjoint. If $x$ is in the center, by the proof of part (1), the components corresponding to $T$ does not intersect with other components. Since the Weyl group $W(T_1,T)\simeq W(H,T)$ is of order 2, there are two components corresponding to $T$, and these two components coincide because $\omega$ is $G=G_x$-invariant in this case. This finishes the proof of (2).
\end{proof}

\subsection{Localization of $I(f)$}
We slightly modify the notation of Section 5.1: If $x\in Z_H(F)$, then $H_x=H$. In this case, we let $\CT_x=\CT$. (Recall that $\CT$ is a subset of subtorus of H defined in Section 5.1.) If $x\notin Z_H(F)$, $H_x$ is either $\GL_1(F)\times \GL_1(F)$ or $\GL_1(F_v)$ for some $v\in F^{\times}/(F^{\times})^2, v\neq 1$. In both cases, let $\CT_x$ be the subset of $\CT$ consisting of those nontrivial subtorus $T\in \CT$ such that $T\in H_x$, i.e. if $H_x=\GL_1(F)\times \GL_1(F)$, $\CT_x$ is empty; and if $H_x=GL_1(F_v)$, $\CT_x=\{T_v\}$. Now for $T\in \CT_x$, we define the function $c_{f,x,\omega}$ on $\Ft(F)$ as follows: It is zero for elements not contained in $\Ft(F)\cap(\omega+\Fz_{\Fg}(F))$. For $X=X'+Y\in \Ft(F)$ with $X'\in \omega, Y\in \Fz_{\Fg}(F)$, define
\begin{equation}
c_{f,x,\omega}(X)=c_f(x\exp(X')).
\end{equation}
In fact, the function $\theta_{f,x,\omega}$ defined in \eqref{6.3} is a quasi-character in $\Fg_x$, and the function $c_{f,x,\omega}$ we defined above is the germ associated to this quasi-character. Now we define the function $\Delta''$ on $\Fh_x(F)$ to be
\begin{equation}
\Delta''(X)=|\det(ad(X)\mid_{\Fu_x(F)/(\Fu_x(F))_X})|_F.
\end{equation}
By Condition $(7)_{\rho}$ of Definition \ref{good nbd defn}, we know that for every $X\in \omega$,
\begin{equation}\label{delta 2}
\Delta(x\exp(X))=\Delta(x)\Delta''(X).
\end{equation}
Let
\begin{equation}\label{I_x}
I_{x,\omega}(f)=\Sigma_{T\in \CT_x} |W(H_x,T)|^{-1} \nu(T) \int_{\Ft_0(F)} c_{f,x,\omega}(X) D^{H_x}(X) \Delta''(X)dX.
\end{equation}
By Proposition \ref{integrable 1}, the integral above is absolutely convergent.
\begin{prop}\label{localization 1}
With the notations above, we have
\begin{equation}\label{7.2}
I(f)=C(x)I_{x,\omega}(f).
\end{equation}
\end{prop}

\begin{proof}
By applying the same argument as Lemma \ref{8}, we have the following properties for the function $c_f(t)$:
\begin{enumerate}
\item If $T\in \CT$, and $t\in T(F)$, then $c_f(t)=0$ if
$$t\notin \cup_{T_1\in \CT_x} \cup_{w\in W(T_1,T)} w(x\exp(\Ft_1(F)\cup \omega))w^{-1}\cdot Z_G(F).$$
\item If $x$ is not contained in the center, each components in $(1)$ are disjoint. If $x$ is contained in the center, two components in $(1)$ either are disjoint or coincide. They coincide if and only if $T=T_1$ in $T(H)$. Therefore, for each component $(T_1,w)$, the number of components which coincide with it (include itself) is equal to $W(H_x,T_1)$.
\end{enumerate}
So we can rewrite the expression \eqref{geometry 1} of $I(f)$ as
\begin{equation}\label{8.10}
I(f)=\Sigma_{T_1\in \CT_x}\Sigma_{T\in \CT}\Sigma_{w_1\in W(T_1,T)}  \mid W(H,T)\mid^{-1} \mid W(H_x,T)\mid^{-1}\nu(T)
\end{equation}
$$\times \int_{\Ft_{1,0}\cap \omega} c_f(w_1(x\exp(X))w_{1}^{-1}) D^H(w_1(x\exp(X))w_{1}^{-1}) \Delta(x\exp(X))dX.$$
Since every integrand in \eqref{8.10} is invariant under $H(F)$-conjugation, together with Proposition \ref{good nbd}(5) and \eqref{delta 2}, we have
$$D^H(x\exp(X))\Delta(x\exp(X))=D^H(x)D^{H_x}(X) \Delta(x)\Delta''(X).$$
Then \eqref{8.10} becomes
\begin{eqnarray*}
I(f)&=&D^H(x)\Delta(x)\Sigma_{T_1\in \CT_x}  \nu(T_1)\mid W(H_x,T)\mid^{-1}\\
&&\times\int_{\Ft_{1,0}(F)} c_{f,x,\omega}(X) D^{H_x}(X)\Delta''(X) dX\\
&=&C(x)I_{x,\omega}(f).
\end{eqnarray*}
This finishes the proof of the Proposition.
\end{proof}

\section{Integral Transfer}
\subsection{The Problem}
In this section, let $(G',H',U')$ be one of the following:
\begin{enumerate}
\item $G'=\GL_6(F), H'=\GL_2(F)$, $U'$ is the unipotent radical of the parabolic subgroup whose Levi is $\GL_2(F)\times \GL_2(F)\times \GL_2(F)$.
\item $G'=\GL_3(D), H'=\GL_1(D)$, $U'$ is the unipotent radical of the parabolic subgroup whose Levi is $\GL_1(D)\times \GL_1(D)\times \GL_1(D)$.
\item $G'=\GL_3(F)\times \GL_3(F), H'=\GL_1(F)\times \GL_1(F)$, $U'$ is the unipotent radical of the Borel subgroup.
\item $G'=\GL_3(F_v), H'=\GL_1(F_v)$, for some $v\in F^{\times}/(F^{\times})^2$ with $\;v\neq 1$, $U'$ is the unipotent radical of the Borel subgroup.
\end{enumerate}
This basically means that $(G',H',U')$ is of the form $(G_x,H_x,U_x)$ for some $x\in H_{ss}(F)$. Our goal is to simplify the integral $I_{x,\omega,N}(f)$ defined in last section. To be specific, in the definition of $I_{x,\omega,N}(f)$, we first integrate over the Lie algebra of $H_xU_x$, then integrate over $U_xH_x\backslash G_x$. In this section, we are going to transfer this integral into the form $\int_{\Ft^0(F)} \int_{A_T(F)\backslash G(F)}$ where $T$ runs over maximal torus in $G_x$ and $\Ft^0(F)$ is a subset of $\Ft(F)$ which will be defined later. The reason for doing this is that we want to apply Arthur's local trace formula which is of the form $\int_{A_T(F)\backslash G(F)}$. Our method is to study the orbit of the slice representation, we will only write down the proof for the first two situations. The proof for the last two situations follows from the same, but easier arguments, and hence we will skip the proof here.
So we will still use $(G,H,U)$ instead of $(G',H',U')$ in this section. We fix a truncated function $\kappa\in C_{c}^{\infty}(U(F)H(F)\backslash G(F))$, and a function $f\in C_{c}^{\infty}(\Fg_0(F))$. Recall in Section 5.3, we have defined
$$f^{\xi}(Y)=\int_{\Fu(F)} f(Y+N)\xi(N) dN$$
and
$$I(f,g)=\int_{\Fh_0(F)} {}^g f^{\xi}(Y) dY.$$
Let
\begin{equation}
I_{\kappa}(f)=\int_{U(F)H(F)\backslash G(F)} I(f,g) \kappa(g)dg.
\end{equation}
We are going to study $I_{\kappa}(f)$.

\subsection{Premier Transform}
For $\Xi=\begin{pmatrix} 0 & 0 & 0 \\ aI_2 & 0 & 0 \\ 0 & bI_2 & 0  \end{pmatrix}$, we have that $\xi(N)=\psi(<\Xi, N>)$ for $N\in \Fu(F)$. Here we use $I_2$ to denote the identity element in $\Fh(F)$, i.e. in split case, $I_2$ is the identity two by two matrix; and in nonsplit case, $I_2$ is $1$ in the quaternion algebra. Define
$$\Lambda_0=\{\begin{pmatrix} A & 0 & 0 \\ 0 & B & 0 \\ 0 & 0 & C \end{pmatrix} \mid A+B+C=0\}$$
and
$$\Sigma=\Lambda_0+\Fu.$$

\begin{lem}\label{Fourier}
For any $f\in C_{c}^{\infty}(\Fg_0(F))$ and $Y\in \Fh_0(F)$, we have
$$(f^{\xi})^{\hat{}}(Y)=\int_{\Sigma} \hat{f}(\Xi+Y+X) dX.$$
\end{lem}

\begin{proof}
Since $\Fg=\bar{\Fu}\oplus \Fh_0\oplus \Lambda_0\oplus \Fu$, we may assume that $f=f_{\bar{\Fu}}\otimes f_{\Fh_0}\otimes f_{\Lambda_0}\otimes f_{\Fu}$. Then we have
\begin{eqnarray*}
\hat{f}&=&\hat{f}_{\bar{\Fu}}\otimes \hat{f_{\Fh_0}}\otimes \hat{f_{\Lambda_0}}\otimes \hat{f_{\Fu}},\\
f^{\xi}(Y)&=&f_{\bar{\Fu}}(0)\otimes f_{\Fh_0}(Y)\otimes f_{\Lambda_0}(0)\otimes \hat{f_{\Fu}}(\Xi), \\
(f^{\xi})^{\hat{}}(Y)&=&f_{\bar{\Fu}}(0)\otimes \hat{f_{\Fh_0}}(Y)\otimes f_{\Lambda_0}(0)\otimes \hat{f_{\Fu}}(\Xi).
\end{eqnarray*}
On the other hand,
\begin{eqnarray*}
\int_{\Sigma} \hat{f}(\Xi+Y+X)dX
&=&\hat{f_{\Fu}}(\Xi)\hat{f_{\Fh_0}}(Y) \int_{\Sigma} \hat{f_{\Lambda_0}} \otimes \hat{f}_{\bar{\Fu}}(X) dX\\
&=&f_{\bar{\Fu}}(0)\otimes \hat{f_{\Fh_0}}(Y)\otimes f_{\Lambda_0}(0)\otimes \hat{f_{\Fu}}(\Xi).
\end{eqnarray*}
This finishes the proof of the Lemma.
\end{proof}

\subsection{Description of affine space $\Xi+\Sigma$}
Let $\Lambda=\{\begin{pmatrix} 0 & 0 & \ast \\ 0 & 0 & \ast \\ 0 & 0 & 0 \end{pmatrix}\}$ be a subset of $\Fu(F)$.
\begin{lem}\label{unipotent orbit}
$\Xi+\Sigma$ is stable under the $U(F)$-conjugation. The map
\begin{equation}\label{9.1}
U(F)\times (\Xi+\Lambda)\rightarrow \Xi+\Sigma:\; (u,x)\mapsto u^{-1}Xu
\end{equation}
is an isomorphism of algebraic varieties.
\end{lem}

\begin{proof}
We have the following two equations
$$\begin{pmatrix} I_2 & X & Z \\ 0 & I_2 & Y \\ 0 & 0 & I_2 \end{pmatrix} \begin{pmatrix} 0 & 0 & 0 \\ aI_2 & 0 & 0 \\ 0 & bI_2 & 0 \end{pmatrix} \begin{pmatrix} I_2 & -X & XY-Z \\ 0 & I_2 & -Y \\ 0 & 0 & I_2 \end{pmatrix}$$
$$=\begin{pmatrix} aX & bZ-X^2 & aX^2 Y-aXZ-bYZ \\ aI_2 & bY-aX & aXY-aZ-bY^2 \\ 0 & bI_2 & -bY \end{pmatrix},$$
and
$$\begin{pmatrix} I_2 & X & Z \\ 0 & I_2 & Y \\ 0 & 0 & I_2 \end{pmatrix} \begin{pmatrix} 0 & 0 & B \\ 0 & 0 & C \\ 0 & 0 & 0 \end{pmatrix} \begin{pmatrix} I_2 & -X & XY-Z \\ 0 & I_2 & -Y \\ 0 & 0 & I_2 \end{pmatrix}=\begin{pmatrix} 0 & 0 & B+XC \\ 0 & 0 & C \\ 0 & 0 & 0 \end{pmatrix}.$$
Then the map \eqref{9.1} is clearly injective. On the other hand, for any element in $\Xi+\Sigma$, applying the first equation above, we can choose $X$ and $Y$ to match the elements in the diagonal. Then applying the second equation, we can choose $Z$ to match the element in the first row second column. Finally applying second equation again, we can choose $B$ and $C$ to match the element in the first row third column and second row third column. Therefore the map \eqref{9.1} is surjective.

Now we have proved the map \eqref{9.1} is a bijection of points. In order to show it is an isomorphism of algebraic varieties, we only need to find the inverse map. Let $\begin{pmatrix} A' & T_1 & T_2 \\ aI_2 & B' & T_3 \\ 0 & bI_2 & C' \end{pmatrix}$ be an element in $\Xi+\Sigma$. Set
\begin{equation}\label{9.2}
X=\frac{1}{a}A',\; Y=-\frac{1}{b}C',\; Z=\frac{T_1+X^2}{b},
\end{equation}
$$C=T_3-aXY+aZ+bY^2,\;B=T_2-aX^2Y+aXZ+bYZ-XC,$$
then by the two equations above, we have
$$\begin{pmatrix} I_2 & X & Z \\ 0 & I_2 & Y \\ 0 & 0 & I_2 \end{pmatrix} \begin{pmatrix} 0 & 0 & B \\ aI_2 & 0 & C \\ 0 & bI_2 & 0 \end{pmatrix} \begin{pmatrix} I_2 & -X & XY-Z \\ 0 & I_2 & -Y \\ 0 & 0 & I_2 \end{pmatrix}=\begin{pmatrix} A' & T_1 & T_2 \\ aI_2 & B' & T_3 \\ 0 & bI_2 & C' \end{pmatrix}.$$
Therefore the map \eqref{9.2} is the inverse map of \eqref{9.1}, also it is clearly algebraic. This finishes the proof of the Lemma.
\end{proof}

\begin{defn}
We say an element $W\in\Xi+\Sigma$ is in "generic position" if it satisfies the following two conditions:
\begin{enumerate}
\item $W$ is semisimple regular.
\item $W$ is conjugated to an element $\begin{pmatrix} 0 & 0 & X \\ aI_2 & 0 & Y \\ 0 & bI_2 & 0 \end{pmatrix} \in \Sigma+\Lambda$ such that $X,Y$ are semisimple regular and $XY-YX$ is not nilpotent. In particular, this implies $H_X\cap H_Y=Z_H$.
\end{enumerate}
Let $\Xi+\Sigma^0$ be the subset of $\Xi+\Sigma$ consisting of elements in "generic position". It is a Zariski open subset of $\Xi+\Sigma$. Let $\Xi+\Lambda^0=(\Xi+\Sigma^0)\cap (\Xi+\Lambda)$.
\end{defn}

\subsection{Orbit in $\Xi+\Lambda^0$}
\begin{lem}\label{orbit 1}
The group $Z_G(F)\backslash H(F)U(F)$ acts by conjugation on $\Xi+\Sigma^0$, and this action is free. Two elements in $\Xi+\Sigma^0$ are conjugated to each other in $G(F)$ if and only if they are conjugated to each other by an element in $H(F)U(F)$.
\end{lem}

\begin{proof}
For the first part, by Lemma \ref{unipotent orbit}, we only need to show that the action of $Z_G(F)\backslash H(F)$ on $\Xi+\Lambda^0$ is free. This just follows from the "generic position" assumption.

For the second part, given $x,y\in \Xi+\Sigma^0$ which are conjugated to each other by an element in $G(F)$. By conjugating both elements by some elements in $U(F)$, may assume $x,y\in \Xi+\Lambda^0$. Let
$$x=\begin{pmatrix} 0 & 0 & X_1 \\ aI_2 & 0 & X_2 \\ 0 & bI_2 & 0 \end{pmatrix},\; y=\begin{pmatrix} 0 & 0 & Y_1 \\ aI_2 & 0 & Y_2 \\ 0 & bI_2 & 0 \end{pmatrix}.$$
We only need to find $h\in H(F)$ such that $h^{-1}X_i h=Y_i$ for $i=1,2$. The characteristic polynomial of $x$ is
$$
\det(x-\lambda I_6)
=\det(\begin{pmatrix} -\lambda I_2 & 0 & X_1 \\ aI_2 & -\lambda I_2 & X_2 \\ 0 & bI_2 & -\lambda I_2 \end{pmatrix}),
$$
which can be calculated as follows:
\begin{eqnarray*}
\det(x-\lambda I_6)
&=&\det(\begin{pmatrix} 0 & -\lambda^2/a I_2 & X_1+\lambda/a X_2 \\ aI_2 & -\lambda I_2 & X_2 \\ 0 & bI_2 & -\lambda I_2 \end{pmatrix})\\
&=&a^2 \cdot \det(\begin{pmatrix} -\lambda^2/a I_2 & X_1+\lambda/a X_2 \\ bI_2 & -\lambda I_2 \end{pmatrix})\\
&=&a^2 \cdot \det(\begin{pmatrix} 0 & X_1+\lambda/a X_2-\frac{\lambda^3}{ab} I_2 \\ bI_2 & -\lambda I_2 \end{pmatrix}).
\end{eqnarray*}
Hence we have
$$
\det(x-\lambda I_6)
=a^2 b^2 \det(X_1+\lambda/a X_2-\frac{\lambda^3}{ab} I_2).
$$
Therefore, up to some constants only depending on $a$ and $b$, the coefficients of the characteristic polynomial of $x$ correspond to some data of $X_1,X_2$ given as follows:
\begin{equation}\label{lambda 4}
\mathrm{coefficient\; of}\;\lambda^4 = b\tr(X_2),
\end{equation}
\begin{equation}\label{lambda 3}
\mathrm{coefficient\; of}\;\lambda^3 = ab\tr(X_1),
\end{equation}
\begin{equation}\label{lambda 2}
\mathrm{coefficient\; of}\;\lambda^2 = b^2\det(X_2),
\end{equation}
\begin{equation}\label{lambda 1}
\mathrm{coefficient\; of}\;\lambda = ab^2(\mathrm{\lambda-coefficient \; of} \; \det(X_1+\lambda X_2)),
\end{equation}
and
\begin{equation}\label{lambda 0}
\mathrm{coefficient\; of}\;\lambda^0 =a^2b^2 \det(X_1).
\end{equation}
Here the equation holds up to $\pm1$ which will not affect our later calculation. \textbf{Note that in the nonsplit case, the determinant means the composition of the determinant of the matrix and the norm of the quaternion algebra; and the trace means the composition of the trace of the matrix and the trace of the quaternion algebra.}

We can have the same results for $y$. Now if $x$ and $y$ are conjugated to each other by element in $G(F)$, their characteristic polynomials are equal, and hence we have
\begin{equation}\label{lambda 4'}
\tr(X_2)=\tr(Y_2),
\end{equation}
\begin{equation}\label{lambda 3'}
\tr(X_1)=\tr(Y_1),
\end{equation}
\begin{equation}\label{lambda 2'}
\det(X_2)=\det(Y_2),
\end{equation}
\begin{equation}\label{lambda 1'}
\mathrm{\lambda-coefficient \; of} \; \det(X_1+\lambda X_2)=\mathrm{\lambda-coefficient \; of} \; \det(Y_1+\lambda Y_2),
\end{equation}
and
\begin{equation}\label{lambda 0'}
\det(X_1)=\det(Y_1).
\end{equation}
By the "generic positive" assumption, $X_i$ and $Y_i$ are semisimple regular. Then the above equations tell us $X_i$ and $Y_i$ are conjugated to each other by elements in $H(F)$ (i=1,2).

\textbf{We first deal with the split case, i.e. $G=GL_6(F)$ and $H=GL_2(F)$.} By further conjugating by elements in $H(F)$, may assume that $X_1=Y_1$ be one of the following forms:
$$X_1=Y_1=\begin{pmatrix} s & 0 \\ 0 & t \end{pmatrix};\; X_1=Y_1=\begin{pmatrix} s & tv \\ t & s \end{pmatrix}$$
where $v\in F^{\times}/(F^{\times})^2, v\neq 1$. By the "generic positive" assumption, if we are in the first case, $s\neq t$; and if we are in the second case, $t\neq 0$. Let
$$X_2=\begin{pmatrix} x_{11} & x_{12} \\ x_{21} & x_{22} \end{pmatrix}, Y_2=\begin{pmatrix} y_{11} & y_{12} \\ y_{21} & y_{22} \end{pmatrix}.$$

\textbf{Case 1}: If $X_1=Y_1=\begin{pmatrix} s & 0 \\ 0 & t \end{pmatrix}$ where $s\neq t$, then by \eqref{lambda 1'}, we have $sx_{22}+tx_{11}=sy_{22}+ty_{11}$. Combining this with \eqref{lambda 4'}, we know $x_{11}=y_{11}, x_{22}=y_{22}$. Now applying \eqref{lambda 2'}, we have $x_{12}x_{21}=y_{12}y_{21}$. By the "generic position" assumption, $x_{12}x_{21}y_{12}y_{21}\neq 0$, and hence $\frac{x_{12}}{y_{12}}=\frac{y_{21}}{x_{21}}$. So we can conjugate $X_2$ by an element of the form $\begin{pmatrix} \ast & 0 \\ 0 & \ast \end{pmatrix}$ to get $Y_2$. Therefore we can conjugate $X_1,X_2$ to $Y_1,Y_2$ simultaneously via an element in $H(F)$.

\textbf{Case 2}: If $X_1=Y_1=\begin{pmatrix} s & tv \\ t & s \end{pmatrix}$ where $t\neq 0$, by \eqref{lambda 1'}, we have $s \tr(X_2)-t(vx_{21}+x_{12})=s \tr(Y_2)-t(vy_{21}+y_{12})$. Combining with \eqref{lambda 4'}, we have $vx_{21}+x_{12}=vy_{21}+y_{12}$. Let
$$x_{11}+x_{22}=y_{11}+y_{22}=A,$$
$$x_{11}x_{22}-x_{12}x_{21}=y_{11}y_{22}-y_{12}y_{21}=B,$$
and
$$vx_{21}+x_{12}=vy_{21}+y_{12}=C.$$
By the first and third equations, we can replace $x_{12},x_{22}$ by $x_{21},x_{11}$ in the second equation. We can do the same thing for the $y$'s.
It follows that
\begin{equation}\label{discriminant 2}
Ax_{11}-x_{11}^{2}-Cx_{21}+vx_{21}^{2}=Ay_{11}-y_{11}^{2}-Cy_{21}+vy_{21}^{2}=B.
\end{equation}
Now for all $k\in F$, we have
$$\begin{pmatrix} k & v \\ 1 & k \end{pmatrix}x \begin{pmatrix} k & v \\ 1 & k \end{pmatrix}^{-1}$$
$$=\frac{1}{k^2-v} \begin{pmatrix} k^2 x_{11}+kvx_{21}-kx_{12}-vx_{22} & k^2 x_{12}+kvx_{22}-kvx_{11}-v^2 x_{21} \\ kx_{11}+k^2 x_{21}-x_{12}-kx_{22} & kx_{12}+k^2 x_{22}-vx_{11}-kvx_{21} \end{pmatrix}.$$
If we write the above action in terms of $x_{11},x_{21}$, we have
\begin{eqnarray*}
x_{11}&\mapsto& (x_{11}k^2+(2vx_{21}-C)k+vx_{11}-vA)/(k^2-v):= k.x_{11},\\
x_{21}&\mapsto& (x_{21}k^2+(2x_{11}-A)k+vx_{21}-C)/(k^2-v):= k.x_{21}.
\end{eqnarray*}
If we want $y_{21}=k.x_{21}$, we need
\begin{equation}\label{discriminant 1}
((x_{21}-y_{21})k^2+(2x_{11}-A)k+vx_{21}-C+vy_{21})=0.
\end{equation}
The discriminant of \eqref{discriminant 1} is equal to
\begin{eqnarray*}
\Delta \;\text{of} \; \eqref{discriminant 1}
&=&4x_{11}^{2}-4Ax_{11}+A^2-4v(x_{21}^{2}-y_{21}^{2})+4C(x_{21}-y_{21})\\
&=&A^2-4B+4vy_{21}^{2}-4Cy_{21}\\
&=&\Delta \; \text{of} \; \eqref{discriminant 2},
\end{eqnarray*}
where the second equality comes from \eqref{discriminant 2}. So the discriminant of \eqref{discriminant 1} is a square in $F$. Hence we can find some $k\in F$ such that $y_{21}=k.x_{21}$. By conjugating by element of the form $\begin{pmatrix} k & v \\ 1 & k \end{pmatrix}$, we may assume $x_{21}=y_{21}$. This also implies $x_{12}=y_{12}$. Then by \eqref{lambda 2'} and \eqref{lambda 4'}, we have $x_{11}=y_{11},x_{22}=y_{22}$ or $x_{11}=y_{22},x_{22}=y_{11}$.

If $x_{11}\neq x_{22}$, the discriminant of \eqref{discriminant 2} is nonzero, so \eqref{discriminant 1} also has nonzero discriminant. Therefore,
it have two solutions $k_1,k_2$. Both $k_1$ and $k_2$ will make $x_{12}=y_{12},x_{21}=y_{21}$. By the "generic positive" assumption, $k_1,k_2$ conjugate $x$ to different elements. So one of them will conjugate $x$ to $y$. Therefore we have proved that we can conjugate $X_1,X_2$ to $Y_1,Y_2$ simultaneously via an element in $H(F)$.

\textbf{We now deal with the non-split case.} We can just use the same argument as in Case 2. The calculation is very similar, and the detail will
be omitted here. This finishes the proof of the Lemma.
\end{proof}

\begin{rmk}
As point out by the referee, there is another way to prove Case 2 by extension of scalars. Let $E/F$ be a finite Galois extension such that $X_1$ is split over $E$. Then by the argument in Case 1, we can find an element $h=\begin{pmatrix} a&b\\c&d\end{pmatrix}$ in $H(E)$ conjugating $X_1,X_2$ to $Y_1,Y_2$. Without loss of generality, may assume that $a\neq 0$. Also up to an element in $Z_H(E)$, we may assume that $a=1$. For any $\tau\in Gal(E/F)$, $\tau(h)$ will also conjugate $X_1,X_2$ to $Y_1,Y_2$. By the generic position assumption, $\tau(h)=hz$ for some $z\in Z_H(E)$. But since $1=a=\tau(a)$, $z$ must be the identity element which implies that $h=\tau(h)$. Therefore $h\in H(F)$, and this proves Case 2. The same argument can be also applied to the non-split case.
\end{rmk}

\begin{rmk}\label{orbit 2}
To summarize, we have an injective analytic morphism
\begin{equation}
(\Xi+\Sigma_0)/H(F)U(F)\longrightarrow \coprod_{T\in T(G)} \Ft(F)/W(G,T).
\end{equation}
For each $T\in T(G)$, let $\Ft^0(F)/W(G,T)$ be the image of the map above. Then it is easy to see the following.
\begin{enumerate}
\item $\Ft^0(F)\subset \Ft_0(F)$. Recall $\Ft_0(F)$ is the subset of $\Ft(F)$ consisting of the elements with zero trace.
\item $\Ft^0(F)$ is invariant under scalar in the sense that for $t\in \Ft^0(F), \lambda\in F^{\times}$, and $\lambda t\in \Ft^0(F)$.
\item $\Ft^0(F)$ is an open subset of $\Ft_0(F)$ under the topology on $\Ft_0(F)$ as an $F$-vector space.
\item If $T$ is split which is only possible when $G=GL_6(F)$. Then $\Ft^0(F)=\Ft_{0,reg}(F)$. This will be proved in Lemma \ref{U-invariant 2}$)$.
\end{enumerate}
Now we have a bijection
\begin{equation}\label{9.3}
(\Xi+\Sigma_0)/H(F)U(F)\longrightarrow \coprod_{T\in T(G)} \Ft^0(F)/W(G,T).
\end{equation}
\end{rmk}
Now we study the change of measures under the map \eqref{9.3}. We fix selfdual measures on $\Xi+\Sigma_0$ and $H(F)U(F)$, this induces a measure on the quotient which gives a measure $d_1 t$ on $\Ft^0(F)/W(G,T)$ via the bijection \eqref{9.3} for any $T\in T(G)$. On the other hand, we also have a selfdual measure $dt$ on $\Ft^0(F)/W(G,T)$. The following lemma tells the relations between $d_1 t$ and $dt$.

\begin{lem}\label{measure 1}
For any $T\in T(G)$, $d_1 t=D^G(t)^{1/2} dt$ for all $t\in \Ft^{0}(F)$.
\end{lem}

\begin{proof}
Let $d_2 t$ be the measure on $\Ft^0(F)/W(G,T)$ coming from the quotient $\Xi+\Lambda^0/H(F)$. By Lemma \ref{unipotent orbit},
\begin{equation}\label{9.4}
d_2t=a^4b^8d_1t.
\end{equation}
For $T_H\in T(H)$, define $\Xi+T_H=\{\Lambda(X_1,X_2)=\begin{pmatrix} 0 & 0 & X_1 \\ aI_2 & 0 & X_2 \\ 0 & bI_2 & 0 \end{pmatrix}\in \Xi+\Lambda^0| X_1\in \Ft_H(F)\}$. Then the bijection
$$\Xi+\Lambda^0/H(F)\rightarrow \coprod_{T\in T(G)}\Ft^0(F)/W(G,T)$$
factor through
$$\Xi+\Lambda^0/H(F)\rightarrow \coprod_{T_H\in T(H)} \Xi+T_H/T_H(F)\rightarrow \coprod_{T\in T(G)}\Ft^0(F)/W(G,T).$$
By the Weyl Integration Formula, the Jacobian of the first map is $D^{H}(X_1)^{-1}$ at $\Lambda(X_1,X_2)$. Combining with \eqref{9.4}, we only need to show the Jacobian of the map
$$\coprod_{T_H\in T(H)} \Xi+T_H/T_H(F)\rightarrow \coprod_{T\in T(G)}\Ft^0(F)/W(G,T)$$
is $a^4b^8 D^{H}(X_1) D^G(\Lambda(X_1,X_2))^{-1/2}$ at $\Lambda(X_1,X_2)$. We consider the composite map
\begin{equation}\label{9.5}
\Xi+T_H/T_H(F)\rightarrow \coprod_{T\in T(G)}\Ft^0(F)/W(G,T) \rightarrow F^5
\end{equation}
where the second map is taking the coefficient of the characteristic polynomial. (since the trace is always $0$, we only take the coefficients of deg $0$ to $4$.) As the Jacobian of the second map is $D^G(t)^{1/2}$ at $t\in \Ft^{0}(F)$, we only need to show the Jacobian of the composite map \eqref{9.5} is $a^4b^8 D^{H}(X_1)$ at $\Lambda(X_1,X_2)$.

We only write down the proof for the case $T_H$ is split, the proof for the rest case is similar. If $T_H$ is split, may assume that $T_H=\{\begin{pmatrix} \ast & 0\\ 0 & \ast \end{pmatrix}\}$. By the generic position assumption, we know
$$\Xi+T_H/T_H(F)=\{ \Lambda(X_1,X_2)|X_1=\begin{pmatrix} m & 0\\ 0 & n \end{pmatrix},\; X_2=\begin{pmatrix} m_1 & 1\\ x & n_1 \end{pmatrix},\; m\neq n,\; x\neq 0\}.$$
The measure on $\Xi+T_H/T_H(F)$ is just $dm dn dm_1 dn_1 dx$. Note that we always use the selfdual measure on $F$. In the proof of Lemma \ref{orbit 1}, we have write down the map \eqref{9.5} explicitly (i.e \eqref{lambda 4} to \eqref{lambda 0}):
\begin{equation}\label{9.6}
(m,n,m_1,n_1,x)\mapsto (b(m_1+n_1),ab(m+n),b^2(m_1n_1-x),ab^2(mn_1+m_1n),a^2b^2mn).
\end{equation}
By a simple computation, the Jacobian of \eqref{9.6} is
$$a^4b^8|(m-n)^2|_F=a^4b^8 D^H(X_1).$$
This finishes the proof of the lemma.
\end{proof}

\subsection{Local Sections}
For $T\in T(G)$, we can fix a locally analytic map
\begin{equation}
\Ft^0(F)\rightarrow \Xi+\Sigma^0: Y\rightarrow Y_{\Sigma}
\end{equation}
such that the following diagram commutes
$$
\begin{matrix}
\Xi+\Sigma^0 & & & &\longrightarrow & & & & \Ft^0(F)/W(G,T)\\
&&&&&&&&\\
&&\nwarrow&&&&\nearrow&&\\
&&&&&&&&\\
&&&&\Ft^0(F)&&&&
\end{matrix}
$$
Then we can also find a map $Y\rightarrow \gamma_Y$ such that $Y_{\Sigma}=\gamma_{Y}^{-1} Y\gamma_Y$.

\begin{lem}\label{section}
If $\omega_T$ is a compact subset of $\Ft_0(F)$, we can choose the map $Y\rightarrow Y_{\Sigma}$ such that the image of $\Ft^0(F)\cap \omega_T$ is contained in a compact subset of $\Xi+\Lambda$.
\end{lem}

\begin{proof}
We only write down the proof for split case, the non-split case are similar. Given $t\in \Ft^0(F)$, we want to find an element of the form $\begin{pmatrix} 0 & 0 & X \\ aI_2 & 0 & Y \\ 0 & bI_2 & 0 \end{pmatrix}$ that is a conjugation of $t$. As in the proof of Lemma \ref{orbit 1}, the characteristic polynomial of $t$ gives us the determinant and trace of both $X$ and $Y$, and also an extra equation (i.e. the $\lambda$-coefficient). Once $t$ lies in a compact subset, all these five values are bounded. Hence we can definitely choose $X$ and $Y$ such that their coordinates are bounded. Therefore, both elements belong to a compact subset.
\end{proof}

Combining the above Lemma and Proposition \ref{h-c}, we can choose the map $Y\rightarrow \gamma_Y$ with the property that there exists
$c>0$ such that
\begin{equation}\label{local 2}
\sigma(\gamma_Y)\leq c(1+\mid \log D^G(Y)\mid)
\end{equation}
for every $Y\in \Ft^0(F)\cap \omega_T$.

\subsection{Calculation of $I_{\kappa}(f)$}
By Lemma \ref{Fourier},
$$I(f,g)=(\gf)^{\hat{}}(0)=\int_{\Sigma} {}^g \hat{f}(\Xi+X)dX.$$
This implies
$$I_{\kappa}(f)=\int_{H(F)U(F)\backslash G(F)} \int_{\Sigma} {}^g \hat{f}(\Xi+X) dX \kappa(g) dg.$$
By Lemma \ref{orbit 1}, Remark \ref{orbit 2} and Lemma \ref{measure 1}, the interior integral equals

$$\Sigma_{T\in T(G)} \mid W(G,T)\mid^{-1} \int_{Z_H(F)\backslash H(F)U(F)} \int_{\Ft^0(F)} {}^g \hat{f}(y^{-1} \gamma_{Y}^{-1} Y\gamma_Y y)D^G(Y)^{1/2} dY dy$$

$$=\Sigma_{T\in T(G)} \mid W(G,T)\mid^{-1} \int_{Z_H(F)\backslash H(F)U(F)} \int_{\Ft^0(F)} {}^{\gamma_Y yg} \hat{f}(Y) D^G(Y)^{1/2} dY dy.$$
So we can rewrite $I_{\kappa}(f)$ as

$$I_{\kappa}(f)=\Sigma_{T\in T(G)} \mid W(G,T)\mid^{-1} \int_{\Ft^0(F)} \int_{Z_G(F)\backslash G(F)} \hat{f}(g^{-1}Yg) \kappa(\gamma_{Y}^{-1} g) dg D^G(Y)^{1/2} dY.$$

For $T\in T(G), Y\in \Ft^0(F)$, define $\kappa_Y$ on $A_T(F)\backslash G(F)$ to be
\begin{equation}\label{kappa 1}
\kappa_Y(g)=\nu(A_T)\int_{Z_G(F)\backslash A_T(F)} \kappa(\gamma_{Y}^{-1} ag) da.
\end{equation}
Then we have
\begin{equation}\label{9}
\begin{split}
I_{\kappa}(f)=&\Sigma_{T\in T(G)} \nu(A_T)^{-1} \mid W(G,T)\mid^{-1}\\
&\times \int_{\Ft^0(F)} \int_{A_T(F)\backslash G(F)} \hat{f}(g^{-1}Yg) \kappa_Y(g) dg D^G(Y)^{1/2} dY.
\end{split}
\end{equation}

\section{Calculation of the limit $\lim_{N\rightarrow \infty} I_{x,\omega,N}(f)$}
In the last section, we made the transfer of the integral $I_{x,\omega,N}(f)$ to the form that is similar to Arthur's local trace formula. The only difference is that our truncated function is different from the one given by Arthur. In this section, we first show that we are able to change the truncated function. Then by applying Arthur's local trace formula, we are going to compute the limit $\lim_{N\rightarrow \infty} I_{x,\omega,N}(f)$. This is the most technical section of this paper. In 9.1 and 9.2, we study our truncated function $\kappa_N$ and introduce Arthur's truncated function. From 9.3 to 9.5, we prove that we are able to change the truncated function when $x$ is in the center or not split. In 9.6, we deal with the case when $x$ is split but not belongs to the center. In 9.7, we compute the limit $\lim_{N\rightarrow \infty} I_{x,\omega,N}(f)$ by applying Arthur's local trace formula. We will also postpone the proof of two technical lemmas (i.e. Lemma \ref{major 2} and Lemma \ref{split zero}) to Appendix A.
\subsection{Convergence of a premier expression}
For $x\in H_{ss}(F)$, using the same notation as in Section 7.2, we have

$$I_{x,\omega}(f,g)=\int_{\Fh_x '(F)} \int_{\Fh''(F)} \gf_{x,\omega}(X'+X'') dX'' dX'.$$
Then we can write $I_{x,\omega,N}(f)$ as
$$I_{x,\omega,N}(f)=\int_{\Fh_x '(F)} \int_{H_x(F)U_x(F)\backslash G(F)} \int_{\Fh''(F)} \gf_{x,\omega}(X'+X'') dX'' \kappa_N(g) dg dX'.$$
Rewrite the two interior integrals above as
$$\int_{G_x(F)\backslash G(F)} \int_{H_x(F) U_x(F)\backslash G_x(F)} \int_{\Fh''(F)} {}^{g'' g}f^{\xi}_{x,\omega}(X'+X'') dX'' \kappa_N(g''g)dg'' dg.$$
After applying the formula \eqref{9}, together with the fact that we have defined $\Ft'=\Fh_x '$ in Section 7.2, we have
\begin{equation}\label{10.1}
\begin{split}
I_{x,\omega,N}(f)=&\Sigma_{T\in T(G_x)} \nu(A_T\cap Z_{G_x}\backslash A_T)^{-1} \mid W(G_x,T)\mid^{-1}\\
&\times \int_{\Ft'(F) \times (\Ft'')^0(F)} D^{G_x}(X'')^{1/2}\\
&\times \int_{Z_{G_x} A_T(F)\backslash G(F)} {}^g f^{\sharp}_{x,\omega}(X'+X'') \kappa_{N,X''}(g) dg dX'' dX'
\end{split}
\end{equation}
where
\begin{equation}\label{kappa 2}
\kappa_{N,X''}(g)=\nu(A_T\cap Z_{G_x}\backslash A_T) \int_{Z_{G_x}\cap A_T(F)\backslash A_T(F)} \kappa_N(\gamma_{X''}^{-1} ag) da.
\end{equation}
Note that the formula $\eqref{9}$ is only for the case when $x$ is in the center. However, as we explained at the beginning of Section 8, when $x$ is not contained in the center, the computation is easier, and we can get a similar formula as \eqref{9} with replacing $\Ft^0$ by $(\Ft'')^0$ and replacing $G$ by $G_x$.

\begin{lem}\label{major 2}
For $T\in T(G_x)$, let $\omega_{T''}$ be a compact subset of $\Ft''(F)$. There exist a rational function $Q_T(X'')$ on $\Ft''(F)$, $k\in \BN$ and $c>0$ such that
$$\kappa_{N,X''}(g)\leq C N^k \sigma(g)^k (1+|\log (| Q_T(X'')|_F) |)^k (1+| \log D^{G_x}(X'') |)^k$$
for every $X''\in (\Ft'')^0(F)\cap \omega_{T''}, g\in G(F), N\geq 1$.
\end{lem}

\begin{proof}
This is a technical lemma, we will postpone the proof to Appendix A.1.
\end{proof}

Now let $\CQ_T$ be a finite set of polynomials on $\Ft''(F)$ that contains $D^{G_x}(X'')$, the denominator and numerator of $Q_T(X'')$ and some other polynomials that will be defined later in Section 9.5. For $l>0$, let $\Ft_0(F)[\leq l]$ be the set of $X=X'+X''\in \Ft_0(F)$ such that there exists $Q\in \CQ_T$ with $|Q(X'')|_F\leq l$, and let $\Ft_0(F)[>l]$ be its complement. We define $I_{N,\leq l}$ to be the integral of the expression of $I_{x,\omega,N}(f)$ restricted on $(\Ft'(F)\times (\Ft'')^0(F))\cap \Ft_0(F)[\leq l]$ (as in \eqref{10.1}). Similarly we can define $I_{N,>l}$. We then have
\begin{equation}
I_{x,\omega,N}(f)=I_{N,\leq l}+I_{N,>l}.
\end{equation}

\begin{lem}\label{I to I*}
The following statements hold.
\begin{enumerate}
\item There exist $k\in \BN$ and $c>0$ such that $\mid I_{x,\omega,N}(f)\mid \leq cN^k$ for all $N\geq 1$.
\item There exist $b\geq 1$ and $c>0$ such that $\mid I_{N,\leq N^{-b}} \mid\leq cN^{-1}$ for all $N\geq 1$.
\end{enumerate}
\end{lem}

\begin{proof}
By condition (5) of a good neighborhood (as in Definition \ref{good nbd defn}), there exists a compact subset $\Gamma\subset G(F)$ such that $(\gf_{x,\omega})^{\hat{}}=0$ if $g\notin G_x(F)\Gamma$.

By replacing $Z_{G_x}A_T(F)\backslash G(F)$ by $Z_{G_x}A_T(F)\backslash G_x\cdot \gamma$ for some $\gamma\in \Gamma$,
we can majorize ${}^{\gamma} f_{x,\omega}^{\sharp}$ by a linear combination of function $f'\otimes f''$ where $f'\in C_{c}^{\infty}(\Fg_x '(F))$,
and $f''\in C_{c}^{\infty}(\Fg''(F))$. So the integral in \eqref{10.1} is majored by
\begin{equation}\label{major 2.1}
\int_{\Ft'(F) \times (\Ft'')^0(F)} D^{G_x}(X'')^{1/2}  \int_{Z_{G_x} A_T(F)\backslash G_x(F)} f'(X')f''(g^{-1} X''g) \kappa_{N,X''}(\gamma g) dg dX'' dX'.
\end{equation}
Now we fix a compact subset $\omega_{T''}\subset \Ft''(F)$ such that for every $g\in G_x(F)$, the function $X''\rightarrow f''(g^{-1}X'' g)$ on $\Ft''(F)$ is supported on $\omega_{T''}$. By Proposition \ref{h-c}, up to an element in $Z_{G_x}(F)A_T(F)$, we may choose g such that $\sigma(g)\ll 1+|\log(D^{G_x}(X''))|$. Using the lemma above, we have
$$\kappa_{N,X''}(\gamma g)\ll N^k \phi(X'')$$
where
$$\phi(X'')=(1+| \log (| Q_T(X'')|_F )|)^k (1+|\log( D^{G_x}(X''))|)^{2k}.$$
So the expression \eqref{major 2.1} is majored by
$$N^k \int_{\Ft'(F) \times (\Ft'')^0(F)} D^{G_x}(X'')^{1/2}\int_{Z_{G_x} A_T(F)\backslash G_x(F)} f'(X')f''(g^{-1} X''g) \phi(X'') dg dX'' dX.$$
This is majored by
\begin{equation}\label{major 2.2}
N^k \int_{\Ft_0(F)} J_{G_x}(X'+X'',f'\otimes f'')\phi(X'') dX'' dX'
\end{equation}
where $J_{G_x}$ is the orbital integral. Due to the work of Harish-Chandra, the orbital integral is always bounded, and hence \eqref{major 2.2} is majored by
\begin{equation}\label{major 2.3}
N^k\int_{\omega} \phi(X'')dX'' dX'
\end{equation}
where $\omega$ is a compact subset of $\Ft_0(F)$. By Lemma 2.4 of \cite{W10}, $\phi(X)$ is locally integrable, and hence the integral in \eqref{major 2.3} is convergent. This finishes the proof of the first part.

For the second part, by the same argument, we have majorization
$$\mid I\mid_{N,\leq N^{-b}} \ll N^k \int_{\omega\cap \Ft_{0}(F)[\leq N^{-b}]} \phi(X)dX.$$
Then, by the Schwartz inequality, the right hand side is majored by
\begin{eqnarray*}
&&N^k(\int_{\omega\cap \Ft_{0}(F)[\leq N^{-b}]} dX)^{1/2}(\int_{\omega\cap \Ft_{0}(F)[\leq N^{-b}]} \phi(X)^2 dX)^{1/2}\\
&\ll& N^k\cdot \Sigma_{Q\in \CQ_T} mes\{X\in \omega\mid \mid Q(X)\mid_F \leq N^{-b}\} \ll N^k(N^{-b})^r
\end{eqnarray*}
for some $r>0$ that only depends on the dimension of $\Ft_0$. Now we just need to let $b$ large such that $N^k(N^{-b})^r \ll N^{-1}$. This finishes the proof of the Lemma.
\end{proof}

\begin{defn}
Define that $I^{\ast}_{x,\omega,N}(f)=I_{N,>N^{-b}}$.
\end{defn}
By the Lemma above, we have
\begin{equation}
\lim_{N\rightarrow \infty}(I_{x,\omega,N}(f)-I^{\ast}_{x,\omega,N}(f))=0.
\end{equation}

\subsection{Combinatorial Definition}
Fix $T\in T(G_x)$, let $M_{\sharp}$ be the centralizer of $A_T$ in $G$. This is a Levi subgroup of G, it is easy to check that $A_T=A_{M_{\sharp}}$. \textbf{We assume that $x$ is non split or $x$ is in the center from Section 9.2 to 9.5. We will deal with the case where $x$ is split and not in the center separately in Section 9.6.} In particular, under this hypothesis, we know $Z_{G_x}\cap A_T=Z_G$ for any $T\in T(G_x)$,
and hence we have $\nu(A_T\cap Z_{G_x}\backslash A_T)=\nu(Z_G\backslash A_T)=\nu(A_T)$. Note that we always choose the Haar measure on $G$ so that $\nu(Z_G)=1$.

Let $\CY=(Y_{P_{\sharp}})_{ P_{\sharp}\in \CP(M_{\sharp})}$ be a family of elements in $\Fa_{M_{\sharp}}$ that are $(G,M_{\sharp})$-orthogonal and positive. Then for $Q=LU_Q\in \CF(M_{\sharp})$, let $\zeta\rightarrow \sigma_{M_{\sharp}}^{Q}(\zeta,\CY)$ be the characteristic function on $\Fa_{M_{\sharp}}$ that supports on the sum of $\Fa_L$ and the convex envelop generated by the family $(Y_{P_{\sharp}})_{ P_{\sharp}\in \CP(M_{\sharp}),P_{\sharp}\subset Q}$. Let $\tau_Q$ be the characteristic function on $\Fa_{M_{\sharp}}$ that supports on  $\Fa_{M_{\sharp}}^{L}+\Fa_{Q}^{+}$. The following proposition follows from 3.9 of \cite{Ar91}.
\begin{prop}\label{truncation}
The function
$$\zeta\rightarrow \sigma_{M_{\sharp}}^{Q}(\zeta,\CY) \tau_Q(\zeta-Y_Q)$$
is the characteristic function on $\Fa_{M_{\sharp}}$, whose support is on the sum of $\Fa_{Q}^{+}$ and the convex envelope generated by $(Y_{P_{\sharp}})_{ P_{\sharp}\in \CP(M_{\sharp}),P_{\sharp}\subset Q}$. Moreover, for every $\zeta\in \Fa_{M_{\sharp}}$, the following identity
holds.
\begin{equation}
\Sigma_{Q\in \CF(M_{\sharp})} \sigma_{M_{\sharp}}^{Q}(\zeta,\CY) \tau_Q(\zeta-Y_Q)=1.
\end{equation}
\end{prop}

\subsection{Change the truncated function}
We use the same notation as Section 9.2. Fix a minimal Levi subgroup $M_{min}$ of $G$ contained in $M_{\sharp}$, a hyperspecial subgroup $K_{min}$ of $G$ related to $M_{min}$ and $P_{min}=M_{min}U_{min}\in \CP(M_{min})$. Let $\Delta_{min}$ be the set of simple roots of $A_{M_{min}}$ in $\Fu_{min}$.
Given $Y_{min}\in \Fa_{P_{min}}^{+}$, for any $P'\in \CP(M_{min})$, there exists a unique element $w\in W(G,M_{min})$ such that $wP_{min}w^{-1}=P'$. Set $Y_{P'}=wY_{P_{min}}$. The family $(Y_{P'})_{P'\in \CP(M_{min})}$ is $(G,M_{min})$-orthogonal and positive. For $g\in G(F)$, define $\CY(g)=(Y(g)_Q)_{Q\in \CP(M_{\sharp})}$ to be
$$Y(g)_Q=Y_Q-H_{\bar{Q}}(g).$$
Then it is easy to show the following statements.

{\bf (1)} There exists $c_1>0$ such that for any $g\in G(F)$ with $\sigma(g)< c_1 \inf\{\alpha(Y_{P_{min}}); \alpha\in \Delta_{min}\}$, the family $\CY(g)$ is $(G,M_{\sharp})$-orthogonal and positive. And $Y(g)_Q\in \Fa_{Q}^{+}$ for all $Q\in \CF(M_{\sharp})$.

We fix such a $c_1$. Note that for $m\in M_{\sharp}(F)$, $\CY(mg)$ is a translation of $\CY(g)$ by $H_{M_{\sharp}}(m)$. Hence $\CY(g)$ is $(G,M_{\sharp})$-orthogonal and positive for
$$g\in M_{\sharp}(F) \{g'\in G(F)\mid \sigma(g')<c_1 \inf\{\alpha(Y_{P_{min}}); \alpha\in \Delta_{min}\}\}.$$
For such a $g$, let
\begin{equation}\label{truncation 2}
\tilde{v}(g)=\nu(A_T)\int_{Z_{G}(F) \backslash A_T(F)} \sigma_{M_{\sharp}}^{G}(H_{M_{\sharp}}(a),\CY(g))da.
\end{equation}

{\bf (2)} There exist $c_2>0$ and a compact subset $\omega_T$ of $\Ft_0(F)$ satisfying the following condition: If $g\in G(F)$, and
$$
X\in \Ft_0(F)[>N^{-b}]\cap (\Ft'(F)\times (\Ft'')^0(F))
$$
with $({}^g f_{x,\omega})^{\sharp}(X)\neq 0$, then $X\in \omega_T$ and $\sigma_T(g)< c_2\log(N)$.

In fact, since $({}^g f_{x,\omega})^{\sharp}(X)=(f_{x,\omega})^{\hat{}}(g^{-1}Xg)$, $g^{-1}Xg$ is contained in compact subset of $\Fg_{x,0}(F)$.
This implies that $X$ is in a compact subset of $\Ft_{0}(F)$. By Proposition \ref{h-c}, we have
$$\sigma_T(g)\ll 1+\mid \log D^{G_x}(X) \mid = 1+\mid \log D^{G_x}(X'') \mid \ll \log(N)$$
where the last inequality holds because $X\in \Ft_0(F)[>N^{-b}]$ and is contained in a compact subset.

Now, fix $\omega_T$ and $c_2$ as in (2). We may assume that $\omega_{T}=\omega_{T'}\times \omega_{T''}$ where $\omega_{T'}$ is a compact subset of $\Ft'(F)$ and $\omega_{T''}$ is a compact subset of $\Ft''(F)$. Suppose that
$$c_2\log(N)< c_1\inf\{\alpha(Y_{min})\mid \alpha\in \Delta_{min}\}.$$
Here $c_1$ comes from (1). Then $\tilde{v}(g)$ is defined for all $g\in G(F)$ satisfying condition (2).

\begin{prop}\label{change truncation}
There exist $c>0$ and $N_0\geq 1$ such that if $N\geq N_0$ and $c\log(N)<\inf\{\alpha(Y_{min})\mid \alpha\in \Delta_{min}\}$, we have
\begin{equation}
\int_{Z_{G_x}(F) A_T(F)\backslash G(F)} {}^g f^{\sharp}_{x,\omega}(X) \kappa_{N,X''}(g) dg=\int_{Z_{G_x}(F) A_T(F)\backslash G(F)} {}^g f^{\sharp}_{x,\omega}(X) \tilde{v}(g) dg
\end{equation}
for every $X\in \Ft_0(F)[>N^{-b}]\cap (\Ft'(F)\times (\Ft'')^0(F))$.
\end{prop}

\begin{proof}
For any $Z_{P_{min}}\in \Fa_{P_{min}}^{+}$, replacing $Y_{P_{min}}$ by $Z_{P_{min}}$, we can construct the family $\CZ(g)$ in the same way as $\CY(g)$. Assume
\begin{equation}\label{condition 1}
c_2\log(N)<c_1 \inf\{\alpha(Z_{min})\mid \alpha\in \Delta_{min}\}.
\end{equation}

For $g\in G(F)$ with $\sigma(g)<c_2 \log(N)$, $\CZ(g)$ is still $(G,M_{\sharp})$-orthogonal and positive. So for $a\in A_T(F)$, by Proposition \ref{truncation}, we have
$$\Sigma_{Q\in F(M_{\sharp})} \sigma_{M_{\sharp}}^{Q}(H_{M_{\sharp}}(a), \CZ(g))\tau_Q(H_{M_{\sharp}}(a)-\CZ(g)_Q)=1.$$
Then we know
\begin{equation}
\tilde{v}(g)=\nu(A_T) \Sigma_{Q\in F(M_{\sharp})} \tilde{v}(Q,g)
\end{equation}
and
\begin{equation}
\kappa_{N,X''}(g)=\nu(A_T) \Sigma_{Q\in F(M_{\sharp})} \kappa_{N,X''}(Q,g)
\end{equation}
where
\begin{equation}\label{J part 3}
\tilde{v}(Q,g)=\int_{Z_G\backslash A_T(F)} \sigma_{M_{\sharp}}^{G}(H_{M_{\sharp}}(a),\CY(g)) \sigma_{M_{\sharp}}^{Q}(H_{M_{\sharp}}(a), \CZ(g))\tau_Q(H_{M_{\sharp}}(a)-\CZ(g)_Q) da
\end{equation}
and
\begin{equation}\label{I part 3}
\kappa_{N,X''}(Q,g)=\int_{Z_G\backslash A_T(F)} \kappa_N(\gamma_{X''}^{-1} ag) \sigma_{M_{\sharp}}^{Q}(H_{M_{\sharp}}(a), \CZ(g))\tau_Q(H_{M_{\sharp}}(a)-\CZ(g)_Q) da.
\end{equation}

{\bf (3)} The functions $g\rightarrow \tilde{v}(Q,g)$ and $g\rightarrow \kappa_{N,X''}(Q,g)$ are left$A_T(F)$-invariant.

Since for $t\in A_T(F)$, $H_{P'}(tg)=H_{M_{\sharp}}(t)+H_{P'}(g)$ for all $P'\in \CP(M_{\sharp})$. We can just change variable $a\rightarrow at$ in the definition of $\tilde{v}(Q,g)$ and $\kappa_{N,X''}(Q,g)$, this gives us the left $A_T(F)$-invariant of both functions, and (3) follows.

Now for $X\in \Ft'(F)\times (\Ft'')^0(F)$, we have
\begin{equation}
\int_{Z_{G_x} A_T(F)\backslash G(F)} {}^g f^{\sharp}_{x,\omega}(X) \kappa_{N,X''}(g) dg=\nu(A_T)\Sigma_{Q\in \CF(M_{\sharp})} I(Q,X)
\end{equation}
and
\begin{equation}
\int_{Z_{G_x} A_T(F)\backslash G(F)} {}^g f^{\sharp}_{x,\omega}(X) \tilde{v}(g) dg=\nu(A_T)\Sigma_{Q\in \CF(M_{\sharp})} J(Q,X)
\end{equation}
where
\begin{equation}\label{I part}
I(Q,X)=\int_{Z_{G_x} A_T(F)\backslash G(F)} {}^g f^{\sharp}_{x,\omega}(X) \kappa_{N,X''}(Q,g) dg
\end{equation}
and
\begin{equation}\label{J part}
J(Q,X)=\int_{Z_{G_x} A_T(F)\backslash G(F)} {}^g f^{\sharp}_{x,\omega}(X) \tilde{v}(Q,g) dg.
\end{equation}
Then it is enough to show for all $Q\in \CF(M_{\sharp})$, $I(Q,X)=J(Q,X)$.

\textbf{Firstly we consider the case when $Q=G$}. Suppose
\begin{equation}\label{condition 2}
\sup\{\alpha(Z_{P_{min}})\mid \alpha\in \Delta_{min}\}
\leq \begin{cases}
\inf\{\alpha(Y_{P_{min}}) \mid \alpha\in \Delta_{min}\}, \\
\log(N)^2.
\end{cases}
\end{equation}
Then we are going to prove

{\bf (4)} There exists $N_1>1$ such that for all $N\geq N_1$, $g\in G(F)$ with $\sigma_T(g)\leq c_2\log(N)$, and for all $X''\in \omega_{T''}\cap (\Ft'')^0(F) [>N^{-b}]$, we have
\begin{equation}\label{change 1}
\kappa_{N,X''}(G,g)=\tilde{v}(G,g).
\end{equation}
Here $\Ft''[>N^{-b}]$ means that we only consider the polynomials $D^{G_x}(X'')$ together with the numerator and the denominator of $Q_T(X'')$ which are elements in $\CQ_T$.

In order to prove (4), it is enough to show that for all $a\in A_T(F)$ with $\sigma_{M_{\sharp}}^{G}(H_{M_{\sharp}}(a), \CZ(g))=1$, we have $\sigma_{M_{\sharp}}^{G}(H_{M_{\sharp}}(a), \CY(g))=\kappa_N(\gamma_{X}^{-1}ag)$. Since both sides of \eqref{change 1} are left $A_T(F)$-invariant, we may assume $\sigma(g)\leq c_2\log(N)$.

By the first inequality of \eqref{condition 2}, $\sigma_{M_{\sharp}}^{G}(H_{M_{\sharp}}(a), \CZ(g))=1$ will implies
$$\sigma_{M_{\sharp}}^{G}(H_{M_{\sharp}}(a), \CY(g))=1.$$
Then by the second inequality of \eqref{condition 2}, together with the fact that $\sigma(g)\ll \log(N)$, we know $\mid \CZ(g)_{P'}\mid \ll \log(N)^2$ for every $P'\in \CP(M_{\sharp})$, here $| \cdot |$ is the norm on $\Fa_{M_{\sharp}} /\Fa_G$. Then combining with the fact that $\sigma_{M_{\sharp}}^{G}(H_{M_{\sharp}}(a), \CZ(g))=1$, we know up to an element in center, $\sigma(a)\ll \log(N)^2$. Since the integrals defining $I(Q,X)$ and $J(Q,X)$ are integrating modulo the center, we may just assume that $\sigma(a)\ll \log(N)^2$.

By \eqref{local 2} and the fact that $X''\in \omega_{T''}\cap (\Ft'')^0(F) [>N^{-b}]$, we know $\sigma(\gamma_X)\ll 1+\mid \log D^{G_x}(X)\mid \ll \log(N)$, and hence $\sigma(\gamma_{X}^{-1} ag)\ll \log(N)^2$. By the definition of $\kappa_N$, together with the relations between the norm of an element and the norm of its Iwasawa decomposition, we can find $c_3>0$ such that for any $g'\in G(F)$ with $\sigma(g')<c_3 N$, we have $\kappa_N(g')=1$. Now for $N$ large enough, we definitely have $\sigma(\gamma_{X}^{-1} ag)<c_3 N$. In this case, we have $\kappa_N(\gamma_{X}^{-1} ag)=1=\sigma_{M_{\sharp}}^{G}(H_{M_{\sharp}}(a), \CY(g))$. This proves (4).

Combining (2) and (4), together with \eqref{I part} and \eqref{J part}, we have
\begin{equation}
I(G,X)=J(G,X)
\end{equation}
for every $N\geq N_1, X\in \Ft_0(F)[>N^{-b}] \cap (\Ft'(F)\times (\Ft'')^0(F))$.
\\

\textbf{Now for $Q=LU_Q\in \CF(M_{\sharp})$ with $Q\neq G$} We can decompose the integrals in \eqref{I part} and \eqref{J part} by
\begin{equation}\label{I part 2}
I(Q,X)=\int_{K_{min}} \int_{Z_{G_x} A_T(F)\backslash L(F)}\int_{U_{\bar{Q}}(F)} {}^{\bar{u}lk} f^{\sharp}_{x,\omega}(X) \kappa_{N,X''}(Q,\bar{u}lk) d\bar{u} \delta_Q(l)dl dk
\end{equation}
and
\begin{equation}\label{J part 2}
J(Q,X)=\int_{K_{min}} \int_{Z_{G_x} A_T(F)\backslash L(F)}\int_{U_{\bar{Q}}(F)} {}^{\bar{u}lk} f^{\sharp}_{x,\omega}(X) \tilde{v}(Q,\bar{u}lk) d\bar{u} \delta_Q(l)dl dk.
\end{equation}

The following two properties will be proved in Section 9.4 and 9.5.

{\bf (5)} If $g\in G(F)$ and $\bar{u}\in U_{\bar{Q}}(F)$ with
$$
\sigma(g),\sigma(\bar{u}g)< c_1\inf\{\alpha(Z_{P_{min}})\mid \alpha\in \Delta_{min}\},
$$
then $\tilde{v}(Q,\bar{u}g)=\tilde{v}(Q,g)$.

{\bf (6}) Given $c_4>0$, we can find $c_5>0$ such that if
$$c_5\log(N)<\inf\{ \alpha(Z_{P_{min}})\mid \alpha\in \Delta_{min}\},$$
we have $\kappa_{N,X''}(Q,\bar{u}g) =\kappa_{N,X''}(Q,g)$ for all $g\in G(F)$ and $\bar{u}\in U_{\bar{Q}}(F)$ with $\sigma(g),\sigma(\bar{u}),\sigma(\bar{u}g)<c_4 \log(N)$, and for all $X''\in \omega_{T''}\cap (\Ft'')^0(F)) [>N^{-b}]$.

Based on (5) and (6), we are going to show:

{\bf (7)} There exists $c_5>0$ such that if
\begin{equation}\label{condition 3}
c_5\log(N)<\inf\{\alpha(Z_{P_{min}})\mid \alpha\in \Delta_{min}\},
\end{equation}
we have $I(Q,X)=J(Q,X)=0$ for all $X\in \Ft_0(F)[> N^{-b}]\cap (\Ft'(F)\times (\Ft'')^0(F))$.

In fact, by (2), we may assume that $X\in \omega_T$. We first consider $I(Q,X)$. By (2), we can restrict the integral $\int_{Z_{G_x} A_T(F)\backslash L(F)}$ in \eqref{I part 2} to those $l$ for which their exist $\bar{u}\in U_{\bar{Q}}(F)$ and $K\in K_{min}$ such that $\sigma_T(\bar{u}lk)<c_2\log(N)$. Then up to an element in $A_T(F)$, $l$ can be represented by an element in $L(F)$ such that $\sigma(l)<c_6\log(N)$ for some constant $c_6$. We can find $c_7>0$ such that for all $l,\bar{u}$ and $k$ with $\sigma(l)<c_6\log(N)$ and $\sigma(\bar{u}lk)< c_2\log(N)$, we have $\sigma(\bar{u})< c_7\log(N)$. Now let $c_4=c_2+c_7$, and choose $c_5$ as in (6). Then by applying (6) we know that for fixed $k\in K_{min}, l\in L(F)$ with $\sigma(l)<c_6\log(N)$, we have
\begin{equation}\label{U-invariant 1}
{}^{\bar{u}lk} f^{\sharp}_{x,\omega}(X) \kappa_{N,X''}(Q,\bar{u}lk)={}^{\bar{u}lk} f^{\sharp}_{x,\omega}(X) \kappa_{N,X''}(Q,lk)
\end{equation}
for all $\bar{u}\in U_{\bar{Q}}(F)$.
\\
On the other hand, if $\sigma(\bar{u}lk)\geq c_2\log(N)$, both side of \eqref{U-invariant 1} are equal to $0$ by (2). Therefore \eqref{U-invariant 1} holds for all $\bar{u},\;l$ and $k$.

From \eqref{U-invariant 1}, we know that in the expression of $I(Q,X)$ (as in \eqref{I part 2}), the inner integral is just
$$\int_{U_{\bar{Q}}(F)} {}^{\bar{u}lk} f^{\sharp}_{x,\omega}(X) d\bar{u}.$$
This is zero for $Q\neq G$ by Lemma \ref{strongly cuspidal}. Hence $I(Q,X)=0$. By applying the same argument except replacing (6) by (5), we can also show $J(Q,X)=0$. This proves (7), and finishes the proof of the Proposition.

The last thing that remains to do is to verify that we can find $Z_{P_{min}}$ satisfies condition \eqref{condition 1}, \eqref{condition 2} and \eqref{condition 3}. This just follows from the conditions we imposed on $N$ and $Y_{P_{min}}$.
\end{proof}

\subsection{Proof of 9.3(5)}
By \eqref{J part 3}, we know
$$\tilde{v}(Q,G)=\int_{Z_G(F)\backslash A_T(F)} \sigma_{M_{\sharp}}^{G}(H_{M_{\sharp}}(a),\CY(g)) \sigma_{M_{\sharp}}^{Q}(H_{M_{\sharp}}(a), \CZ(g))\tau_Q(H_{M_{\sharp}}(a)-\CZ(g)_Q) da.$$
The function $\zeta\rightarrow \sigma_{M_{\sharp}}^{Q}(\zeta,\CZ(g))$ and $\zeta\rightarrow \tau_Q(\zeta-\CZ(g)_Q)$ only depend on $H_{\bar{P'}}(g)$ for $P'\in \CF(M_{\sharp})$ with $P'\subset Q$. For such $P'$, $H_{\bar{P'}}(\bar{u}g)=H_{\bar{P'}}(g)$ for $\bar{u}\in U_{\bar{Q}}(F)$. Therefore for all $\bar{u}\in U_{\bar{Q}}(F)$,
$$\sigma_{M_{\sharp}}^{Q}(H_{M_{\sharp}}(a), \CZ(g))\tau_Q(H_{M_{\sharp}}(a)-\CZ(g)_Q)=\sigma_{M_{\sharp}}^{Q}(H_{M_{\sharp}}(a), \CZ(\bar{u}g))\tau_Q(H_{M_{\sharp}}(a)-\CZ(\bar{u}g)_Q).$$

Now for all $a\in A_T(F)$ with the property that
$$
\sigma_{M_{\sharp}}^{Q}(H_{M_{\sharp}}(a),\CZ(g))\tau_Q(H_{M_{\sharp}}(a)-\CZ(g)_Q)\neq 0,
$$
we need to show
\begin{equation}\label{10.2}
\sigma_{M_{\sharp}}^{G}(H_{M_{\sharp}}(a),\CY(g))=\sigma_{M_{\sharp}}^{G}(H_{M_{\sharp}}(a),\CY(\bar{u}g)).
\end{equation}
For any $P'\in \CP(M_{\sharp})$ with $P'\subset Q$, it determines a chamber $\Fa_{P'}^{L,+}$ in $\Fa_{M_{\sharp}}^{L}$. Let $\zeta=H_{M_{\sharp}}(a)$,
and fix a $P'$ such that $proj_{M_{\sharp}}^{L}(\zeta)\in CL(\Fa_{P'}^{L,+})$ where $CL$ means closure.

\begin{lem}$\zeta\in CL(\Fa_{P'}^{+})$.
\end{lem}

\begin{proof}
By the definition of the functions $\sigma_{M_{\sharp}}^{Q}$ and $\tau_Q$, together with the fact that $\sigma_{M_{\sharp}}^{Q}(H_{M_{\sharp}}(a),\CY(g))\tau_Q(H_{M_{\sharp}}(a)-\CZ(g)_Q)\neq 0$, we know that $\zeta$ is the summation of an element $\zeta'\in \Fa_{Q}^{+}$ and an element $\zeta''$ belonging to the convex envelop generated by $\CZ(g)_{P''}$ for $P''\in \CP(M_{\sharp})$ with $P''\subset Q$. For any root $\alpha$ of $A_{M_{\sharp}}$ in $\Fg$, positive with respect to $P'$, if $\alpha$ is in $U_Q$, then it is positive for all $P''\subset Q$ above. By 10.3(1), $\CZ(g)_{P''}\in \Fa_{P''}^{+}$, and $\alpha(\zeta'')>0$. Also we know $\alpha(\zeta')>0$ because $\alpha$ is in $U_Q$ and $\zeta'\in \Fa_{Q}^{+}$. Combining these two inequalities, we have $\alpha(\zeta)>0$.

If $\alpha$ is in $U_{P'}\cap L$, then $\alpha(\zeta)=\alpha(proj_{M_{\sharp}}^{L}(\zeta))\geq 0$ by the choice of $P'$. So the lemma follows.
\end{proof}

By Lemma 3.1 of \cite{Ar91}, for $\zeta\in CL(\Fa_{P'}^{+})$, $\sigma_{M_{\sharp}}^{G}(\zeta,\CY(g))=1$ is equivalent to certain inequality on $\zeta-\CY(g)_{P'}$. This only depends on $H_{\bar{P'}}(g)$. Since $P'\subset Q$ and $H_{\bar{P'}}(g)=H_{\bar{P'}}(\bar{u}g)$, \eqref{10.2} follows, and hence 9.3(5) is proved.

\subsection{Proof of 9.3(6)}
As same as in Section 9.6, we fix a map $X''\rightarrow \gamma_{X''}$ such that
\begin{enumerate}
\item There exists a compact subset $\Omega$ of $\Xi+\Sigma$ such that $X_{\Sigma}''=\gamma_{X''}^{-1} X''\gamma_{X''}\in \Omega$ for all $X''\in \omega_{T''}\cap (\Ft'')^0(F)$.
\item There exists $c_1>0$ such that $\sigma(\gamma_{X''})<c_1\log(N)$ for all $X''\in \omega_{T''}\cap (\Ft'')^0(F)[>N^{-b}]$
\end{enumerate}
For $Q=LU_Q\in \CF(M_{\sharp})$, let $\Sigma_{Q}^{+}$ be the roots of $A_{M_{\sharp}}$ in $\Fu_Q$.

\begin{lem}\label{U-invariant 2}
For $c>0$, there exists $c'>0$ satisfying the following condition:
For given $a\in A_T(F)$, $g\in G(F)$, $\bar{u}\in U_{\bar{Q}}(F)$ and $X''\in \omega_{T''} \cap (\Ft'')^0(F)[>N^{-b}]$,
assume that $\sigma(g), \sigma(\bar{u}),\sigma(\bar{u}g)<c\log(N)$, and $\alpha(H_{M_{\sharp}}(a))>c'\log(N)$ for all $\alpha\in \Sigma_{Q}^{+}$. Then $\kappa_N(\gamma_{X''}^{-1}a\bar{u}g)=\kappa_N(\gamma_{X''}^{-1}ag)$.
\end{lem}

\begin{proof}
We first prove:

{\bf (3)} It's enough to treat the case when $T\in T(G_x)$ is split.

In fact, if $F'/F$ is a finite extension, we can still define $\kappa_{N}^{F'}$ on $G(F')$ in the same way. It is easy to see that $\kappa_{N}^{F'}=\kappa_{Nval_{F'}(\varpi_F)}$ on $G(F)$, and hence we can pass to a finite extension of $F$. Therefore we may assume that $T$ and $G_x$ are split. This proves (3).

{\bf (4)} Let $X''\rightarrow \uline{\gamma_{X''}} ,\uline{X_{\Sigma}}''=(\uline{\gamma_{X''}} )^{-1}X \uline{\gamma_{X''}} $ be another local sections satisfying Conditions (1) and (2). Then the lemma holds for $\gamma_{X''},X_{\Sigma}''$ if and only if it holds for $\uline{\gamma_{X''}} ,\uline{X_{\Sigma}} ''$.

For $X''\in \Ft''(F)$, by Lemma \ref{orbit 1}, there exist $u(X'')\in U_x(F)$ and $t(X'')\in H_x(F)$ such that
$$\uline{X_{\Sigma}}''=u(X'')^{-1}t(X'')^{-1}X_{\Sigma}''t(X'')u(X'').$$
By the choice of $X_{\Sigma}''$, we have $t(X'')^{-1}X_{\Sigma}''t(X'')\in \Xi+\Lambda$. It follows that $u(X'')$ and $t(X'')^{-1}X_{\Sigma}''t(X'')$ can be expressed in terms of polynomials of $\uline{X_{\Sigma}}''$. Hence they are bounded. By Lemma \ref{major 2}, we know
$$\sigma(t(X''))\ll 1+\mid \log\mid Q_T(X'')\mid_F \mid.$$
So for $X''\in (\Ft'')^0(F)[>N^{-b}]\cap \omega_{T''}$, we have $\sigma(t(X''))\ll \log(N)$.

Note that the conjugations of $X''$ by $\uline{\gamma_{X''}} $ and by $\gamma_{X''} t(X'')u(X'')$ are same. Since $X''$ is regular, there exists $y(X'')\in T(F)$ such that $\uline{\gamma_{X''}} =y(X'')\gamma_{X''} t(X'')u(X'')$. The majorization of $\uline{\gamma_{X''}}$,
$\gamma_{X''},t(X'')$, and $u(X'')$ implies that $\sigma(y(X''))\ll \log(N)$ for $X\in \Ft_0(F)[>N^{-b}] \cap \omega_T$.
Let $c>0, a,g,\bar{u},X''$ as in the statement of lemma. Since $\kappa_N$ is left $H(F)U(F)$-invariant, we have
$$\kappa_N((\uline{\gamma_{X''}} )^{-1} a\bar{u}g)=\kappa_N( \gamma_{X''}^{-1}a\bar{u}'g'),\kappa_N((\uline{\gamma_{X''}} )^{-1}ag)=\kappa_N( \gamma_{X''}^{-1}ag')$$
where $g'=y(X'')^{-1}g$ and $\bar{u}'=y(X'')^{-1}\bar{u}y(X'')$.

Now suppose that the Lemma holds for $\gamma_{X''}, X_{\Sigma}''$. By the above discussion, there exists $c''>0$ such that $\sigma(g'),\sigma(\bar{u}'),\sigma(\bar{u}'g') <c''\log(N)$ for $g$ and $\bar{u}$ as in the lemma. Let $c'$ be the $c'$ associated to $c=c''$ for $\gamma_{X''}$ and $X_{\Sigma}''$. This $c'$ is what we need for $\uline{\gamma_{X''}}$ and $\uline{X_{\Sigma}} ''$. The reverse direction is similar. This proves (4).

We go back to the proof of the lemma. We only deal with the case when $x$ is in the center, the other cases followed by the same method and the calculation is much easier. In this case, $X=X''$.  We replace $X''$ by $X$ for the rest of the proof. Since $T$ is split, $M_{\sharp}=T$. May choose $P_{\sharp}=M_{\sharp}N_{\sharp}\in \CP(M_{\sharp})$ and only consider those $a\in A_T(F)$ with $H_{M_{\sharp}}(a) \in CL(\Fa_{P_{\sharp}}^{+})$. Then we must have $P_{\sharp}\subset Q$. By conjugating by a Weyl element $w$, we may assume that $P_{\sharp}\subset \bar{P}$ is the lower Borel subgroup. Note that when we conjugate it by $w$, we just need to change $X\rightarrow wXw^{-1}$, $\gamma_X\rightarrow w\gamma_X$, $a\rightarrow waw^{-1}$, $\bar{u}\rightarrow w\bar{u}w^{-1}$ and $g\rightarrow wg$. This is allowable by (4). We note that although in (3) we reduce to the case where $T$ split, it still matters whether we are starting from the split case or the nonsplit case since the definition of $\kappa_N$ really depends on it. If we are in the nonsplit case, we can make $\bar{P}\subset Q$ since $\bar{P}$ is the minimal parabolic subgroup in this case; but this is not possible in the split case since $\bar{P}$ will no longer be the minimal parabolic subgroup.

For $X=\diag(x_1,x_2,x_3,x_4,x_5,x_6)\in \Ft_{0,reg}(F)$, if $\mid x_2-x_1\mid_F \geq \max\{\mid x_3-x_4\mid_F, \mid x_5-x_6 \mid_F\}$, define
$$X_{\Sigma}'=\begin{pmatrix} X_1 & 0 & 0 \\ aI_2 & X_2 & 0 \\ 0 & bI_2 & X_3 \end{pmatrix}$$
where we define
$X_1=\begin{pmatrix} x_1 & 0 \\ 0 & x_2 \end{pmatrix}$, $X_2=\begin{pmatrix} x_3+m & 1 \\ -m^2+Bm & x_4-m \end{pmatrix}$, and
$X_3=\begin{pmatrix} x_5+n & -n^2+Cn \\ 1 & x_6-m \end{pmatrix}$ with
$$m=\frac{A+B+C}{2}\cdot\frac{A+B-C}{2A}, \ \
n=\frac{A+B+C}{2}\cdot\frac{A+C-B}{2A},$$
where $A=x_2-x_1$, $B=x_4-x_3$, and $C=x_6-x_5$.
Then the map $X\rightarrow X_{\Sigma}'$ satisfies condition (1). (Note that we assume $\mid A\mid \geq \max\{\mid B\mid,\mid C\mid \}$.)
We can find an element $p_X \in \bar{P}$ of the form $p_X=\bar{u}_X m_X$ such that $p_X X_{\Sigma}' p_{X}^{-1}=X$ where
$$m_X=\begin{pmatrix} m_1 & 0 & 0 \\ 0 & m_2 & 0 \\ 0 & 0 & m_3 \end{pmatrix} \in M, \bar{u}_X\in \bar{U}.$$
It follows that $m_X \diag(X_1,X_2,X_3) m_{X}^{-1}=X$. So we can choose
$$m_1=I_2, m_{2}^{-1}=\begin{pmatrix} 1 & 1 \\ -m & B-m \end{pmatrix}, m_{3}^{-1}=\begin{pmatrix} 1 & 1 \\ -n & C-n \end{pmatrix}.$$
Similarly, we can define $m_X$ and $X_{\Sigma}'$ for the case when $\mid x_3-x_4\mid_F \geq \max\{\mid x_1-x_2\mid_F, \mid x_5-x_6 \mid_F\}$ or $\mid x_5-x_6\mid_F \geq \max\{\mid x_1-x_2\mid_F, \mid x_3-x_4 \mid_F\}$.

Now by adding polynomials $x_1-x_2$, $x_3-x_4$ and $x_5-x_6$ into the set $\CQ_T$, for any $X\in \omega_T\cap \Ft^0(F)[>N^{-b}]$,
we have $\sigma(m)\ll \log(N)$. Applying Proposition \ref{h-c} again, we know that $p_X,X_{\Sigma}'$ satisfy Conditions (1) and (2). In fact, here we know that $\sigma_T(p_X)\ll \log(N)$ and $\sigma(m_X)\ll \log(N)$ for $X\in \omega_T\cap \Ft^0(F)[>N^{-b}]$, this forces $\sigma(\bar{u}_X)\ll \log(N)$. Now by (4), it is enough to prove this Lemma for $p_X,X_{\Sigma}'$.

We will only deal with the case when $\mid x_2-x_1\mid_F \geq \max\{\mid x_3-x_4\mid_F, \mid x_5-x_6 \mid_F\}$, the rest cases follow from a similar calculation. Applying the Bruhat decomposition, we have
$$m_{2}^{-1}=\left( \begin{array}{cc} 1 & 1 \\ 0 & B-m \end{array} \right) \left( \begin{array}{cc} 1 & 0 \\ \frac{m}{m-B} & 1 \end{array} \right)=b_{X,2} w_{X,2}.$$
Similarly we can decompose $m_3$ and $m_1$ in this way. Let
$$b_X=\diag(b_{X,1},b_{X,2},b_{X,3}), w_X=\diag(w_{X,1},w_{X,2},w_{X,3}).$$
By adding some more polynomials on $\CQ_T$, we may still assume that $\sigma(w_X)\ll \log(N)$. (Note $\frac{m}{m-B}$ and $\frac{n}{n-C}$ are rational functions of the $x_i$'s.) It follows that $\sigma(b_X)\ll \log(N)$. Now we can write
$$p_{X}^{-1}=b_X w_X (\bar{u}_X)^{-1}=b_X v_X$$
for some $v_X=w_X (\bar{u}_X)^{-1}\in U_{\sharp}(F)$, and still have $\sigma(b_X),\sigma(v_X)\ll \log(N)$. Since $P_{\sharp}\subset Q$, we can write $v_{X}=n_X u_X$ where $n_X\in U_{\sharp}(F)\cap L(F)$ and $u_X\in U_Q(F)$. Then we have
\begin{eqnarray*}
v_{X}a\bar{u}g&=&n_X u_X a\bar{u}g=n_Xa\bar{u}g \cdot(g^{-1}\bar{u}^{-1} a^{-1} u_X a\bar{u}g)\\
&=&a((a^{-1}n_X a)^{-1}\bar{u}(a^{-1}n_Xa))\cdot (a^{-1}n_{X}^{-1} ag)\cdot(g^{-1}\bar{u}^{-1}a^{-1} u_X a\bar{u}g)\\
&=&a\bar{u}' g' k.
\end{eqnarray*}

For all $a\in A_T(F)$ with $\inf\{\alpha(H_{M_{\sharp}}(a))\mid \alpha\in \Sigma_{Q}^{+} \}>c_4 \log(N)$ for some $c_4>0$ large, $a^{-1}u_Xa-1$ is very close to zero. Hence we can make
$$k=g^{-1}\bar{u}^{-1}a^{-1} u_X a\bar{u}g\in K$$
for all $\sigma(g),\sigma(\bar{u})< c\log(N)$. Since $\kappa_N$ is right $K$-invariant, we have
\begin{equation}\label{10.3}
\begin{split}
\kappa_N(p_{X}^{-1} a\bar{u}g)&=\kappa_N(b_X v_{X}a\bar{u}g)=\kappa_N(b_X a\bar{u}'g'),\\
\kappa_N(p_{X}^{-1} ag)&=\kappa_N(b_X v_{X}ag)=\kappa_N(b_X ag').
\end{split}
\end{equation}
Also since $H_{M_{\sharp}}(a)\in CL(\Fa_{P_{\sharp}}^{+})$, $a^{-1}n_Xa$ is contraction of $n_X$, and hence we still have $\sigma(\bar{u}'),\sigma(g')\ll \log(N)$.

\textbf{If we are in the non-split case}, then we have already make $\bar{P}\subset Q$,
and hence $U_{\bar{Q}}\subset U$. So the $\bar{u}'$ of the first equation in \eqref{10.3} can be moved to the
very left via the $a$-conjugation and the $b_X$-conjugation. Then we can eliminate
it by using left U-invariance property of $\kappa_N$. This proves the Lemma.

\textbf{If we are in the split case}, we may assume that $\bar{u}'\in U_{\bar{Q}}(F)\cap M(F)$ since the rest part can be switched to the front via the $a$-conjugation and the $b_X$-conjugation, and then be eliminated by the left $U$-invariance property of $\kappa_N$. Let $g'=u'm'k'$ be the Iwasawa decomposition with $u'\in U(F), m'\in M(F)$ and $k'\in K$. Then $\sigma(m')\leq c_0\log(N)$ for $c_0=lc$ where $l$ is a fixed constant only depends on $G$. (Here we use the fact that the Iwasawa decomposition preserves the norm up to a bounded constant which only depends on the group and the parabolic subgroup.) We can eliminate $u'$ and $k'$ by the left $U$-invariance and right $K$-invariance property of $\kappa_N$. Now applying the Iwasawa decomposition again, we can write $m'=b'k'$ with $b'$ upper triangle. By the same reason, we have $\sigma(b')\leq c_1\log(N)$ for some $c_1=l' c_0=ll'c$. Again by the right $K$-invariance we can eliminate $k'$. $b'$ can be absorbed by $a$ and $\bar{u}'$. After this process, we will still have the majorization for $\bar{u}'$ (i.e. $\sigma(\bar{u}')\ll \log(N)$), and we will still have $\alpha(H_{M_{\sharp}}(a))>c''\log(N)$ for all $\alpha\in \Sigma_{Q}^{+}$, here $c''=c'-c_1$. So we may assume $m'=1$. In this case, we have
$$\kappa_N(b_X ag')=\kappa_N(b_X a),\kappa_N(a\bar{u}' g')=\kappa_N(b_X a\bar{u}').$$

Now let $b_X a=\diag(l_1,l_2,l_3)$ and $b_X a \bar{u}'=\diag(l_1 ',l_2 ',l_3')$ where $l_i$ and $l_i '$ are all upper triangle $2$-by-$2$ matrices. Since $\bar{u}'$ is an unipotent element and $\sigma(\bar{u}')\ll \log(N)$, $l_i'=l_i n_i$ for some unipotent element $n_i$ with $\sigma(n_i)\ll \log(N)$. Then we know for any $1\leq i,j \leq 3$, $l_{i}^{-1} l_j=\begin{pmatrix} a & x \\ 0 & c \end{pmatrix}$, and $(l_i ')^{-1} l_j '=n_{i}^{-1} \begin{pmatrix} a & x \\ 0 & c \end{pmatrix} n_j$. Since in the definition of $\kappa_N$ for the split case (as in \eqref{kappa split}), we do allow the unipotent part to be bounded by $(1+\epsilon)N$ while the diagonal part is bounded by $N$. Now those $n_i$'s will only add something majorized by $N+C \log(N)$ on the unipotent part and not change the semisimple part. So if we take $N$ large so that $\epsilon N> C\log(N)$, we have
$$\kappa_N(b_X a\bar{u}' g')=\kappa_N(b_X ag').$$
This finishes the proof of the split case, and finishes the proof of the Lemma.
\end{proof}

{\bf We prove 9.3(6)}.

For $c_4>0$, by 10.3(1), we impose the mirror condition
$$c_4\log(N)<c_1\inf\{\alpha(Z_{P_{min}})\mid \alpha\in \Delta_{min} \}$$
to $Z_{P_{min}}$ to make sure all terms are well defined.

By the same argument as in Section 9.4, we know the function $\zeta\rightarrow \sigma_{M_{\sharp}}^{Q}(\zeta,\CZ(g))  \tau_Q(\zeta-\CZ(g)_Q)$ is invariant under $g\rightarrow \bar{u}g$. Therefore
\begin{equation}\label{10.4}
\kappa_{N,X''}(Q,\bar{u}g)-\kappa_{N,X''}(Q,g)
\end{equation}
$$=\int_{Z_G\backslash A_T(F)} \sigma_{M_{\sharp}}^{Q}(H_{M_{\sharp}}(a),\CZ(g))  \tau_Q(H_{M_{\sharp}}(a)-\CZ(g)_Q) (\kappa_N(\gamma_{X''}^{-1} a\bar{u}g)-\kappa_N(\gamma_{X''}^{-1}ag)) du.$$
Let $c=c_4$ as in the Lemma \ref{U-invariant 2}. Then we get some $c'>0$. For $a\in A_T(F)$ with $\sigma_{M_{\sharp}}^{Q}(H_{M_{\sharp}}(a),\CZ(g))  \tau_Q(H_{M_{\sharp}}(a) -\CZ(g)_Q)\neq 0$, by the definition of $\sigma_{M_{\sharp}}^{Q},\; \tau_Q$, and the majorization of $g$, we have
\begin{equation}\label{10.5}
\inf\{\alpha(H_{M_{\sharp}}(a))\mid \alpha\in \Sigma_{Q}^{+}\}-\inf\{ \alpha(Z_{P_{min}})\mid \alpha\in \Delta_{min}\} \gg -\log(N).
\end{equation}
Now choose $c_5>0$ such that $c_5>\frac{c_4}{c_1}$. We also require that the condition
$$\inf\{ \alpha(Z_{P_{min}})\mid \alpha\in \Delta_{min}\}>c_5\log(N)$$
together with \ref{10.5} implies that
$$\inf\{\alpha(H_{M_{\sharp}}(a))\mid \alpha\in \Sigma_{Q}^{+}\}>c'\log(N).
$$
We claim that this is the $c_5$ we need for 9.3(6). In fact, by the discussion above together with Lemma \ref{U-invariant 2}, we know for $g$ and $\bar{u}$ in 10.3(6), $\kappa_N(\gamma_{X''}^{-1}a\bar{u}g)=\kappa_N(\gamma_{X''}^{-1}ag)$ whenever $\sigma_{M_{\sharp}}^{Q}(H_{M_{\sharp}}(a),\CZ(g))  \tau_Q(H_{M_{\sharp}}(a) -\CZ(g)_Q)\neq 0$. This means that the right hand side of \eqref{10.4} equals zero. Hence $\kappa_{N,X''}(Q,\bar{u}g)-\kappa_{N,X''}(Q,g)=0$. This finishes the proof of 9.3(6).

\subsection{The split case}
Now if $x$ is split and not in the center, we are going to show the inner integral of the expression of $I_{x,\omega,N}(f)$ (as in \eqref{10.1}) is zero for $N$ large. In this case, $G_x=GL_3(F)\times GL_{3}(F)$ is a Levi subgroup of $G(F)$, where the first $GL_3(F)$ is the first, third, and fifth rows and columns of $GL_6(F)$, and the second $GL_3(F)$ is the second, fourth, and sixth rows and columns of $GL_6(F)$. Let $S=G_x U_S$ be the parabolic subgroup of $G$ with $U_S$ consists of elements in first, third, and fifth rows, and the second, fourth, and sixth columns. Also we know $Z_{G_x}\subset A_T$ for any $T\in T(G_x)$. By Lemma \ref{I to I*}, we only need to show for $X\in \Ft'(F)\times (\Ft'')^0 \cap \Ft_0(F)[>N^{-b}]$, the inner integral of the $I_{x,\omega,N}(f)$ equals $0$, i.e.
\begin{equation}\label{split 1}
\int_{A_T(F)\backslash G(F)} {}^g f^{\sharp}_{x,\omega}(X'+X'') \kappa_{N,X''}(g) dg =0.
\end{equation}
The idea to prove \eqref{split 1} is that we still want to prove that the truncated function $\kappa_{N,X''}(g)$ is invariant under some unipotent subgroups, and then prove the vanishing result by
applying Lemma \ref{strongly cuspidal}. \textbf{The difference between this case and the case in Section 9.3 is that we not just
prove $I(Q,X)=0$ for proper $Q$, also we show $I(G_x,X)=0$ because $G_x$ itself gives us some unipotent subgroup $U_S$ of $G$.}

Now for $T\in T(G_x)$, let $M_{\sharp}$ be the centralizer of $A_T$ in $G_x$. It is a Levi subgroup of $G_x$, and it is easy to check $A_T=A_{M_{\sharp}}$. In this section, we temporality denote $\CP(M)$ to be the set of parabolic subgroups of $G_x$ with Levi subgroup $M$ (i.e. use $\CP(M)$ instead of $\CP^{G_x}(M)$), and the same for $\CF(M)$ and $\CL(M)$. Let $\CY=(Y_{P_{\sharp}})_{P_{\sharp}\in \CP(M_{\sharp})}$ be a family of elements in $\Fa_{M_{\sharp}}$ that are $(G_x,M_{\sharp})$-orthogonal and positive. For $Q=LU_Q\in \CF(M_{\sharp})$, we can still define the functions $\zeta\rightarrow \sigma_{M_{\sharp}}^{Q}(\zeta,\CY)$ and $\tau_Q(\zeta)$ on $\Fa_{M_{\sharp}}$ as in Section 10.2. Applying the work of Arthur again (\cite{Ar91}), we have a similar result as Proposition \ref{truncation}.

\begin{prop}
The function
$$\zeta\rightarrow \sigma_{M_{\sharp}}^{Q}(\zeta,\CY) \tau_Q(\zeta-Y_Q)$$
is the characteristic function on $\Fa_{M_{\sharp}}$, whose support is on the sum of $\Fa_{Q}^{+}$ and the convex envelop generated by the family $(Y_{P_{\sharp}})_{ P_{\sharp}\in \CP(M_{\sharp}),P_{\sharp}\subset Q}$.
Moreover, for every $\zeta\in \Fa_{M_{\sharp}}$, the following equation holds.
$$\Sigma_{Q\in \CF(M_{\sharp})} \sigma_{M_{\sharp}}^{Q}(\zeta,\CY) \tau_Q(\zeta-Y_Q)=1.$$
\end{prop}

Now fix a minimal Levi subgroup $M_{min}$ of $G_x$ contained in $M_{\sharp}$, a hyperspecial subgroup $K_{min}$ of $G_x$ related to $M_{min}$ and $P_{min}=M_{min}U_{min}\in \CP(M_{min})$. Let $\Delta_{min}$ be the set of simple roots of $A_{M_{min}}$ in $\Fu_{min}$.

Given $Y_{min}\in \Fa_{P_{min}}^{+}$, for any $P'\in \CP(M_{min})$, there exists a unique $w\in W(G_x,M_{min})$ such that $wP_{min}w^{-1}=P'$. Set $Y_{P'}=wY_{P_{min}}$. The family $(Y_{P'})_{P'\in \CP(M_{min})}$ is $(G_x,M_{min})$-orthogonal and positive. For $g\in G(F)$, define $\CY(g)=(Y(g)_Q)_{Q\in \CP(M_{\sharp})}$ to be
$$Y(g)_Q=Y_Q-H_{\bar{Q}}(g_x),$$
where $g=ug_x k$ is the Iwasawa decomposition with respect to the parabolic subgroup $S=U_S G_x$. In particular, this function is left $U_S$-invariant.

As same as 9.3(1) and 9.3(2), we have

{\bf (1)} There exists $c_1>0$ such that for any $g\in G(F)$ with $\sigma(g)< c_1 \inf\{\alpha(Y_{P_{min}}); \alpha\in \Delta_{min}\}$, the family $\CY(g)$ is $(G_x,M_{\sharp})$-orthogonal and positive. And $Y(g)_Q\in \Fa_{Q}^{+}$ for all $Q\in \CF(M_{\sharp})$.
\\

{\bf (2)} There exist $c_2>0$ and a compact subset $\omega_T$ of $\Ft_0(F)$ satisfying the following condition: If $g\in G(F)$, and
$$X\in \Ft_0(F)[>N^{-b}]\cap (\Ft'(F)\times (\Ft'')^0(F))$$
with $({}^g f_{x,\omega})^{\hat{}}(X)\neq 0$, then $X\in \omega_T$ and $\sigma_T(g)< c_2\log(N)$.
\\

Now, fix $\omega_T$ and $c_2$ as in (2). We may assume that $\omega_{T}=\omega_{T'}\times \omega_{T''}$ where $\omega_{T'}$ is a compact subset of $\Ft'(F)$ and $\omega_{T''}$ is a compact subset of $\Ft''(F)$. Assume that
\begin{equation}\label{split condition 1}
c_2\log(N)< c_1\inf\{\alpha(Y_{min})\mid \alpha\in \Delta_{min}\}.
\end{equation}
Here $c_1$ comes from (1). For $Q\in \CF(M_{\sharp})$ and $g\in G(F)$ with $\sigma_T(g)<c_2 \log(N)$, let
\begin{equation}
\kappa_{N,X''}(Q,g)=\int_{Z_{G_x}\backslash A_T(F)} \kappa_N(\gamma_{X''}^{-1} ag) \sigma_{M_{\sharp}}^{Q}(H_{M_{\sharp}}(a), \CY(g))\tau_Q(H_{M_{\sharp}}(a)-\CY(g)_Q) da.
\end{equation}
Then same as 9.3(3), we have

{\bf (3)} The function $g\rightarrow \kappa_{N,X''}(Q,g)$ is left $A_T(F)$-invariant.
\\

It follows that for $X\in \Ft'(F)\times (\Ft'')^0(F)$, we have
\begin{equation}\label{I part 4}
\int_{ A_T(F)\backslash G(F)} {}^g f^{\sharp}_{x,\omega}(X) \kappa_{N,X''}(g) dg=\nu(A_T)\Sigma_{Q\in \CF(M_{\sharp})} I(Q,X)
\end{equation}
where
\begin{equation}\label{split 2}
I(Q,X)=\int_{ A_T(F)\backslash G(F)} {}^g f^{\sharp}_{x,\omega}(X) \kappa_{N,X''}(Q,g) dg.
\end{equation}
So it is enough to show that for all $Q\in \CF(M_{\sharp})$, $I(Q,X)=0$. Suppose
\begin{equation}\label{split condition 2}
\sup\{\alpha(Y_{min})\mid \alpha\in \Delta_{min}\} \leq (\log(N))^2.
\end{equation}

\textbf{Firstly, we consider the case when $Q=G_x$}. We can write the integral \eqref{split 2} as
\begin{equation}\label{split 4}
\int_{K} \int_{Z_{G_x} A_T(F)\backslash G_x(F)} \int_{U_S(F)} {}^{ulk} f^{\sharp}_{x,\omega}(X) \kappa_{N,X''}(G_x, ulk) du\delta_R(l) dl dk.
\end{equation}
By the same argument as (7) of Proposition \ref{change truncation}, we only need to prove the function $\kappa_{N,X''}(G_x,g)$ is left
$U_S$-invariant. We use the same method as in Section 9.5.

\begin{lem}
For $c>0$, there exists $N_1>0$ satisfying the following condition: For given $a\in A_T(F), g\in G(F), u\in U_S(F)$ and $X''\in \omega_{T''}\cap (\Ft'')^0(F)[>N^{-b}]$, assume that $\sigma(u), \sigma(g),\sigma(ug) \leq c\log(N)$ and
$$
\sigma_{M_{\sharp}}^{G_x}(H_{M_{\sharp}}(a), \CY(g)_{G_x})\tau_{G_x}(H_{M_{\sharp}}(a)-\CY(g)_{G_x}) =1.
$$
Then for $N>N_1$, we have
\begin{equation}\label{split 3}
\kappa_{N}(\gamma_{X''}^{-1} aug)=\kappa_N(\gamma_{X''}^{-1} ag).
\end{equation}
\end{lem}

\begin{proof}
The condition imposed on $a$ and the inequality \eqref{split condition 2} tell us up to an element in $Z_{G_x}$, $\sigma(a)\leq (\log(N))^2$. Since both side of the equation \eqref{split 3} are left $Z_{G_x}$-invariant, we may assume that $\sigma(a)\leq (\log(N))^2$. Then the condition imposed on $X'', u$ and $g$ implies
$$\sigma(\gamma_{X''}^{-1} aug),\sigma(\gamma_{X''}^{-1} ag) \ll (\log(N))^2.$$
Therefore for $N$ large, both sides of \eqref{split 3} equal $1$, and hence they are equal. This finishes the proof of the Lemma.
\end{proof}

Now by the previous Lemma, we can have a Corollary which is an analogy of (6) of Proposition \ref{change truncation}. The proof is the same, except here we need to add the fact that the function
$$\sigma_{M_{\sharp}}^{G_x}(H_{M_{\sharp}}(a), \CY(g))\tau_{G_x}(H_{M_{\sharp}}(a)-\CY(g)_{G_x})$$
is invariant under the transform $g\rightarrow ug$ for all $u\in U_S(F)$.
\begin{cor}
If $c>0$, there exists $N_1>0$ such that if $N>N_1$, we have $\kappa_{N,X''}(G_x,ug)=\kappa_{N,X''}(G_x,g)$ for all $g\in G(F)$ and $u\in U_S(F)$ with $\sigma(u), \sigma(ug),\sigma(g) \leq c\log(N)$, and for all $X''\in \omega_{T''}\cap (\Ft'')^0(F)[>N^{-b}]$ .
\end{cor}
Finally, by applying the same argument as in (7) of Proposition \ref{change truncation}, we know that for N large, the inner integral of \eqref{split 4} is zero, and therefore
\begin{equation}
I(G_x,X)=0.
\end{equation}

For $Q\neq G_x$, we prove the following lemma.

\begin{lem}\label{split zero}
For $N$ large, we have $I(Q,X)=0$ for all $Q\neq G_x$ and $X''\in \omega_{T''} \cap (\Ft'')^0(F)[>N^{-b}]$.
\end{lem}

\begin{proof}
The proof is very similar to the case in Section 9.3. We will postpone the proof to Appendix A.2.
\end{proof}

Now the only thing left is to find $Y_{P_{min}}$ satisfies our assumptions \eqref{split condition 1} and \eqref{split condition 2}, which are
$$\sup\{\alpha(Y_{min})\mid \alpha\in \Delta_{min}\} \leq (\log(N))^2$$
and
$$c_1\inf\{\alpha(Y_{min})\mid \alpha\in \Delta_{min}\} \geq c_2\log(N).$$
This is always possible for $N$ large.

To summarize, we have proved the following statement:
\begin{itemize}
\item For $N$ large, and for $X''\in \omega_{T''} \cap (\Ft'')^0(F)[>N^{-b}]$, $I(Q,X)=0$ for all $Q\in \CF(M_{\sharp})$.
\end{itemize}
Combining it with \eqref{I part 4}, we know
\begin{equation}
\int_{ A_T(F)\backslash G(F)} {}^g f^{\sharp}_{x,\omega}(X) \kappa_{N,X''}(g) dg=0
\end{equation}
for $N$ large. Then combining it with \eqref{10.1} and (2), we know
\begin{equation}
I_{x,\omega,N}(f)=0
\end{equation}
for $N$ large. This finishes the proof for the split case.

\subsection{Principal proposition}
\begin{prop}\label{10 final}
There exists $N_1>0$ such that for $N>N_1$, $X\in \Ft_0(F)[>N^{-b}]$, and $x\in H_{ss}(F)$ non-split or belongs to the center, we have
$$\int_{A_T(F) Z_{G_x}(F)\backslash G(F)} ({}^g f_{x,\omega})^{\hat{}}(X) \kappa_{N,X''}(g) dg= \nu(A_T) \nu(Z_{G_x}) \theta_{f,x,\omega}^{\sharp}(X).$$
\end{prop}

\begin{proof}
By Proposition \ref{change truncation}, we can replace the function $\kappa_{N,X''}$ by the function $\tilde{v}(g,Y_{P_{min}})$ in the integral above. Then by the computation of $\tilde{v}(g,Y_{P_{min}})$ in \cite{Ar91}, together with the same argument as in Proposition 10.9 of \cite{W10}, as $Y_{P_{min}}$ goes to infinity, the integral equals
\begin{equation}\label{10 final 2}
(-1)^{a_{M_{\sharp}}-a_G} \Sigma_{Q\in \CF(M_{\sharp})} c_Q ' I(Q)
\end{equation}
where $c_Q'$ are some constant numbers with $c_G '=1$, and
\begin{equation}\label{10 final 1}
I(Q)=\int_{Z_{G_x}(F) A_T(F)\backslash G(F)} {}^g f_{x,\omega}^{\sharp} (X) v_{M_{\sharp}}^{Q}(g) dg.
\end{equation}

If $Q=LU_Q\neq G$, we can decompose the integral in \eqref{10 final 1} as $\int_{Z_{G_x}(F)A_T(F)\backslash L(F)}, \int_{K_{min}}$ and $\int_{U_Q(F)}$. Since $v_{M_{\sharp}}^{Q}(g)$ is $U_Q(F)$-invariant, the inner integral becomes
$$\int_{U_Q(F)} {}^{ulk} f_{x,\omega}^{\sharp}(X)du.$$
By Lemma \ref{strongly cuspidal}, this is zero because $f$ is strongly cuspidal. Therefore
\begin{equation}\label{10 final 3}
I(Q)=0.
\end{equation}

For $Q=G$, we can replace the integral on $Z_{G_x}(F) A_T(F)\backslash G(F)$ by $T(F) \backslash G(F)$ and multiply it by $meas(T(F)/Z_{G_x} A_T(F))$. Then we get
\begin{equation}\label{10 final 4}
I(G)=meas(T(F)/Z_{G_x} A_T(F)) D^{G_x}(X)^{1/2} J_{M_{\sharp},x,\omega}^{\sharp}(X,f)
\end{equation}
where $J_{M_{\sharp},x,\omega}^{\sharp}(X,f)$ is defined in \eqref{6.1}.

Now combining \eqref{10 final 2}, \eqref{10 final 3} and \eqref{10 final 4}, together with the definition of $\theta_{f,x,\omega}^{\sharp}$ (as in \eqref{6.2}) and the fact that
$$\nu(T)meas(T(F)/Z_{G_x} A_T(F))=\nu(A_T)\nu(Z_{G_x}),
$$
we have
$$\int_{A_T(F) Z_{G_x}\backslash G(F)} ({}^g f_{x,\omega})^{\hat{}}(X) \kappa_{N,X''}(g) dg= \nu(A_T) \nu(Z_{G_x})\theta_{f,x,\omega}^{\sharp}(X).$$
This finishes the proof of the Proposition.
\end{proof}

Finally, for $x\in H_{ss}(F)$ non-split or belongs to the center, let
\begin{equation}\label{10 final 5}
\begin{split}
J_{x,\omega}(f)=&\Sigma_{T\in T(G_x)} \mid W(G_x,T)\mid^{-1} \nu(Z_{G_x})\\
&\times \int_{\Ft'(F)\times (\Ft'')^0(F)} D^{G_x}(X'')^{1/2} \theta_{f,x,\omega}^{\sharp}(X) dX.
\end{split}
\end{equation}
If $x\in H_{ss}(F)$ is split and not contained in the center, let
$$J_{x,\omega}(f)=0.$$
\begin{prop}\label{local theorem}
The integral in \eqref{10 final 5} is absolutely convergent, and
$$\lim_{N\rightarrow \infty} I_{x,\omega,N}(f)=J_{x,\omega}(\theta,f).$$
\end{prop}

\begin{proof}
The proof for the first part is the same as Lemma 10.10(1) in \cite{W10}. For the second part, if $x\in H_{ss}(F)$ is non-split or belongs to the center, by Lemma \ref{I to I*}, it is enough to consider $\lim_{N\rightarrow \infty} I_{x,\omega,N}^{\ast}(\theta,f)$. Then just use
Proposition \ref{10 final} together with \eqref{10.1}. For the case when $x$ is split and not contained in the center, applying Lemma \ref{I to I*} again, it is enough to consider $\lim_{N\rightarrow \infty} I_{x,\omega,N}^{\ast}(\theta,f)$. Then by Section 9.6, we know the limit is zero.
\end{proof}

\section{Proof of Theorem \ref{main 1} and Theorem \ref{main 2}}\label{prove of theorem}
\subsection{Calculation of $\lim_{N\rightarrow \infty} I_N(f)$: the Lie algebra case}
If $f\in C_{c}^{\infty}(\Fg_0(F))$ is a strongly cuspidal function, we define
\begin{equation}\label{11.1}
J(f)=\Sigma_{T\in T(G)} \mid W(G,T)\mid^{-1} \int_{\Ft^0(F)} D^G(X)^{1/2} \hat{\theta}_f(X) dX.
\end{equation}

\begin{lem}\label{lim I}
The integral in \eqref{11.1} is absolutely convergent and
$$\lim_{N\rightarrow \infty}I_N(f)=J(f).$$
\end{lem}

\begin{proof}
The first part is similar to the first part of Proposition \ref{local theorem}. For the second part, let $\omega \subset \Fg_0(F)$ be a good neighborhood of 0. Suppose that $Supp(f)\subset \omega$. Then we can relate $f$ to a function $\Phi$ on $Z_G(F)\backslash G(F)$, supported on $Z_G(F) \exp(\omega)$. By Proposition \ref{localization}, we know $I_N(f)=I_N(\Phi)$. Then by Proposition \ref{local theorem}, applying to the function $\Phi$ and $x=1$, we have $\lim_{N\rightarrow \infty} I_N(f)=J_{1,\omega}(\Phi)$. By Proposition \ref{partial FT}, $\theta_{\Phi,1,\omega}^{\sharp}$ is the partial Fourier transform of $\theta_{\Phi,1, \omega}=\theta_f$. But for $x=1$, partial Fourier transform is just the full Fourier transform. Thus $\theta_{\Phi,1,\omega}^{\sharp}=\hat{\theta}_f$. Also we know that $\nu(Z_{G_x})=\nu(Z_G)=1$, and therefore
$$\lim_{N\rightarrow \infty} I_N(f)=J_{1,\omega}(\Phi)=J(f).$$
This proves the Lemma for those $f$ whose support is contained in $\omega$.

In general, replacing $(a,b)$ in the definition of $\xi$ (as in \eqref{xi}) by $(\lambda a,\lambda b)$ for some $\lambda\in F^{\times}$, we get a new character $\xi'$, and let $f'=f^{\lambda}$. Then for $Y\in \Fh(F)$, we have
$$(f')^{\xi '}(Y)=\mid \lambda \mid_{F}^{-\dim(U)} f^{\xi}(\lambda Y).$$
This implies
\begin{equation}\label{11.2}
I_{\xi ',N}(f')=\mid \lambda \mid_{F}^{-\dim(U)-\dim(H/Z_H)} I_{\xi,N}(f).
\end{equation}
On the other hand, we know
\begin{eqnarray*}
\hat{\theta}_{f'}(X)&=&\mid \lambda \mid_{F}^{-\dim(G/Z_G)} \hat{\theta}_f(\lambda^{-1} X),\\
D^G(\lambda X)^{1/2}&=&\mid \lambda \mid_{F}^{\delta(G)/2} D^G(X)^{1/2},
 \end{eqnarray*}
 and $\Ft^0(F)$ will not change under this transform. By changing of variable in \eqref{11.1}, (Note that this is allowable since $\Ft^0(F)$ is invariant under scalar in the sense that for $t\in \Ft^0(F), \lambda\in F^{\times}$, we have $\lambda t\in \Ft^0(F)$, see Remark \ref{orbit 2}) we have
\begin{equation}\label{11.3}
J_{\xi '}(f')=\mid \lambda \mid_{F}^{-\dim(G/Z_G)+\dim(T/Z_G)+\delta(G)/2} J_{\xi}(f).
\end{equation}
Because
$$-\dim(G/Z_G)+\dim(T/Z_G)+\delta(G)/2=-\dim(U)-\dim(H/Z_H)=-15,
$$
together with \eqref{11.2} and \eqref{11.3}, we know $\lim_{N\rightarrow \infty} I_{\xi,N}(f)=J_{\xi}(f)$ if and only if $\lim_{N\rightarrow \infty} I_{\xi ',N}(f')=J_{\xi '}(f')$. Then for any $f$, we can choose $\lambda$ such that $Supp(f')\subset \omega$. Applying the first part of the proof to $f'$, we get $\lim_{N\rightarrow \infty} I_{\xi ',N}(f')=J_{\xi '}(f')$, which implies $\lim_{N\rightarrow \infty} I_{\xi,N}(f)=J_{\xi}(f)$. This finishes the proof of the Lemma.
\end{proof}

\subsection{A Premier Result}
During this section, consider the following hypothesis.
\\

\textbf{Hypothesis}: For every strongly cuspidal $f\in C_{c}^{\infty} (\Fg_0(F))$ whose support dose not contain any nilpotent element, we have
$$\lim_{N\rightarrow \infty} I_N(f)=I(f).$$

In this section, we will prove the following proposition.
\begin{prop}\label{premier}
If the above hypothesis holds, we have
$$\lim_{N\rightarrow \infty} I_N(f)=I(f)$$
for every strongly cuspidal $f\in C_{c}^{\infty} (\Fg_0(F))$.
\end{prop}

In order to prove the above proposition, consider the following morphism:
\begin{equation}\label{I-J}
f\rightarrow E(f)=\lim_{N\rightarrow \infty} I_N(f)-I(f)=J(f)-I(f)
\end{equation}
defined on the space of strongly cuspidal functions $f\in C_{c}^{\infty} (\Fg_0(F))$. This morphism is obviously linear.

\begin{lem}\label{I-J 1}
The map $E$ is a scalar multiple of the morphism $f\rightarrow c_{\theta_f,\CO}$ where $\CO$ is the regular nilpotent orbit of $\Fg(F)$. In particular, $E=0$ if $G=GL_3(D)$.
\end{lem}

\begin{proof}
We first prove:

{\bf (1)} $E(f)=0$ if $c_{\theta_f,\CO}=0$ for every $\CO\in Nil(\Fg(F))$.
\\

Suppose that $c_{\theta_f,\CO}=0$ for every $\CO\in Nil(\Fg(F))$. We can find a $G$-domain $\omega$ in $\Fg_0(F)$, which has compact support modulo conjugation and contains 0, such that $\theta_f(X)=0$ for every $X\in \omega$. Let $f'=f 1_{\omega}$ and $f''=f-f'$. Then these two functions are also strongly cuspidal. The support of $f''$ does not contain nilpotent elements, and therefore, by the hypothesis we know $E(f'')=0$.

On the other hand, since $\theta_{f}(X)=0$ for every $X\in \omega$, we must have that $\theta_{f'}=0$, and hence $\hat{\theta}_{f'}=0$. By the definition of $I(f)$ and $J(f)$, we know $J(f')=0=I(f')$. Hence $E(f)=E(f')+E(f'')=0$. This proves (1).
\\

Now for $\lambda\in (F^{\times})^2$, let $f'=f^{\lambda}$. We know $\theta_{f'}=(\theta_f)^{\lambda}$. For $\CO\in Nil(\Fg(F))$, by \eqref{scalar 3}, we have
\begin{equation}\label{scalar 4}
c_{\theta_{f'},\CO}=\mid \lambda \mid_{F}^{-\dim(\CO)/2} c_{\theta_f,\CO}.
\end{equation}
We then show:

{\bf (2)} $E(f')=\mid \lambda \mid_{F}^{-\delta(G)/2} E(f)=\mid \lambda \mid_{F}^{-15} E(f)$
\\

By \eqref{11.3}, we know
\begin{equation}\label{11.4}
J(f')=\mid \lambda \mid_{F}^{-15} J(f).
\end{equation}
Now for $I(f)$, let $T\in \CT$ as in Section 7.2. The expression for $I(f)$ related to $T$ is
\begin{equation}\label{11.5}
\int_{t_0(F)} c_f(Y) D^H(Y) \Delta(Y) dY.
\end{equation}

If $T=\{1\}$, $\eqref{11.5}=c_f(0)$ and the nilpotent orbit is the unique regular nilpotent orbit of $\Fg(F)$. By \eqref{scalar 3}, we have
$$c_{f'}(0)=\mid \lambda \mid_{F}^{-\delta(G)/2} c_f(0)=\mid \lambda \mid_{F}^{-15} c_f(0).$$

If $T=T_v$ for some $v\in F^{\times}/(F^{\times})^2, v\neq 1$ as in Section 7.2, the nilpotent orbit associated to $c_f$ is the unique regular nilpotent orbit $\CO_v$ of $GL_3(F_v)$, it is of dimension 12. By \eqref{scalar 3} again, we have
$$c_{f'}(X)=\mid \lambda \mid_{F}^{-6} c_f(\lambda X).$$
Moreover, $D^H(\lambda^{-1}X)=\mid \lambda \mid_{F}^{-2}$ since $\dim(\Fh)-\dim(\Fh_x)=2$, and $\Delta(\lambda^{-1}X)=\mid \lambda \mid_{F}^{-6} \Delta(X)$ since $\dim(\Fu)-\dim(\Fu_x)=6$. Therefore by changing variable $X\rightarrow \lambda^{-1}X$ we have
\begin{equation}\label{11.6}
\int_{t_0(F)} c_{f'}(Y) D^H(Y) \Delta(Y) dY=\mid \lambda \mid_{F}^{b} \int_{t_0(F)} c_f(Y) D^H(Y) \Delta(Y) dY
\end{equation}
where $b=-6-2-6-\dim(\Ft_0)=-15$. Combining \eqref{11.5} and \eqref{11.6}, we have
\begin{equation}\label{11.7}
I(f')=\mid \lambda \mid_{F}^{-15} I(f).
\end{equation}
Then (2) just follows from \eqref{11.4} and \eqref{11.7}.
\\

Now (1) tells that $E$ is a linear combination of $c_{\theta_f,\CO}$ for $\CO\in Nil(\Fg(F))$. We know $\dim(\CO)\leq 30$ and the equality holds if and only if $G=GL_6(F)$ and $\CO$ is regular. Hence the Lemma follows from (2) and \eqref{scalar 4}.
\end{proof}

In particular, by the lemma above, we have proved Proposition \ref{premier} for $G=GL_3(D)$. Now we are going to prove the $G=GL_6(F)$ case.
\\

By the discussion above, in this case, $E(f)=c_{reg}c_{\theta_f,\CO_{reg}}$ for some complex number $c_{reg}$. It is enough to show $c_{reg}=0$. Our method is to find some special $f$ such that $E(f)=0$ and $c_{\theta_f,\CO_{reg}}=1$, which implies that $c_{reg}=0$. The way to find this $f$ is due to Waldspurger, see \cite{W10}.

By 6.3(3) and 11.5 of \cite{W10}, for $T\in T(G)$, here $T(G)$ is the set of equivalent classes of maximal subtorus of $G(F)$, and $X\in \Ft_0(F)\cap \Fg_{reg}(F)$, we can construct a neighborhood $\omega_X$ of $X$ in $\Ft_0(F)$ and a strongly cuspidal function $f[X]\in C_{c}^{\infty}(\Fg_0(F))$ satisfy the following conditions:
\begin{enumerate}
\item For $T'\in T(G)$ with $T'\neq T$, the restriction of $\hat{\theta}_{f[X]}$ to $\Ft_0 '(F)$ is zero.
\item For every locally integrable function $\varphi$ on $\Ft_0(F)$ which is invariant under the conjugation of Weyl group, we have
$$\int_{\Ft_0(F)} \varphi(X')D^G(X')^{1/2} \hat{\theta}_{f[X]}(X')dX'=\mid W(G,T)\mid meas(\omega_X)^{-1} \int_{\omega_X} \varphi(X')dX'$$
\item For every $\CO\in Nil(\Fg)$, we have
$$c_{\theta_{f[X]},\CO}=\Gamma_{\CO}(X)$$
where $\Gamma_{\CO}(X)$ is the Shalika germ defined in Section 3.3.
\end{enumerate}

Now let $T_d$ be the unique split torus of $T(G)$. This is possible since we are in the split case now. Fix
$X_d\in \Ft_{d,0}(F)\cap \Fg_{reg}(F)$. Then we can find $\omega_{X_d}$ and $f[X_d]$ as above. Let $f=f[X_d]$. By condition (3) above and Lemma 11.4(i) of \cite{W10}, we know that $c_{\theta_f,\CO_{reg}}=1$, and
\begin{equation}\label{11.8}
E(f)=c_{reg}.
\end{equation}
Now by condition (1) above, we know each components of the summation in $I(f)$ is 0 for $T\in \CT$ with $T\neq \{1\}$. Then by condition (3) we know
\begin{equation}\label{11.9}
I(f)=c_{\theta_f,\CO_{reg}}=1.
\end{equation}
On the other hand, by condition (1) and (2),
\begin{equation}\label{11.10}
\begin{split}
J(f)&=\Sigma_{T\in T(G)} \mid W(G,T)\mid^{-1} \int_{\Ft^0(F)} D^G(X)^{1/2} \hat{\theta}_f(X) dX\\
&=\mid W(G,T_d)\mid^{-1} \int_{\Ft_{d,0}(F)} D^G(X)^{1/2} \hat{\theta}_f(X) dX\\
&=meas(\omega_{X_d})^{-1} meas(\omega_{X_d})=1.
\end{split}
\end{equation}
Here we use the fact $(\Ft_d)^0(F)=\Ft_{d,0,reg}(F)$, which has been proved in the proof of Lemma \ref{U-invariant 2}.

Now combining \eqref{11.8}, \eqref{11.9} and \eqref{11.10}, we have
$$c_{reg}=E(f)=I(f)-J(f)=1-1=0.$$ This finishes the proof of Proposition \ref{premier}.

\subsection{Proof of Theorem \ref{main 1} and Theorem \ref{main 2}}
Consider the following four assertions:
\\
\\
$(th)_G$: For every strongly cuspidal function $f\in C_{c}^{\infty}(Z_G(F)\backslash G(F))$, we have $\lim_{N\rightarrow \infty} I_N(f)=I(f)$.
\\
\\
$(th')_G$: For every strongly cuspidal function $f\in C_{c}^{\infty}(Z_G(F)\backslash G(F))$ whose support does not contain any unipotent element, we have $$\lim_{N\rightarrow \infty} I_N(f)=I(f).$$
\\
\\
$(th)_{\Fg}$: For every strongly cuspidal function $f\in C_{c}^{\infty}(\Fg_0(F))$, we have $\lim_{N\rightarrow \infty} I_N(f)=I(f)$.
\\
\\
$(th')_{\Fg}$: For every strongly cuspidal function $f\in C_{c}^{\infty}(\Fg_0(F))$ whose support does not contain any nilpotent element, we have $\lim_{N\rightarrow \infty} I_N(f)=I(f)$.

\begin{lem}\label{group to Lie}
The assertion $(th)_G$ implies $(th)_{\Fg}$. The assertion $(th')_G$ implies $(th')_{\Fg}$.
\end{lem}

\begin{proof}
Suppose $(th)_G$ holds, for any strongly cuspidal function $f\in C_{c}^{\infty}(\Fg_0(F))$, we need to show $E(f)=0$. In the proof of Lemma \ref{I-J 1}, we have proved that $E(f)=\mid \lambda \mid_{F}^{15} E(f^{\lambda})$. So by changing $f$ to $f^{\lambda}$, we may assume that the support of $f$ is contained in a good neighborhood $\omega$ of 0 in $\Fg_0(F)$. Same as in Lemma \ref{lim I}, we can construct a strongly cuspidal function $F\in C_{c}^{\infty}(Z_G(F)\backslash G(F))$ such that $J(f)=J_{1,\omega}(F)$ and $I(f)=I_{1,\omega}(F)$. By Propositions \ref{localization}, \ref{localization 1}, and \ref{local theorem}, we know $J_{1,\omega}(F)=\lim_{N\rightarrow \infty}I_N(F)$ and $I_{1,\omega}(F)=I(F)$. By $(th)_G$, we know $I(F)=J_{1,\omega}(F)$, which implies $E(f)=0$.

The proof of the second part is similar that of the first part: we only need to add the fact that if the support of $f$ does not contain
any nilpotent element, then the support of $F$ does not contain any unipotent element.
\end{proof}

\textbf{We first prove $(th')_G$}.
\begin{proof}
Let $f\in C_{c}^{\infty}(Z_G(F)\backslash G(F))$ be a strongly cuspidal function whose support contains no unipotent element. For $x\in G_{ss}(F)$, let $\omega_x$ be a good neighborhood of 0 in $\Fg_x(F)$, and $\Omega_x=(x\exp(\omega_x))^G \cdot Z_G$. We require that
$\omega_x$ satisfies the following conditions:
\begin{enumerate}
\item If $x$ belongs to the center, since $f$ is $Z_G(F)$-invariant, we may assume that $x=1$. We require that $\Omega_x \cap Supp(f)=\Omega_1\cap Supp(f)=\emptyset$. This is possible since the support of $f$ contains no unipotent element.
\item If $x$ is not conjugated to any element in $H(F)$, choose $\omega_x$ satisfying the condition in Section 8.1.
\item If $x$ is conjugated to an element $x'\in H(F)$ not in the center, we choose a good neighborhood $\omega_{x'}$ of 0 in $\Fg_{x'}(F)$ as in Section 8.2, and let $\omega_x$ be the image of $\omega_{x'}$ by conjugation. Moreover, if $x'$ is split, we choose $\omega_{x'}$ such that $\Omega_{x'}$ does not contain non-split element, and does not contain the identity element; and if $x'$ is non-split, we choose $\omega_{x'}$ such that $\Omega_{x'}$ does not contain split element.
\end{enumerate}
Then we can choose a finite set $\CX\subset G_{ss}(F)$ such that $f=\Sigma_{x\in \CX} f_x$ where $f_x$ is the product of $f$ and the characteristic function on $\Omega_x$. Since $\lim_{N\rightarrow \infty} I_N(f)$ and $I(f)$ are linear functions on $f$, we may just assume $f=f_x$.
\\

If $x=1$, by the choice of $\Omega_1$ we know $f=0$, and the assertion is trivial.

If $x$ is not conjugated to an element of $H(F)$, then the assertion follows from the choice of $\Omega_x$ and the same argument as in Section 8.1.

If $x$ is conjugated to a split element of $H(F)$, by the choice of $\Omega_x$ and the definition of $I(f)$ we know $I(f)=0$. Also by Section 10.6, we know $\lim_{N\rightarrow \infty} I_N(f)=0$, and the assertion follows.

If $x$ is conjugated to a non-split element of $H$. By Propositions \ref{localization} and \ref{localization 1}, it is enough to prove
\begin{equation}\label{11.3.1}
\lim_{N\rightarrow \infty} I_{x,\omega,N}(f)=I_{x,\omega}(f).
\end{equation}
Now we can decompose $\theta_{f,x,\omega}$ as
\begin{equation}\label{11.3.2}
\theta_{f,x,\omega}(X)=\Sigma_{b\in B} \theta_{f,b}'(X')\theta_{f,b} '' (X'')
\end{equation}
where $B$ is a finite index set, and for every $b\in B$, $\theta_{f,b}'(X')$ (resp. $\theta_{f,b}''(X'')$) is a quasi-character on $\Fg_x'(F)$ (resp. $\Fg''(F)$). By Proposition 6.4 of \cite{W10}, for every $b\in B$, we can find $f_b ''\in C_{c}^{\infty}(\Fg''(F))$ strongly cuspidal such that $\theta_{f,b} '' (X'')=\theta_{f_{b} ''}$. Then by the definition of $I_{x,\omega}(f)$ (as in \eqref{I_x}), we have
$$I_{x,\omega}(f)=\Sigma_{b\in B} I'(b) I(f_b '')$$
where
$$I'(b)=\nu(Z_{G_x}) \int_{\Fg_x '(F)} \theta_{f,b} '(X') dX',\; I(f_b '')=c_{\theta_{f,b}'',\CO}(1)$$
with $\CO$ be the unique regular nilpotent orbit in $\Fg''(F)$. Here we use the fact that the only torus in $\CT_x$ is $Z_{G_x}$. Hence $\nu(T)=\nu(Z_{G_x})$ and $D^{H_x}(X)=\Delta''(X)=1$ for $X\in \Ft_0(F)$ .

On the other hand, by Proposition \ref{local theorem},
$$\lim_{N\rightarrow \infty} I_{x,\omega,N}(f)=J_{x,\omega}(f)=\Sigma_{b\in B} I'(b) J(f_b '')$$
where
$$J(f_b '')=\Sigma_{T\in T(G_x)} \mid W(G_x,T)\mid^{-1} \int_{(\Ft'')^0(F)} D^{G_x}(X)^{1/2} \hat{\theta}_{f_b ''}(X) dX.$$

In order to prove \eqref{11.3.1}, we only need to show $I(f_b '')=J(f_b '')$. This is just the Lie algebra analogue of the trace formula for the model $$(G_x,U_x),
$$
which is exactly the Whittaker model of $\GL_3(F_v)$. The proof is very similar to the Ginzburg-Rallis model case, we will prove it in the next section..
\end{proof}

Finally we can finish the proof of Theorem \ref{main 1} and \ref{main 2}. By Lemma \ref{group to Lie}, we only need to prove Theorem \ref{main 1}. We use the same argument as in the proof of $(th')_G$ above, except in the $x=1$ case, we don't have $\Omega_1\cap Supp(f)=\emptyset$. In this case, still by using localization, we can reduce to the Lie algebra argument (i.e. Theorem \ref{main 2}). Now since we have proved $(th')_G$, together with Lemma \ref{group to Lie}, we know $(th')_{\Fg}$ holds. Then using Proposition \ref{premier}, we get $(th)_{\Fg}$, which gives us $(th)_G$,
and hence Theorem \ref{main 1}.

\subsection{The proof of $I(f_b '')=J(f_b '')$}
In this section, we are going to prove
\begin{equation}\label{1}
I(f_b '')=J(f_b ''),
\end{equation}
which is the geometric side of the Lie algebra version of the relative trace formula for the Whittaker model of $GL_3(F_v)$. There are two ways to prove it, one is to apply the method we used in previous sections to the Whittaker model case, the other one is to use the spectral side of the trace formula together with the multiplicity formula of Whittaker model proved by Rodier in \cite{Rod81}.

\textbf{Method I:} By the same argument as in Section 10.2, we only need to prove \eqref{1} for $f_b ''$ whose support does not contain any nilpotent element. Then by changing $f_b ''$ to $(f_b '')^{\lambda}$, we may assume that the function $f_b ''$ is supported on a small neighborhood of $0$. Then we can relate $f_b ''$ to a function $\Phi_x$ on $G_x(F)/Z_{G_x}(F)$. By the same argument as in Section 7-10, we know that in order to prove \eqref{1}, it is enough to prove the geometric side of the local relative trace formula for $\Phi_x$, i.e. $\lim_{N\rightarrow \infty} I_N(\Phi_x)=c_{\Phi_x}(1)$. Here $I_N(\Phi_x)$ is defined in the same way as $I_N(f)$ in Section 5.2. In other word, we first integrate over $U_x$, then integrate on $G_x/U_x Z_{G_x}$. $c_{\Phi_x}(1)$ is the germ of $\theta_{\Phi_x}$ at $1$.

Since $f_b''$ does not support on nilpotent element, $\Phi_x$ does not support on unipotent element. This implies that $c_{\Phi_x}(1)=0$. On the other hand, since the only semisimple element in $U_x$ is $1$, by the same argument as in Section 7.1, the localization of $I_N(\Phi_x)$ at $y\in G_x(F)_{ss}$ is zero if $y$ is not in the center. If we are localizing at $1$, since the support of $\Phi_x$ does not contain unipotent element, we will still get zero once we choose the neighborhood small enough. Therefore $\lim_{N\rightarrow \infty} I_N(\Phi_x)=0=c_{\Phi_x}(1)$, and this proves \eqref{1}.

\textbf{Method II:} Same as in Method I, we reduce to prove the group version of the relative trace formula, i.e. $\lim_{N\rightarrow \infty} I_N(\Phi_x)=c_{\Phi_x}(1)$. By applying the same method as in \cite{Wan16}, we can prove a spectral expansion of $\lim_{N\rightarrow \infty} I_N(\Phi_x)$:
\begin{equation}\label{2}
\lim_{N\rightarrow \infty} I_N(\Phi_x)=\int_{\Pi_{temp}(G_x(F),1)} \theta_{\pi}(\Phi_x) m'(\bar{\pi}) d\pi
\end{equation}
where $\Pi_{temp}(G_x(F),1)$ is the set of all tempered representations of $G_x(F)$ with trivial central character, $d\pi$ is a measure on $\Pi_{temp}(G_x(F),1)$ defined in Section 2.8 of \cite{Wan16}, $\theta_{\pi}(\Phi_x)$ is defined in (3.4) of \cite{Wan16} via the weighted character, and $m'(\bar{\pi})$ is the multiplicity for the Whittaker model (here we are in $\GL$ case, all tempered representations are generic, so $m'(\bar{\pi})$ is always $1$).

By the work of Rodier, $m'(\bar{\pi})=c_{\bar{\pi}}(1)$ where $c_{\bar{\pi}}(1)$ is the germ of $\theta_{\bar{\pi}}$ at $1$, therefore \eqref{2} becomes
\begin{equation}\label{3}
\lim_{N\rightarrow \infty} I_N(\Phi_x)=\int_{\Pi_{temp}(G_x(F),1)} \theta_{\pi}(\Phi_x) c_{\bar{\pi}}(1) d\pi.
\end{equation}
Finally, as in Proposition 3.5 of \cite{Wan16}, we have
$$\theta_{\Phi_x}=\int_{\Pi_{temp}(G_x(F),1)} \theta_{\pi}(\Phi_x) \theta_{\bar{\pi}} d\pi.$$
Combining with \eqref{3}, we have $\lim_{N\rightarrow \infty} I_N(\Phi_x)=c_{\Phi_x}(1)$ and this proves \eqref{1}.

\appendix

\section{The Proof of Lemma \ref{major 2} and Lemma \ref{split zero}}
\subsection{The Proof of Lemma \ref{major 2}}
We refer the reader to Section 9.1 for the setup and notations. We first prove the following statement:

{\bf (1)}\ There exist $c',c>0$ such that $\kappa_{N,X''}(g'g)\leq {\kappa}_{c'N+c\sigma(g)}''(g')$ for all $g\in G(F)$ and $g'\in G_x(F)$. Here ${\kappa}_N ''$ is the truncated function for $G_x$ defined in the similar way as $\kappa_N$.

In fact, let $g'=m'u'k', k'g=muk$ with $m,m'\in M(F),\; u,u'\in U(F)$ and $k,k'\in K$. Then $\kappa_N(g'g)=\kappa_N(m'm)$. If this is nonzero, let
$$m'=\begin{pmatrix} m_1 ' & 0 & 0 \\ 0 & m_2 ' & 0 \\ 0 & 0 & m_3 ' \end{pmatrix}, m=\begin{pmatrix} m_1 & 0 & 0 \\ 0 & m_2 & 0 \\ 0 & 0 & m_3 \end{pmatrix}.$$
By the definition of $\kappa_N$ (as in \eqref{kappa split} and \eqref{kappa nonsplit}), we have
$$
\sigma((m_j ')^{-1} (m_j)^{-1} m_i m_i ')\ll N.
$$
On the other hand, we know $\sigma(m)\ll \sigma(g)$. Hence $\sigma(m_i)\ll \sigma(g)$, which implies that
$$\sigma(m_i '(m_j ')^{-1})\ll \sigma((m_j ')^{-1} (m_j)^{-1} m_i m_i ')+\sigma(m_i)+\sigma(m_j) \ll N+\sigma(g).$$
This proves {\bf (1)}.

Now we have
\begin{eqnarray*}
\kappa_{N,X''}(g)&=&\nu(A_T) \int_{Z_{G_x}\cap A_T(F)\backslash A_T(F)} \kappa_N(\gamma_{X''}^{-1} ag) da\\
&\leq& \nu(A_T) \int_{Z_{G_x}\cap A_T(F)\backslash A_T(F)} {\kappa}_{c'N+c\sigma(g)}''((\gamma_{X''})^{-1} a) da \\
&\leq& {\kappa}_{c'N+c\sigma(g),X''}''(1).
\end{eqnarray*}
So it reduces to show the following:

{\bf (2)}\ There exist an integer $k\in \BN$, and $c>0$ such that
$${\kappa}_{N,X''}''(1)\leq cN^k(1+| \log(| Q_T(X'')|_F) |)^k (1+| \log D^{G_x}(X'') |)^k.$$

Again here we only prove for the case where $x$ is in the center. Otherwise, we have the lower rank case, which is similar and easier. If $x$ is in the center, $G_x=G$ and $X''=X$. For simplicity, we will replace $X''$ by $X$, ${\kappa}_N ''$ by $\kappa_N$ and $D^{G_x}(X'')$ by $D^G(X)$ for the rest of the proof. We first deal with the case when $T$ is split. By Lemma \ref{section}, we know for $X\in \omega_T$, $X_{\Sigma}$ belongs to a compact subset of $\Xi+\Lambda$, and $\sigma(\gamma_X)\ll 1+|\log D^{G}(X)|$.

If $a\in A_T(F)$ such that $\kappa_N (\gamma_{X}^{-1}a)=1$. By the definition of $\kappa_N $ (as in \eqref{kappa split} and \eqref{kappa nonsplit}), we have $\gamma_{X}^{-1}a=hvy$ where $v\in U(F)$, $h\in H(F)$, and $y\in G(F)$ with $\sigma(y)\ll N$. Therefore $yXy^{-1}=v^{-1}h^{-1}X_{\Sigma}hv$. Since $X_{\Sigma}$ belongs to a compact subset, $\sigma(yXy^{-1})\ll N$, and hence
$$
\sigma(v^{-1}h^{-1}X_{\Sigma}hv)\ll N.
$$
By Lemma \ref{unipotent orbit}, the isomorphism \eqref{9.1} is algebraic, we have $\sigma(v)\ll N$ and $\sigma(h^{-1}X_{\Sigma}h)\ll N$.

Now let
$$X_{\Sigma}=\begin{pmatrix} 0 & 0 & Z \\ aI_2 & 0 & Y \\ 0 & bI_2 & 0 \end{pmatrix}.$$
By Proposition \ref{h-c}, we can find $s\in \GL_2(E)$ such that $s^{-1} Zs$ is a diagonal matrix and $\sigma(s)\ll 1+|\log(D^{GL_2(E)}(s^{-1}Zs))|$. Here $E/F$ is a finite extension generated by the elements in $F^{\times}/(F^{\times})^2$. Note that $D^{\GL_2(E)}(s^{-1}Zs)=\tr(Z)^2-4\det(Z)$, while the right hand side can be expressed as a polynomial of the coefficients of the characteristic polynomial of $X_{\Sigma}$, so it can be expressed as a polynomial on $\Ft_0(F)$. We remark that if $x$ is not in center, this will be polynomial on $\Ft''(F)$.

After conjugating by $s$, we may assume that $Z$ is a diagonal matrix with distinct eigenvalues $\lambda_1$ and $\lambda_2$ (we only need to change $h$ to $sh$). Here the eigenvalues are distinct because of the "generic position" assumption. After multiplying by elements in the center and in the open compact subgroup, using the Iwasawa decomposition, we may assume that
$$h=\begin{pmatrix} 1 & x \\ 0 & 1 \end{pmatrix}\begin{pmatrix} A & 0 \\ 0 & 1 \end{pmatrix}$$
and
$$\begin{pmatrix} 1 & -x \\ 0 & 1 \end{pmatrix} Y\begin{pmatrix} 1 & x \\ 0 & 1 \end{pmatrix}=\begin{pmatrix} y_{11} & y_{12} \\ y_{21} & y_{22} \end{pmatrix}.$$
Since $\sigma(h^{-1}X_{\Sigma}h)\ll N$, we have $\sigma(h^{-1}Zh),\sigma(h^{-1}Yh)\ll N$. This implies
$$\sigma(x(\lambda_1-\lambda_2)), \sigma(Ay_{12}),\sigma(A^{-1}y_{21})\ll N$$
Here for element in $t\in F$, $\sigma(t)=\log(\max\{1,| t|\})$.
Therefore, we obtain that $\sigma(x) \ll \max\{1,N-\log(| \lambda_1-\lambda_2|)\}$. Here $Z$ and $Y$ belong to a fixed compact subset before conjugation. Furthermore, after conjugating by $s$ and $\begin{pmatrix} 1 & x \\ 0 & 1 \end{pmatrix}$, $\sigma(Y)\ll \sigma(s)+\sigma(\begin{pmatrix} 1 & x \\ 0 & 1 \end{pmatrix})$. So we have
\begin{eqnarray}\label{major 2.5}
\sigma(A) &\ll& \max\{1,N-\sigma(y_{12})\} \nonumber\\
&\ll& \max\{1,N+\sigma(\begin{pmatrix} 1 & x \\ 0 & 1 \end{pmatrix})+\sigma(s)-\sigma(y_{12}y_{21})\}
\end{eqnarray}
and
\begin{eqnarray}\label{major 2.6}
\sigma(A^{-1}) &\ll& \max\{1,N-\sigma(y_{21})\}\nonumber\\
&\ll& \max\{1,N+\sigma(\begin{pmatrix} 1 & x \\ 0 & 1 \end{pmatrix})+\sigma(s)-\sigma(y_{12}y_{21})\}.
\end{eqnarray}
Note that here by the "generic position" assumption, we must have that $y_{12}y_{21}\neq 0$.

Recall as in the proof of Lemma \ref{orbit 1}, we have the following relations between the coefficients of the characteristic polynomial of $X_{\Sigma}$ and the data given by $Z$ and $Y$:
$$\text{coefficient\;of\;}\lambda^4 = b\tr(Y):=ba_4,$$
$$\text{coefficient\;of\;}\lambda^3 =ab \tr(Z):=aba_3,$$
$$\text{coefficient\;of\;}\lambda^2 =b^2 \det(Y):=b^2a_2,$$
$$\text{coefficient\;of\;}\lambda = ab^2( \lambda \; \text{coefficient} \;\text{of} \;\det(Z+\lambda Y)):=ab^2a_1,$$
and
$$\text{coefficient\;of\;}\lambda^0 =a^2b^2 \det(Z):=a^2b^2a_0.$$
Then
$$\begin{array}{cc}\left\{ \begin{array}{ccl} y_{11}+y_{22}=a_4  \\ \lambda_1y_{11}+\lambda_2y_{22}=a_1 \\ \end{array}\right. \end{array}$$
and
$$\begin{array}{cc}\left\{ \begin{array}{ccl} \lambda_1+\lambda_2=a_3  \\ \lambda_1\lambda_2=a_0 \\ \end{array}\right. \end{array}.$$
This implies
$$\begin{array}{cc}\left\{ \begin{array}{ccl} y_{11}=\frac{a_1-\lambda_1 a_4}{\lambda_2-\lambda_1}  \\ y_{22}=\frac{\lambda_2 a_4-a_1 }{\lambda_2-\lambda_1} \\ \end{array}\right. \end{array}.$$
So we have
$$y_{11}y_{22}=-\frac{\lambda_1\lambda_2 a_{4}^{2}-a_1 a_4(\lambda_1+\lambda_2) +a_{1}^{2}}{(\lambda_1-\lambda_2)^2}=\frac{a_0a_{4}^{2}-a_1a_3a_4+a_{1}^{2}}{a_{3}^{2}-4a_0}.$$
In particular, $y_{12}y_{21}=\det(Y)-y_{11}y_{22}=a_2-y_{11}y_{22}$ is a rational function of the $a_i$'s, and hence it is a rational function on $\Ft_0(F)$. Also
\begin{equation}\label{major 2.4}
\sigma(\begin{pmatrix} 1 & x \\ 0 & 1 \end{pmatrix})=\sigma(x)\ll \max\{1,N-\log(|\lambda_1-\lambda_2|)\}
\end{equation}
where the right hand side can be expressed as logarithmic function of some rational function on $\Ft_0(F)$.

Finally,
combining \eqref{major 2.5}, \eqref{major 2.6}, \eqref{major 2.4}, and the majorization of $s$, we can find a rational function $Q_T(X)$ on $\Ft_0(F)$ such that $\sigma(h)\ll N+(1+\log | Q_T(X)|)$. Then combining the majorization of $v$, $y$ and $\gamma_Y$, we know up to an element in the center, if $\kappa_N(\gamma_{X}^{-1} a)=1$, we have
\begin{equation}\label{A.1}
\sigma(a)\ll N+(1+\log Q_T(X))+(1+\log D^{G}(X)).
\end{equation}
Since $mes\{a\in (Z_{G_x}\cap A_T(F))\backslash A_T(F) \mid \sigma_{Z_{G_x}\backslash G_x}(a)\leq r\}\ll r^k$ for some $k\in \BN$,
the Lemma follows from the definition of $\kappa_{N,X''}$ (as in \eqref{kappa 2}).

Now if $T$ is not split, since we are talking about majorization, we may pass to a finite extension. Then by the same argument as above, we can show that if $\kappa_N(\gamma_{X}^{-1} a)=1$ for some $a\in A_T(F)$, up to an element in the center, the estimation \eqref{A.1} will still holds. Then we can still prove the lemma as in the split case.

\subsection{The proof of Lemma \ref{split zero}}
We refer the reader to Section 9.6 for the setup and notations. For $Q\neq G_x$, let $Q=Q_1\times Q_2$ where $Q_1$ is the parabolic subgroup of the first $GL_3(F)$ and $Q_2$ is the parabolic subgroup of the second $GL_3(F)$. We anticipate some unipotent invariance property of $\kappa_{N,X''}(Q,g)$. We will only deal with the case when the $Q_i$'s contain the lower Borel subgroup of $GL_3(F)$. For general $Q$, it can be conjugated to this situation via some weyl element in $G_x$, and then by conjugating it back we get the unipotent subgroup we need for this general $Q$. Finally, by a similar argument as (4) of Section 9.5, we can reduced to the case when the $Q_i$'s contain the lower Borel subgroup.
\\

\textbf{Case 1:} Assume $Q_1$ is contained in the parabolic subgroup $P_{1,2}$ (i.e. the parabolic subgroup of $GL_3(F)$ with Levi $GL_1\times GL_2$ and contains the lower Borel subgroup), then let $U_{1,5}$ be the unipotent subgroup of the parabolic subgroup $P_{1,5}$ of $GL_6(F)$. Here $P_{1,5}=L_{1,5}U_{1,5}$ is the parabolic subgroup with Levi $GL_1\times GL_5$ and contains the upper Borel subgroup. We are going to prove a lemma which is similar to (6) of Proposition \ref{change truncation}. Once this lemma has been proved, we can use the same argument as in (7) of Proposition \ref{change truncation} to conclude that $I(Q,X)=0$ for $N$ large.

\begin{lem}
There exists $N_1>0$ satisfying the following condition: For all $N>N_1$ and $c>0$, we can find $c'>0$ such that if
$$c'\log(N)<\inf\{ \alpha(Y_{P_{min}})\mid \alpha\in \Delta_{min}\},$$
we have $\kappa_{N,X''}(Q,ug) =\kappa_{N,X''}(Q,g)$ for all $g\in G(F)$ and $u\in U_{1,5}(F)$ with $\sigma(g),\sigma(u),\sigma(ug)<c \log(N)$, and for all $X''\in \omega_{T''}\cap (\Ft'')^0(F)) [>N^{-b}]$.
\end{lem}

Since the function
$$g\rightarrow \sigma_{M_{\sharp}}^{Q}(H_{M_{\sharp}}(a), \CY(g))\tau_{Q}(H_{M_{\sharp}}(a)-\CY(g)_{Q})$$
is left $U_S(F) U_{\bar{Q}}(F)$-invariant, and $U_{1,5}(F)\subset U_S(F) U_{\bar{Q}}(F)$, by applying the same argument as in Section 10.5, we only need to prove the following lemma.
\begin{lem}
There exists $N_1>0$ such that for all $N>N_1$ and $c>0$, there exists $c'>0$ satisfying the following condition: For given $a\in A_T(F),g\in G(F), u\in U_{1,5}(F)$ and $X''\in \omega_{T''} \cap (\Ft'')^0(F)[>N^{-b}]$, assume that $\sigma(g),\sigma(u),\sigma(ug)<c\log(N)$, and $\alpha(H_{M_{\sharp}}(a))>c'\log(N)$ for all $\alpha\in \Sigma_{Q}^{+}$. Then $$\kappa_N(\gamma_{X''}^{-1}aug)=\kappa_N(\gamma_{X''}^{-1}ag).$$
\end{lem}

\begin{proof}
The proof is very similar to Lemma \ref{U-invariant 2}, so we will just sketch the key steps. Same as that Lemma, we can reduce to the case when $T$ split, and we can show this argument is independent of the choice of $X_{\Sigma}''$ and $\gamma_{X''}$. Then we can choose our local section $X''\rightarrow X_{\Sigma}''$ to be $X_{\Sigma}''=X''+\Xi$, and choose $\gamma_{X''}\in G_x$ to be of the form:
$$\begin{pmatrix} 1 & 0 & 0 \\ x_1 & 1 & 0 \\ 0 & y_1 & 1 \end{pmatrix}\times \begin{pmatrix} 1 & 0 & 0 \\ x_2 & 1 & 0 \\ 0 & y_2 & 1 \end{pmatrix}.$$
We can write $\gamma_{X''}$ as the product of $n_{X''}$ and $u_{X''}$ where
$$n_{X''}=\begin{pmatrix} 1 & 0 & 0 & 0 & 0 & 0 \\ 0 & 1 & 0 & 0 & 0 & 0 \\ x_1 & 0 & 1 & 0 & 0 & 0 \\ 0 & 0 & 0 & 1 & 0 & 0 \\ 0 & 0 & 0 & 0 & 1 & 0 \\ 0 & 0 & 0 & 0 & 0 & 1 \end{pmatrix} \in \bar{U}_{1,5}$$
and
$$u_{X''}= \begin{pmatrix} 1 & 0 & 0 & 0 & 0 & 0 \\ 0 & 1 & 0 & 0 & 0 & 0 \\ 0 & 0 & 1 & 0 & 0 & 0 \\ 0 & x_2 & 0 & 1 & 0 & 0 \\ 0 & 0 & y_1 & 0 & 1 & 0 \\ 0 & 0 & 0 & y_2 & 0 & 1 \end{pmatrix} \in L_{1,5}.$$

By the conditions imposed on $a$ and $Q$, using the same argument as in Lemma \ref{U-invariant 2}, we can choose $c'$ large to get rid of $n_{X''}$. In fact, beside the diagonal part, $n_{X''}$ only has non-zero value on the (3,1) position. It is easy to see the (2,1) position for the first $GL_3(F)$ corresponds to the (3,1) position of $GL_6(F)$. The fact that $Q$ is contained in $P_{1,2}$ and the condition imposed on $a$ tell us that $a^{-1}n_{X''}a$ is a contraction at least by $c'\log(N)$. So if we let $c'$ large, we can make $g^{-1}u^{-1}a^{-1}n_{X''} aug -1$ very close to zero. Hence we can make $g^{-1}u^{-1}a^{-1}n_{X''} aug \in K$. This allows us to get rid of $n_{X''}$. Then we can conjugate $u$ by $a^{-1}u_{X''} a$ to get rid of $u_{X''}$. Here same as in Lemma \ref{U-invariant 2}, we should assume at the beginning that $a$ lies in the positive chamber defined by the lower Borel subgroup $B=B_1\times B_2$ of $GL_3(F)\times GL_3(F)$ to make sure $a^{-1}u_{X''} a$ is a contraction. Now we have eliminated the effect of $\gamma_{X''}$. By applying the same argument as in Lemma \ref{U-invariant 2}, we may also assume that $g=1$.
Hence we only need to prove
\begin{equation}\label{split 5}
\kappa_N(au)=\kappa_N(a).
\end{equation}

Let $$u=u_1 u_2=\begin{pmatrix} 1 & 0 & w_2 & w_3 & w_4 & w_5 \\ 0 & 1 & 0 & 0 & 0 & 0 \\ 0 & 0 & 1 & 0 & 0 & 0 \\ 0 & 0 & 0 & 1 & 0 & 0 \\ 0 & 0 & 0 & 0 & 1 & 0 \\ 0 & 0 & 0 & 0 & 0 & 1 \end{pmatrix} \begin{pmatrix} 1 & w_1 & 0 & 0 & 0 & 0 \\ 0 & 1 & 0 & 0 & 0 & 0 \\ 0 & 0 & 1 & 0 & 0 & 0 \\ 0 & 0 & 0 & 1 & 0 & 0 \\ 0 & 0 & 0 & 0 & 1 & 0 \\ 0 & 0 & 0 & 0 & 0 & 1 \end{pmatrix}.$$
Note that $u_1$ can be moved to the very left via the $a$-conjugation. Then we can eliminate it by using the left $U(F)$-invariance property of $\kappa_N$. For $u_2$, it will only add some element majorized by $N+c\log(N)$ to the unipotent part, and not change the semisimple part. So we only need to let $N$ large such that $N+c\log(N)<(1+\epsilon)N$, and then the equation \eqref{split 5} just follows from the definition of $\kappa_N$
(as in \eqref{kappa split}). This is the same technique as in the last part of the proof of Lemma \ref{U-invariant 2}. This finishes the proof of the Lemma.
\end{proof}

Finally, by combining the above two Lemmas, together with the argument in (7) of Proposition \ref{change truncation}, we know $I(Q,X)=0$ for $N$ large.
\\

\textbf{Case 2}: If $Q_2$ is contained in $P_{2,1}$, we can use exactly the same argument as in Case 1 except replacing $U_{1,5}$ by $U_{5,1}$, we will still get $I(Q,X)=0$ for $N$ large.
\\

\textbf{Case 3}: If $Q_1=P_{2,1}$ and $Q_2=GL_3(F)$, we need to use some non-standard parabolic subgroup, which means that the unipotent subgroup we use will no longer be upper triangular. Let $P'=L'U'$ be the parabolic subgroup of $GL_6(F)$ where $L'$ and $U'$ are of the following forms:
$$L'=\{\begin{pmatrix} \ast & 0 & \ast & 0 & 0 & 0 \\ 0 & \ast & 0 & \ast & \ast & \ast \\ \ast & 0 & \ast & 0 & 0 & 0 \\ 0 & \ast & 0 & \ast & \ast & \ast \\ 0 & \ast & 0 & \ast & \ast & \ast \\ 0 & \ast & 0 & \ast & \ast & \ast \end{pmatrix} \},
U'=\{\begin{pmatrix} 1 & \ast & 0 & \ast & \ast & \ast \\ 0 & 1 & 0 & 0 & 0 & 0 \\ 0 & \ast & 1 & \ast & \ast & \ast \\ 0 & 0 & 0 & 1 & 0 & 0 \\ 0 & 0 & 0 & 0 & 1 & 0 \\ 0 & 0 & 0 & 0 & 0 & 1 \end{pmatrix}.$$
Then we still have $U'(F)\subset U_R(F) U_{\bar{Q}}(F)$. Same as previous cases, we only need to prove the following lemma.

\begin{lem}\label{split lemma 1}
There exists $N_1>0$ such that for $N>N_1$ and $c>0$, there exists $c'>0$ satisfying the following condition: For given $a\in A_T(F),g\in G(F), u\in U'(F)$ and $X''\in \omega_{T''} \cap (\Ft'')^0(F)[>N^{-b}]$, assume that $\sigma(g),\sigma(u),\sigma(ug)<c\log(N)$, $\alpha(H_{M_{\sharp}}(a))>c'\log(N)$ for all $\alpha\in \Sigma_{Q}^{+}$, and $\sigma_{M_{\sharp}}^{Q}(H_{M_{\sharp}}(a), \CY(g))\tau_{Q}(H_{M_{\sharp}}(a)-\CY(g)_{Q}) =1$. Then $$\kappa_N(\gamma_{X''}^{-1}aug)=\kappa_N(\gamma_{X''}^{-1}ag).$$
\end{lem}

\begin{proof}
Same as previous cases, we can still reduce to the case where $T$ split, and we can show this argument is independent of the choice of $X_{\Sigma}''$ and $\gamma_{X''}$. Then we can choose our local section $X''\rightarrow X_{\Sigma}''$ to be $X_{\Sigma}'' X''+\Xi$, and choose $\gamma_{X''}\in G_x$ to be of the form
$$\begin{pmatrix} 1 & 0 & 0 \\ x_1 & 1 & 0 \\ 0 & y_1 & 1 \end{pmatrix}\times \begin{pmatrix} 1 & 0 & 0 \\ x_2 & 1 & 0 \\ 0 & y_2 & 1 \end{pmatrix}.$$
We can write $\gamma_{X''}$ as the product of $n_{X''}$ and $u_{X''}$ where
$$n_{X''}=\begin{pmatrix} 1 & 0 & 0 & 0 & 0 & 0 \\ 0 & 1 & 0 & 0 & 0 & 0 \\ 0 & 0 & 1 & 0 & 0 & 0 \\ 0 & 0 & 0 & 1 & 0 & 0 \\ 0 & 0 & y_1 & 0 & 1 & 0 \\ 0 & 0 & 0 & 0 & 0 & 1 \end{pmatrix} \in \bar{U}'(F)$$
and
$$u_{X''}= \begin{pmatrix} 1 & 0 & 0 & 0 & 0 & 0 \\ 0 & 1 & 0 & 0 & 0 & 0 \\ x_1 & 0 & 1 & 0 & 0 & 0 \\ 0 & x_2 & 0 & 1 & 0 & 0 \\ 0 & 0 & 0 & 0 & 1 & 0 \\ 0 & 0 & 0 & y_2 & 0 & 1 \end{pmatrix} \in L'(F).$$

Using the same argument as in Case 1, we can make $c'$ large to eliminate $n_{X''}$, and conjugate $u$ by $a^{-1}u_{X''} a$ to eliminate $u_{X''}$. Similarly, we may assume that $g=1$ and only need to show
\begin{equation}\label{split 6}
\kappa_N(au)=\kappa_N(a).
\end{equation}

Let
\begin{equation}\label{split 11}u=u_1 u_2=\begin{pmatrix} 1 & v_1 & 0 & v_2 & v_3 & v_4 \\ 0 & 1 & 0 & 0 & 0 & 0 \\ 0 & 0 & 1 & w_1 & w_2 & w_3 \\ 0 & 0 & 0 & 1 & 0 & 0 \\ 0 & 0 & 0 & 0 & 1 & 0 \\ 0 & 0 & 0 & 0 & 0 & 1 \end{pmatrix} \begin{pmatrix} 1 & 0 & 0 & 0 & 0 & 0 \\ 0 & 1 & 0 & 0 & 0 & 0 \\ 0 & x & 1 & 0 & 0 & 0 \\ 0 & 0 & 0 & 1 & 0 & 0 \\ 0 & 0 & 0 & 0 & 1 & 0 \\ 0 & 0 & 0 & 0 & 0 & 1 \end{pmatrix}.
\end{equation}
Then $u_1$ belongs to the upper triangle unipotent subgroup. The $v_2$, $v_3$, $v_4$, $w_2$ and $w_3$ part of $u_1$ can be eliminated by the left $U(F)$-invariance property of $\kappa_N$. The $v_1$ and $w_1$ part of $u_1$ will only add some element majorized by $N+c\log(N)$ to the unipotent part and not change the semisimple part, so we can just let $N$ large such that $N+c\log(N)<(1+\epsilon)N$. This is the same technique as the last part of the proof of Lemma \ref{U-invariant 2}. This tells us $\kappa_N(au)=\kappa_N(au_2)$. So we only need to prove \eqref{split 6} for the case where $u=u_2$.

If $\mid x\mid \leq 1$, then $u\in K$, the equation \eqref{split 6} just follows from the right $K$-invariance property of $\kappa_N$.

If $\mid x\mid>1$, let $a=\diag(a_1,a_2,a_3,a_4,a_5,a_6)$. Since we assume that $\alpha(H_{M_{\sharp}}(a))>c'\log(N)$ for any $\alpha\in \Sigma_{Q}^{+}$, and
$$
\sigma_{M_{\sharp}}^{Q}(H_{M_{\sharp}}(a), \CY(g))\tau_{Q}(H_{M_{\sharp}}(a)-\CY(g)_{Q}) =1,
$$
together with the inequality \eqref{split condition 2}, we know, up to modulo an element in $Z_{G_x}$,
$$\sigma(a_2),\sigma(a_4),\sigma(a_6)\ll \log(N)^2$$
and
$$\mid a_1\mid \leq \mid a_3\mid <\mid a_5\mid, \sigma(a_{3}^{-1}a_5)>c' \log(N).$$
In this case, the Iwasawa decomposition of $\left( \begin{array}{cc} 1 & 0 \\ x & 1 \end{array} \right)$ is
$$\left( \begin{array}{cc} 1 & 0 \\ x & 1 \end{array} \right)=\left( \begin{array}{cc} x^{-1} & 1 \\ 0 & x \end{array} \right)\left( \begin{array}{cc} 0 & -1 \\ 1 & x^{-1} \end{array} \right).$$
After eliminating the unipotent part and the $K$-part as before, we have
\begin{equation}\label{split 7}
\kappa_N(au)=\kappa_N(a')
\end{equation}
where
$a'=\diag(a_1,x^{-1} a_2,xa_3,a_4,a_5,a_6)$. Now by the condition imposed on $a_2,a_4$ and $a_6$, together with the fact that $\sigma(x)<c\log(N)$, we still have
\begin{equation}\label{split 8}
\sigma(x^{-1}a_2),\sigma(a_4),\sigma(a_6)\ll \log(N)^2.
\end{equation}
Since $\mid a_1\mid \leq \mid a_3\mid <\mid a_5\mid$, $\sigma(a_{3}^{-1}a_5)>c' \log(N)$, $\mid x\mid >1$ and $\sigma(x)<c\log(N)$, if we let $c'>c$, we will still have
$$\mid a_1\mid \leq \mid x a_3\mid <\mid a_5\mid.$$
This implies
\begin{eqnarray}\label{split 9}
\sigma(a_{1}^{-1} a_5)
&=&\max\{\sigma(a_{1}^{-1} a_3),\sigma(a_{1}^{-1} a_5),\sigma(a_{3}^{-1} a_5)\}\nonumber\\
&=&\max\{\sigma(a_{1}^{-1} (x a_3)),\sigma(a_{1}^{-1} a_5),\sigma((xa_{3})^{-1} a_5)\}.
\end{eqnarray}
Then for $N$ large (so that $C\log(N)^2< N$), \eqref{split 6} just follows from the definition of $\kappa_N$ together with \eqref{split 7} \eqref{split 8} and \eqref{split 9}.
\end{proof}

Finally, by the Lemma above, together with the same argument as in previous cases, we have $I(Q,X)=0$ for $N$ large.
\\

\textbf{Case 4}: If $Q_1=GL_3(F)$ and $Q_2=P_{1,2}(F)$, we use exactly the same argument as Case 3 except replacing the unipotent subgroup by
$$U'=\{\begin{pmatrix} 1 & 0 & 0 & \ast & 0 & \ast \\ 0 & 1 & 0 & \ast & 0 & \ast \\ 0 & 0 & 1 & \ast & 0 & \ast \\ 0 & 0 & 0 & 1 & 0 & 0 \\ 0 & 0 & 0 & \ast & 1 & \ast \\ 0 & 0 & 0 & 0 & 0 & 1 \end{pmatrix}\}$$
\\

\textbf{Case 5:} If $Q$ does not belong to Case 1, 2, 3 and 4, then $Q_1=P_{2,1}(F)$ and $Q_2=P_{1,2}(F)$. In this case, for simplicity, we conjugate $Q_2$ and $B_2$ by the Weyl element $w=\begin{pmatrix} 0 & 1 & 0 \\ 1 & 0 & 0 \\ 0 & 0 & 1 \end{pmatrix}$. Then the parabolic subgroup $Q_2$ consists of elements in $GL_3(F)$ of the form
$$\begin{pmatrix} \ast & \ast & \ast \\ 0 & \ast & 0 \\ \ast & \ast & \ast \end{pmatrix},$$
and $a\in A_T(F)$ lies inside the positive chamber defined by the Borel subgroup of the form
\begin{equation}\label{split 13}
\begin{pmatrix} \ast & 0 & 0 \\ \ast & \ast & 0 \\ \ast & \ast & \ast \end{pmatrix} \times \begin{pmatrix} \ast & \ast & 0 \\ 0 & \ast & 0 \\ \ast & \ast & \ast \end{pmatrix}.
\end{equation}
Let $P'=L'U'$ be the parabolic subgroup of $GL_6(F)$ as in Case 3. We only need to prove a similar version of Lemma \ref{split lemma 1} for this case.

\begin{lem}
There exists $N_1>0$ such that for $N>N_1$ and $c>0$, there exists $c'>0$ satisfying the following condition: For given $a\in A_T(F),g\in G(F), u\in U'(F)$ and $X''\in \omega_{T''} \cap (\Ft'')^0(F)[>N^{-b}]$, assume that $\sigma(g),\sigma(u),\sigma(ug)<c\log(N)$, and $\alpha(H_{M_{\sharp}}(a))>c'\log(N)$ for all $\alpha\in \Sigma_{Q}^{+}$. Then $$\kappa_N(\gamma_{X''}^{-1}aug)=\kappa_N(\gamma_{X''}^{-1}ag).$$
\end{lem}

\begin{proof}
Applying the same argument as in the proof of Lemma \ref{split lemma 1}, we can eliminate the effect of $\gamma_{X''}$ and $g$, so we only need to prove
\begin{equation}\label{split 12}
\kappa_N(au)=\kappa_N(a).
\end{equation}
Then we can decompose $u=u_1 u_2$ as in \eqref{split 11}, and apply the same argument to eliminate the $u_1$ part, so we only need to deal with the case when
$$u=\left( \begin{array}{cccccc} 1 & 0 & 0 & 0 & 0 & 0 \\ 0 & 1 & 0 & 0 & 0 & 0 \\ 0 & x & 1 & 0 & 0 & 0 \\ 0 & 0 & 0 & 1 & 0 & 0 \\ 0 & 0 & 0 & 0 & 1 & 0 \\ 0 & 0 & 0 & 0 & 0 & 1 \end{array} \right).$$

Let $a=\diag(a_1,a_2,a_3,a_4,a_5,a_6)$. Since $\alpha(H_{M_{\sharp}}(a))>c'\log(N)$ for all $\alpha\in \Sigma_{Q}^{+}$, and $a$ lies in the positive chamber defined by the Borel subgroup of the form (10.83), we have
\begin{equation}\label{split 14}
\mid a_1\mid \leq \mid a_3\mid <\mid a_5\mid, \sigma(a_{3}^{-1}a_5)>c' \log(N)
\end{equation}
and
\begin{equation}\label{split 15}
\mid a_4\mid \leq \mid a_2\mid <\mid a_6\mid, \sigma(a_{4}^{-1}a_2)>c' \log(N).
\end{equation}

If $\mid x\mid \leq 1$, then $u\in K$, the equation \eqref{split 12} just follows from the right $K$-invariance property of $\kappa_N$.

If $\mid x \mid >1$, then the proof is the same as Lemma \ref{split lemma 1}. By using the Iwasawa decomposition of $\left( \begin{array}{cc} 1 & 0 \\ x & 1 \end{array} \right)$, we have
$$\kappa_N(au)=\kappa_N(a')$$
where
$a'=\diag(a_1,x^{-1} a_2,xa_3,a_4,a_5,a_6)$. Now by \eqref{split 14} and \eqref{split 15}, together with the fact that $\sigma(x)<c\log(N)$, if we make $c'>c$, we still have
$$\mid a_1\mid \leq \mid x a_3\mid <\mid a_5\mid$$
and
$$\mid a_4\mid \leq \mid x^{-1} a_2\mid <\mid a_6\mid.$$
Therefore
\begin{eqnarray*}
\sigma(a_{1}^{-1} a_5)&=&\max\{\sigma(a_{1}^{-1} a_3),\sigma(a_{1}^{-1} a_5),\sigma(a_{3}^{-1} a_5)\}\\
&=&\max\{\sigma(a_{1}^{-1} (x a_3)),\sigma(a_{1}^{-1} a_5),\sigma((xa_{3})^{-1} a_5)\}
\end{eqnarray*}
and
\begin{eqnarray*}
\sigma(a_{4}^{-1} a_6)&=&max\{\sigma(a_{4}^{-1} a_2),\sigma(a_{4}^{-1} a_6),\sigma(a_{2}^{-1} a_6)\}\\
&=&max\{\sigma(a_{4}^{-1} (x^{-1} a_2)),\sigma(a_{4}^{-1} a_6),\sigma((x^{-1} a_{2})^{-1} a_6)\}.
\end{eqnarray*}
It is clear that \eqref{split 12} just follows from the above two equations and the definition of $\kappa_N$.
\end{proof}
Finally, by using the same argument as in previous cases, we have $I(Q,X)=0$ for $N$ large. This finishes the proof of the Lemma.

\subsection{A final remark}
In Section 9.6 and Appendix A.2, we have proved that the localization at split element will always be zero. In this section, we are going to use another method to prove this argument, the main ingredient of our method is the spectral side of the trace formula. The idea comes from Beuzart-Plessis's proof of the local Gan-Gross-Prasad conjecture of unitary group in \cite{B15}.

In Section 7.3 of \cite{Wan16}, we have proved a spectral expansion of $\lim_{N\rightarrow \infty} I_N(f)$:
\begin{equation}\label{A.2}
\lim_{N\rightarrow \infty} I_N(f)=\int_{\Pi_{temp}(G(F),1)} \theta_f(\pi)m(\bar{\pi})d\pi.
\end{equation}
Here $\Pi_{temp}(G(F),1)$ is the set of all tempered representations of $G(F)$ with trivial central character, $d\pi$ is a measure on $\Pi_{temp}(G(F),1)$ defined in Section 2.8 of \cite{Wan16}, $\theta_{\pi}(f)$ is defined in (3.4) of \cite{Wan16} via the weighted character, and $m(\bar{\pi})$ is the multiplicity for the Ginzburg-Rallis model. Combining \eqref{A.2} with the fact that
$$\theta_f=\int_{\Pi_{temp}(G(F),1)} \theta_f(\pi)\theta_{\bar{\pi}}d\pi,$$
we have
\begin{equation}\label{A.3}
\lim_{N\rightarrow \infty} I_N(f)-I(f)=\int_{\Pi_{temp}(G(F),1)} \theta_f(\pi)(m(\bar{\pi})-m_{geom}(\bar{\pi}))d\pi.
\end{equation}

In Corollary 5.15 and Lemma 8.1 of \cite{Wan16}, we have proved that both $m(\bar{\pi})$ and $m_{geom}(\bar{\pi})$ are invariant under parabolic induction, therefore by induction, we may assume that $m(\bar{\pi})=m_{geom}(\bar{\pi})$ when $\pi$ is not a discrete series. Then \eqref{A.3} becomes
\begin{equation}\label{A.4}
\lim_{N\rightarrow \infty} I_N(f)-I(f)=\int_{\Pi_{2}(G(F),1)} \theta_f(\pi)(m(\bar{\pi})-m_{geom}(\bar{\pi}))d\pi
\end{equation}
where $\Pi_{2}(G(F),1)$ is the set of all discrete series of $G(F)$ with trivial central character. Now if the support of $f$ does not contain any elliptic element, for all $\pi\in \Pi_{2}(G(F),1)$, we have $\theta_f(\pi)=\tr(\pi(f))=0$. Together with \eqref{A.4}, we have
$$\lim_{N\rightarrow \infty} I_N(f)=I(f).$$
But since the support of $f$ does not contain any elliptic element, by the definition of $I(f)$, we have
$$\lim_{N\rightarrow \infty} I_N(f)=I(f)=0.$$
In conclusion, we have proved that if the support of $f$ does not contain any elliptic element, $\lim_{N\rightarrow \infty} I_N(f)=0$. In particular, the localization at split element will always be zero.

\section{The Reduced Model}
In this section, we will state some similar results for the reduced models of the Ginzburg-Rallis model, which appear naturally under the parabolic induction. To be specific, we can have analogy results of Theorem \ref{main} and Theorem \ref{main 3} for those models. Since most of proof goes exactly the same as the Ginzburg-Rallis model case which we discussed in previous Sections, we will skip it here. We refer the readers to my thesis \cite{Wan17} for details of the proof. We will need those results in our proof of Conjecture \ref{jiang} for tempered representations, which is the main result of \cite{Wan16}.

\subsection{The general setup}
We still let $(G,R)$ to denote the Ginzburg-Rallis model. We first state some results on the pair $(G,R)$ as a spherical variety. The proof of those results can be found on Section 4 of \cite{Wan16} or my thesis \cite{Wan17}. We say a parabolic subgroup $\bar{Q}$ of $G$ is good if $R\bar{Q}$ is a Zariski open subset of $G$. This is equivalent to say that $R(F)\bar{Q}(F)$ is open in $G(F)$ under the analytic topology.

\begin{prop}\label{spherical variety}
\begin{enumerate}
\item There exist minimal parabolic subgroups of $G$ that are good and they are all conjugated to each other by some elements in $R(F)$. If $\bar{P}_{min}=M_{min}\bar{U}_{min}$ is a good minimal parabolic subgroup, we have $H\cap \bar{U}_{min}=\{1\}$ and the complement of $R(F)\bar{P}_{min}(F)$ in $G(F)$ has zero measure. In particular, $(G,R)$ is a spherical pair.
\item A parabolic subgroup $\bar{Q}$ of $G$ is good if and only if it contains a good minimal parabolic subgroup.
\end{enumerate}
\end{prop}

Now let $Q=LU_Q$ be a good parabolic subgroup of $G$, by the proposition above, we have $R\cap U_Q=\{1\}$. This implies the intersection $R_Q=R\cap Q$ can be viewed as a subgroup of $L$. \textbf{The reduced model we want to study is just $(L,R_Q)$}. Since all good minimal parabolic subgroups are all conjugated to each other by some elements in $R(F)$, the model $(L,R_Q)$ is independent of the choice of $Q$ up to conjugation. There are two types of such reduced models: type I is those models appear on both $GL_6(F)$ case and the $GL_3(D)$ case; type II is those models appear only on the $GL_6(F)$ case.
\\
\\
\textbf{Type I}: There are only two models of type I. One is the trilinear $\GL_2$ model which comes from the parabolic subgroup of $(2,2,2)$ type (or $(1,1,1)$ type in the $\GL_3(D)$ situation). The other model is the "middle model" between the Ginzburg-Rallis model and the trilinear $\GL_2$ model which comes from the parabolic subgroup of type $(4,2)$ and $(2,4)$= (or type $(2,1)$ and $(1,2)$ in the $\GL_3(D)$ situation). We will study this two models in the next three sections.
\\
\\
\textbf{Type II}: The models of this type only appear on the $\GL_6(F)$ case, they don't have an analogy in the quaternion case. This include all the pairs $(L,R_Q)$ where $Q=LU_Q$ is a good proper parabolic subgroup of $\GL_6(F)$ which are not of $(4,2)$ or $(2,2,2)$ type. We will study these models in the last section.

\subsection{The trilinear model}
Choose $Q$ be the parabolic subgroup of $\GL_6(F)$ (resp. $\GL_3(D)$) of $(2,2,2)$ type (resp. $(1,1,1)$ type) which contains the lower minimal parabolic subgroup. Then it is easy to see the model $(L,R_Q)$ is just the trilinear $\GL_2$ model. In other word, $L(F)=(\GL_2(F))^3$, and $R_Q(F)=\GL_2(F)$ diagonally embedded into $L(F)$. This model has been studied by Prasad in his thesis \cite{P90}.

To make our notation simple, in this section we will temporarily let $G=GL_2(F)\times GL_2(F)\times GL_2(F)$ and $H=GL_2(F)$ diagonally embedded into $G$. For a given irreducible representation $\pi$ of $G$, assume $\omega_{\pi}=\chi^2$ for some character $\chi$ of $F^{\times}$. $\chi$ will induce a one-dimensional representation $\sigma$ of $H$. Let
\begin{equation}
m(\pi)=\dim \Hom_{H(F)} (\pi,\sigma).
\end{equation}

Similarly, we have the quaternion algebra version with the pair $G_D=GL_1(D)\times GL_1(D)\times GL_1(D)$ and $H_D=GL_1(D)$. We can still define the multiplicity $m(\pi_D)$. The following theorem has been proved by Prasad in his thesis \cite{P90} for general generic representation using different method. This can also be deduced essentially from Waldspurger's result on the model $(SO(4)\times SO(3),SO(3))$. By using our method in this paper, we can prove the supercuspidal case.

\begin{thm}
If $\pi$ is a supercuspidal representation of $G$, let $\pi_D$ be the Jacquet-Langlands correspondence of $\pi$ to $G_D$. (Since $\pi$ is supercuspidal, $\pi_D$ always exist.) Then
$$m(\pi)+m(\pi_D)=1.$$
\end{thm}

Now let $(G,H)$ be either $(G,H)$ or $(G_D,H_D)$ defined as above. Let $\CT$ be the subset of subtorus $T$ in $H$ defined in Section 5.1. We can also prove the local relative trace formula for this model, this will be used in the proof of the spectral side of the local relative trace formula for the Ginzburg-Rallis model in the forthcoming paper \cite{Wan16}.

Let $\theta$ be a quasi-character on $Z_G(F)\backslash G(F)$ with central character $\eta=\chi^2$, and $T\in \CT$. If $T=\{1\}$, then we are in the split case. Since there is a unique regular nilpotent orbit in $\Fg(F)$, we let $c_\theta(t)=c_{\theta,\CO_{reg}}(t)$. If $T=T_v$ for some $v\in F^{\times}/(F^{\times})^2, v\neq 1$, and $t\in T_v$ is a regular element, then $G_t$ is abelian. Since in this case the germ of quasi-character is just itself, we define $c_\theta(t)=\theta(t)$.

Let $f\in C_{c}^{\infty} (Z_G(F)\backslash G(F),\eta)$ be a strongly cuspidal function. For each $T\in \CT$, let $c_f$ be the function $c_{\theta_f}$ defined above. Define
$$I_{\eta}(f)=\sum_{T\in \CT} \mid W(H,T)\mid^{-1} \nu(T) \int_{Z_G(F)\backslash T(F)} c_f(t) D^H(t) \chi(\det(t))^{-1} dt.$$
By a similar argument as Proposition \ref{integrable 1}, we know the integral is absolutely convergent.

Now for $g\in G(F)$, we define the function $I_{\eta}(f,g)$ to be
$$I_{\eta}(f,g)=\int_{Z_H(F)\backslash H(F)} f(g^{-1}xg) \chi(\det(x))^{-1} dx.$$
And for each $N\in \BN$, we can still define the truncated function $\kappa_N$ in a similar way, let
$$I_{N,\eta}(f)=\int_{H(F)\backslash G(F)} I(f,g) \kappa_N(g) dg$$
The following theorem is an analogy of Theorem \ref{main 3} for this model, it is the geometric side of the local relative trace formula for the trilinear $\GL_2$ model.
\begin{thm}
For every strongly cuspidal function $f$ belonging to the space $C_{c}^{\infty} (Z_G(F)\backslash G(F),\eta)$, we have
$$\lim_{N\rightarrow \infty} I_{N,\eta}(f)=I_{\eta}(f).$$
\end{thm}

\begin{rmk}
As in the Ginzburg-Rallis model case, the relative trace formula above will give us a multiplicity formula for $m(\pi)$ when $\pi$ is supercuspidal. We will skip the details here.
\end{rmk}

\subsection{The generalized trilinear $\GL_2$ models}
In this section, we consider the generalized trilinear $\GL_2$ models. These models was first considered by Prasad in \cite{P92} for general generic representations using different methods. By using our method in this paper, we can prove the supercuspidal case.

\textbf{Case I:}
Let $K/F$ be a cubic field extension, $G(F)=\GL_2(K)$, and $H(F)=\GL_2(F)$. On the mean time, let $G_D(F)=\GL_1(D_K)$ and $H_D(F)=\GL_1(D)$ where $D_K=D\otimes_F K$. For a given irreducible representation $\pi$ of $G(F)$, assume that the restriction of the central character $\omega_{\pi}:K^{\times}\rightarrow \BC^{\times}$ to $F^{\times}$ equals $\chi^2$ for some character $\chi$ of $F^{\times}$. $\chi$ will induce a one-dimensional representation $\sigma$ of $H(F)$. Let
\begin{equation}
m(\pi)=\dim \Hom_{H(F)} (\pi,\sigma).
\end{equation}
Similarly we can define $m(\pi_D)$ for an irreducible representation $\pi_D$ of $G_D(F)$. The following theorem has been proved by Prasad in \cite{P92} for general generic representation using different method. By using our method in this paper, we can prove the supercuspidal case.
\begin{thm}
If $\pi$ is a supercuspidal representation of $G$, let $\pi_D$ be the Jacquet-Langlands correspondence of $\pi$ to $G_D$. (Since $\pi$ is supercuspidal, $\pi_D$ always exist.) Then
$$m(\pi)+m(\pi_D)=1.$$
\end{thm}

We can also prove the relative trace formula for this model and the multiplicity formulas for $m(\pi)$ and $m(\pi_D)$. Since the formulas will be similar to the trilinear $\GL_2$ model case in previous section, we will skip the details here.

\textbf{Case II:} Let $E=F_v$ be a quadratic extension of $F$ where $v$ is a non-trivial square class in $F^{\times}$. Let $G(F)=\GL_2(E)\oplus \GL_2(F),\; H(F)=\GL_2(F),\; G_D(F)=\GL_2(E)\times \GL_1(D)$ and $H_D(F)=\GL_1(D)$. As in the previous cases, we can define the multiplicity $m(\pi)$ (resp. $m(\pi_D)$) for the model $(G(F),H(F))$ (resp. $(G_D(F),H_D(F))$). By using our method in this paper, we can still prove that the summation of the multiplicities over any supercuspidal L-packet is 1. We can also prove the relative trace formula and the multiplicity formula. However, there is one difference between this case and all previous cases, this will be discussed in the following remark.

\begin{rmk}
In previous cases, in the geometric side of the trace formula (or the multiplicity formula), we are integrating the germs of the distribution over all nonsplit tori of $H(F)$. But in this case, we only need to integrate over those nonsplit tori which is not isomorphic to $T_v$. The reason is that in this case, both $G(F)$ and $G_D(F)$ contain $\GL_2(E)$. As a result, for an element in $T_v(F)\cap H(F)_{reg}$ (or $T_v(F)\cap H_D(F)_{reg}$), although it is elliptic in $H(F)$ and $H_D(F)$, it will no longer be elliptic in $G(F)$ or $G_D(F)$. Therefore the localization at this element will be zero. This is why the torus $T_v$ will not show up in the multiplicity formula and the geometric side of the relative trace formula.
\end{rmk}

\subsection{The middle model}
Choose $Q$ be the parabolic subgroup of $\GL_6(F)$ (resp. $\GL_3(D)$) of $(4,2)$ type (resp. $(2,1)$ type) which contains the lower minimal parabolic subgroup. Then the reduced model $(L,R_Q)$ we get is the following (Once again to make our notation simple, we will use $(G,\;H,\; U)$ instead of $(L,H_Q)$):  Let $G=GL_4(F)\times GL_2(F)$ and $P=MU$ be the parabolic subgroup of $G(F)$ with the Levi part $M$ isomorphic to $GL_2(F)\times GL_2(F)\times GL_2(F)$ (i.e. $P$ is the product of the second $GL_2(F)$ and the parabolic subgroup $P_{2,2}$ of the first $GL_4(F)$). The unipotent radical $U$ consists of elements of the form
\begin{equation}
u=u(X):=\begin{pmatrix} 1 & X & 0 \\ 0 & 1 & 0 \\ 0 & 0 & 1 \end{pmatrix},\; X\in M_2(F).
\end{equation}
The character $\xi$ on $U$ is defined to be $\xi(u(X))=\psi(\tr(X))$. Let $H=\GL_2(F)$ diagonally embeded into $M$. For a given irreducible representation $\pi$ of $G$, assume $\omega_{\pi}=\chi^2$ for some character $\chi$ of $F^{\times}$. $\chi$ will induce a one-dimensional representation $\sigma$ of $H$. Combining $\xi$ and $\sigma$, we have a one-dimensional representation $\sigma\otimes \xi$ of $R:=H\ltimes U$. Let
\begin{equation}
m(\pi)=\dim\Hom_{R(F)} (\pi,\sigma\otimes \xi).
\end{equation}
This model can be thought as the "middle model" between the Ginzburg-Rallis model and the trilinear model of $\GL_2$.

Similarly, for the quaternion algebra case, we can define the multiplicity $m(\pi_D)$. The following theorem is an analogy of Theorem \ref{main} for this model, which can be proved by our method in this paper.

\begin{thm}
If $\pi$ is a supercuspidal representation of $G$, let $\pi_D$ be the Jacquet-Langlands correspondence of $\pi$ to $G_D$. (Since $\pi$ is supercuspidal, $\pi_D$ always exist.) Then
$$m(\pi)+m(\pi_D)=1.$$
\end{thm}

Consider the pair $(G,H)$ that is either $(G,H)$ or $(G_D,H_D)$ as defined above, and take $\CT$ to be the subset of subtorus $T$ in $H$ defined in Section 5.1. We can also prove the local relative trace formula for this model, this will be used in the proof of the spectral side of the local relative trace formula for the Ginzburg-Rallis model in the forthcoming paper \cite{Wan16}.

Let $\theta$ be a quasi-character on $Z_G(F)\backslash G(F)$ with central character $\eta=\chi^2$, and $T\in \CT$. If $T=\{1\}$, then we are in the split case. Since there is a unique regular nilpotent orbit in $\Fg(F)$, we let $c_\theta(t)=c_{\theta,\CO_{reg}}(t)$. If $T=T_v$ for some $v\in F^{\times}/(F^{\times})^2, v\neq 1$, and $t\in T_v$ is a regular element, $G_t=GL_2(F_v)\times GL_1(F_v)$. Let $\CO=\CO_1\times \CO_2$ where $\CO_1$ is the unique regular nilpotent orbit in $\Fg \Fl_2(F_v)$ and $\CO_2=\{0\}$ is the unique nilpotent orbit in $\Fg \Fl_1(F_v)$, define $c_\theta(t)=c_{\theta,\CO}(t)$. Note that $\Fg \Fl_1(F_v)$ is abelian, the only nilpotent element is zero, the germ expansion is just evaluation.

Let $f\in C_{c}^{\infty} (Z_G(F)\backslash G(F),\eta)$ be a strongly cuspidal function. For each $T\in \CT$, let $c_f$ be the function $c_{\theta_f}$ defined above. We define a function $\Delta$ on $H_{ss}(F)$ by
$$\Delta(x)=\mid \det((1-ad(x)^{-1})_{\mid U(F)/U_x(F)}) \mid_F.$$
Similarly, define $\Delta$ on $\Fh_{ss}(F)$ by
$$\Delta(X)=\mid \det((1-ad(X)^{-1})_{\mid U(F)/U_x(F)}) \mid_F.$$
Let
$$I_{\eta}(f)=\sum_{T\in \CT} \mid W(H,T)\mid^{-1} \nu(T) \int_{Z_G(F)\backslash T(F)} c_f(t) D^H(t) \Delta(t) \chi(\det(t))^{-1} dt.$$
By a similar argument as Proposition \ref{integrable 1}, we know the integral is absolutely convergent.

Now for $g\in G(F)$, we define the function ${}^g f^{\xi}$ on $H(F)$ to be
$${}^g f^{\xi}(x)=\int_{U(F)} f(g^{-1}xug)\xi(u) du.$$
This is a function belonging to $C_{c}^{\infty}(Z_H(F)\backslash H(F),\eta)$. Let
$$I_{\eta}(f,g)=\int_{Z_H(F)\backslash H(F)} f(g^{-1}xg) \chi(det(x))^{-1}dx.$$
And for each $N\in \BN$, we can still define the truncated function $\kappa_N$ in a similar way. Let
$$I_{N,\eta}(f)=\int_{U(F)H(F)\backslash G(F)} I(f,g) \kappa_N(g) dg.$$
The following theorem is an analogy of Theorem \ref{main 3} for this model, it is the geometric side of the local relative trace formula for the middle model.
\begin{thm}
For every strongly cuspidal function $f$ belonging to the space $C_{c}^{\infty} (Z_G(F)\backslash G(F),\eta)$, we have
\begin{equation}
\lim_{N\rightarrow \infty} I_{N,\eta}(f)=I_{\eta}(f).
\end{equation}
\end{thm}

\begin{rmk}
As in the Ginzburg-Rallis model case, the relative trace formula above will give us a multiplicity formula for $m(\pi)$ when $\pi$ is supercuspidal. We will skip the details here.
\end{rmk}

\begin{rmk}
If $Q$ is of type $(2,4)$, the reduced model will still be the middle model as in the $(4,2)$ case. The same results in this section will still hold.
\end{rmk}

\subsection{The type II models}
Now we focus on the case $G=\GL_6(F)$. Let $Q=LU_Q$ be a good parabolic subgroup of $G$ which is not of type $(2,2,2)$, $(4,2)$ or $(2,4)$. Then we get the reduced model $(L,R_Q)$, and the character $\sigma_Q\times \xi_Q$ on $R_Q$ is just the restriction of the character $\sigma\times \xi$. For an irreducible admissible representation $\pi$ of $L(F)$, we can still define the Hom space and the multiplicity $m(\pi)$. Since there are too many parabolic subgroup of this type, we will not write down all such models. Instead, we will only write down the reduced model for maximal parabolic subgroups (i.e. type $(5,1)$, type $(3,3)$ and type $(1,5)$). All other models can be viewed as the reduced model of the maximal ones. \textbf{The most important feature of such models is that in this cases, all semisimple elements in $R_Q$ are split. As a result, when we study the localization of the local relative trace formula for such models, it will always be zero unless we are localizing at the center. Therefore for such models, the geometric side of the local relative trace formula and the geometric multiplicity formula only contain the germ at 1.} We first write down the group $R_Q$ when $Q$ is a maximal parabolic subgroup.
\\
\\
\textbf{Type (5,1):} Let $Q=LU_Q$ be the parabolic subgroup of $\GL_6(F)$ of type $(5,1)$ which contains the lower Borel subgroup. Then $L=\GL_5(F)\times \GL_1(F)$ and $R_Q=H_QU_Q\subset L$ is of the following form:
$$H_Q=\{h_Q(a,b,x)=diag(\begin{pmatrix} a&0\\ x&b \end{pmatrix}, \begin{pmatrix} a&0\\ x&b \end{pmatrix}, \begin{pmatrix} a \end{pmatrix})\times \begin{pmatrix} b \end{pmatrix}| a,b\in F^{\times}, \; x\in F\}$$
and
$$U_Q=\{u_Q(X,Y_1,Y_2)=\begin{pmatrix} I_2& X& Y_1 \\ 0&I_2&Y_2\\0&0&1  \end{pmatrix} \times \begin{pmatrix} 1 \end{pmatrix}| X\in M_{2\times 2}(F),\; Y_1,Y_2\in M_{1\times 2}(F) \}.$$
Let $Y_i=\begin{pmatrix} y_{i1}\\ y_{i2} \end{pmatrix}$ for $i=1,2$. Then the restriction of $\sigma\otimes \xi$ to $R_{\bar{Q}}$ is
$$\sigma_Q\times \xi_Q:h_Q(a,b,x)u_Q(X,Y_1,Y_2)\rightarrow \chi(ab)\psi(\tr(X)+y_{21}).$$
\\
\\
\textbf{Type (3,3):} Let $Q=LU_Q$ be the parabolic subgroup of $\GL_6(F)$ of type $(3,3)$ which contains the lower Borel subgroup. Then $L=\GL_3(F)\times \GL_3(F)$ and $R_Q=H_QU_Q\subset L$ is of the following form:
$$H_Q=\{h_Q(a,b,x)=\begin{pmatrix}a&0&0\\x&b&0\\0&0&a\end{pmatrix} \times \begin{pmatrix}b&0&0\\0&a&0\\0&x&b\end{pmatrix} | a,b\in F^{\times},\; x\in F\}$$
and
$$U_Q=\{u_Q(x_1,x_2,y_1,y_2)= \begin{pmatrix}1&0&x_1\\0&1&x_2\\0&0&1\end{pmatrix} \times \begin{pmatrix}1&y_1&y_2\\0&1&0\\0&c&1\end{pmatrix}| x_1,x_2,y_1,y_2\in F\}.$$
Then the restriction of $\sigma\otimes \xi$ to $R_{\bar{Q}}$ is
$$\sigma_Q\times \xi_Q: h_Q(a,b,x)u_Q(x_1,x_2,y_1,y_2)\mapsto \chi(ab)\psi(x_1+y_2).$$
\\
\\
\textbf{Type (1,5):} This is similar to the type (5,1) case above, we will skip it here.
\\
\\
By the description above, it is easy to see that all semisimple elements in $R_Q$ are split. Now we are ready to state the multiplicity formula and the trace formula. Let $f\in C_{c}^{\infty} (Z_L(F)\backslash L(F),\eta_L)$ be a strongly cuspidal function where $\eta_L$ is a character on $Z_L(F)$ whose restriction on $Z_G(F)$ equals to $\eta=\chi^2$. For $g\in L(F)$, we define the function ${}^g f^{\xi}$ on $H_Q(F)$ to be
$${}^g f^{\xi}(x)=\int_{U_Q(F)} f(g^{-1}xug)\xi_Q(u) du.$$
This is a function belonging to $C_{c}^{\infty}(Z_H(F)\backslash H_Q(F),\eta)$. Let
$$I_{\eta}(f,g)=\int_{Z_H(F)\backslash H_Q(F)} f(g^{-1}xg) \chi(det(x))^{-1}dx.$$
And for each $N\in \BN$, we can still define the truncated function $\kappa_N$ in a similar way. Let
$$I_{N,\eta}(f)=\int_{U(F)H(F)\backslash G(F)} I(f,g) \kappa_N(g) dg.$$
Now we are able to state the geometric multiplicity formula and the local relative trace formula for these models, which are analogies of Theorem \ref{main} and Theorem \ref{main 3}.

\begin{thm}\label{type II}
\begin{enumerate}
\item If $\pi$ is supercuspidal representation of $L$, we have
$$m(\pi)=c_{\theta_{\pi},\CO_{reg}}(1)=1.$$
\item For all strongly cuspidal function $f\in C_{c}^{\infty} (Z_L(F)\backslash L(F),\eta_L)$, we have
$$\lim_{N\rightarrow \infty} I_{N,\eta}(f)=c_{\theta_f,\CO_{reg}}(1).$$
\end{enumerate}
\end{thm}


\end{document}